\documentclass[reviewfalse,onefignum,onetabnum]{siamart190516}

\usepackage{lipsum}
\usepackage{amsfonts}
\usepackage{graphicx}
\usepackage{epstopdf}
\usepackage{algorithmic}
\usepackage{cite}
\usepackage{minimal-defs}
\usepackage[shortlabels]{enumitem}
\makeatletter
\newcommand{\citecomment}[2][]{\citen{#2}#1\citevar}
\newcommand{\citeone}[1]{\citecomment{#1}}
\newcommand{\citetwo}[2][]{\citecomment[,~#1]{#2}}
\newcommand{\citevar}{\@ifnextchar\bgroup{;~\citeone}{\@ifnextchar[{;~\citetwo}{]}}}
\newcommand{\citefirst}{\@ifnextchar\bgroup{\citeone}{\@ifnextchar[{\citetwo}{]}}}
\newcommand{\cites}{[\citefirst}
\makeatother

\usepackage{pgfplots}
\usepackage{tikz}
\usetikzlibrary{shapes,arrows}
\usetikzlibrary{arrows.meta}
\usetikzlibrary{calc}
\usetikzlibrary{automata}
\pgfplotsset{compat=1.17}

\usepackage{colonequals}
\newcommand{\defeq}{\colonequals}

\newcommand{\bx}{\mathbf{x}}
\newcommand{\bu}{\mathbf{u}}
\newcommand{\by}{\mathbf{y}}
\newcommand{\bH}{\mathbf{H}}

\newcommand{\bPhi}{\boldsymbol{\Phi}}
\newcommand{\R}{\mathbb{R}}
\newcommand{\C}{\mathcal C}
\newcommand{\prox}{\textnormal{prox}}

\newcommand{\bmat}[1]{\begin{bmatrix}#1\end{bmatrix}}
\newcommand{\sbmat}[1]{\left[\begin{smallmatrix}#1\end{smallmatrix}\right]}

\usepackage[mathscr]{euscript}
\DeclareMathAlphabet{\mathpzc}{OT1}{pzc}{m}{it}

\usepackage{xcolor}

\definecolor{purple}{rgb}{0.3,0.0,.4}
\newcommand{\mnote}[1]{}
\newcommand{\snote}[1]{}
\newcommand{\LL}[1]{}



\usepackage{geometry}
\usepackage{amsmath}
\usepackage{mathtools}
\usepackage{float}
\usepackage{mathrsfs}
\usepackage{bbm}
\usepackage{arydshln}

\ifpdf
  \DeclareGraphicsExtensions{.eps,.pdf,.png,.jpg}
\else
  \DeclareGraphicsExtensions{.eps}
\fi

\newcommand{\titleparagraph}[1]{\paragraph{#1}}

\tikzset{
	box/.style={draw,
		minimum height=1.8em,
		text width=1.5em,
		align=center, font=\small}
}

\tikzset{
	box2/.style={draw,
		minimum height=1.8em,
		text width= 8em, rounded corners,
		align=center, font=\small}
}

\tikzset{
	box3/.style={draw,
		minimum height=1.8em,
		text width= 2.6em, rounded corners,
		align=center, font=\small}
}

\tikzset{
	box4/.style={draw,
		minimum height=1.8em,
		text width= 5.2em, rounded corners,
		align=center, font=\small}
}

\tikzset{
	box5/.style={draw,
		minimum height=4.2em,
		text width= 1.5em,
		align=center, font=\small}
}

\tikzset{
	box6/.style={draw,
		minimum height=2.8em,
		text width= 1.5em,
		align=center, font=\small}
}

\tikzset{
	box7/.style={draw,
		circle, minimum size = 1em,
		align=center, font = \large}
}

\tikzset{
	box8/.style = {draw,dashed,
		minimum height = 6.4em,
		text width = 3.4em}
}

\tikzset{
	box9/.style={draw,
		minimum height=2em,
		text width=2.6em,
		align=center, font=\small}
}

\tikzset{
	box10/.style={draw,
		minimum height=2em,
		text width= 2.6em, rounded corners,
		align=center, font=\small}
}

\tikzset{
	box11/.style={draw,
		minimum height=1.8em,
		text width= 5.2em,
		align=center, font=\small}
}

\usepackage[breakable]{tcolorbox}
\usepackage{upquote} 
\usepackage{fancyvrb} 
\tcbset{nobeforeafter} 
\usepackage{upquote} 
\usepackage{hyperref}

\DefineVerbatimEnvironment{Highlighting}{Verbatim}{commandchars=\\\{\}}

    \makeatletter
\newbox\Wrappedcontinuationbox
\newbox\Wrappedvisiblespacebox
\newcommand*\Wrappedvisiblespace {\textcolor{red}{\textvisiblespace}}
\newcommand*\Wrappedcontinuationsymbol {\textcolor{red}{\llap{\tiny$\m@th\hookrightarrow$}}}
\newcommand*\Wrappedcontinuationindent {3ex }
\newcommand*\Wrappedafterbreak {\kern\Wrappedcontinuationindent\copy\Wrappedcontinuationbox}
\newcommand*\Wrappedbreaksatspecials {%
	\def\PYGZus{\discretionary{\char`\_}{\Wrappedafterbreak}{\char`\_}}%
	\def\PYGZob{\discretionary{}{\Wrappedafterbreak\char`\{}{\char`\{}}%
	\def\PYGZcb{\discretionary{\char`\}}{\Wrappedafterbreak}{\char`\}}}%
	\def\PYGZca{\discretionary{\char`\^}{\Wrappedafterbreak}{\char`\^}}%
	\def\PYGZam{\discretionary{\char`\&}{\Wrappedafterbreak}{\char`\&}}%
	\def\PYGZlt{\discretionary{}{\Wrappedafterbreak\char`\<}{\char`\<}}%
	\def\PYGZgt{\discretionary{\char`\>}{\Wrappedafterbreak}{\char`\>}}%
	\def\PYGZsh{\discretionary{}{\Wrappedafterbreak\char`\#}{\char`\#}}%
	\def\PYGZpc{\discretionary{}{\Wrappedafterbreak\char`\%}{\char`\%}}%
	\def\PYGZdl{\discretionary{}{\Wrappedafterbreak\char`\$}{\char`\$}}%
	\def\PYGZhy{\discretionary{\char`\-}{\Wrappedafterbreak}{\char`\-}}%
	\def\PYGZsq{\discretionary{}{\Wrappedafterbreak\textquotesingle}{\textquotesingle}}%
	\def\PYGZdq{\discretionary{}{\Wrappedafterbreak\char`\"}{\char`\"}}%
	\def\PYGZti{\discretionary{\char`\~}{\Wrappedafterbreak}{\char`\~}}%
}
\newcommand*\Wrappedbreaksatpunct {%
	\lccode`\~`\.\lowercase{\def~}{\discretionary{\hbox{\char`\.}}{\Wrappedafterbreak}{\hbox{\char`\.}}}%
	\lccode`\~`\,\lowercase{\def~}{\discretionary{\hbox{\char`\,}}{\Wrappedafterbreak}{\hbox{\char`\,}}}%
	\lccode`\~`\;\lowercase{\def~}{\discretionary{\hbox{\char`\;}}{\Wrappedafterbreak}{\hbox{\char`\;}}}%
	\lccode`\~`\:\lowercase{\def~}{\discretionary{\hbox{\char`\:}}{\Wrappedafterbreak}{\hbox{\char`\:}}}%
	\lccode`\~`\?\lowercase{\def~}{\discretionary{\hbox{\char`\?}}{\Wrappedafterbreak}{\hbox{\char`\?}}}%
	\lccode`\~`\!\lowercase{\def~}{\discretionary{\hbox{\char`\!}}{\Wrappedafterbreak}{\hbox{\char`\!}}}%
	\lccode`\~`\/\lowercase{\def~}{\discretionary{\hbox{\char`\/}}{\Wrappedafterbreak}{\hbox{\char`\/}}}%
	\catcode`\.\active
	\catcode`\,\active
	\catcode`\;\active
	\catcode`\:\active
	\catcode`\?\active
	\catcode`\!\active
	\catcode`\/\active
	\lccode`\~`\~
}
\makeatother

\let\OriginalVerbatim=\Verbatim
\makeatletter
\renewcommand{\Verbatim}[1][1]{%
	\sbox\Wrappedcontinuationbox {\Wrappedcontinuationsymbol}%
	\sbox\Wrappedvisiblespacebox {\FV@SetupFont\Wrappedvisiblespace}%
	\def\FancyVerbFormatLine ##1{\hsize\linewidth
		\vtop{\raggedright\hyphenpenalty\z@\exhyphenpenalty\z@
			\doublehyphendemerits\z@\finalhyphendemerits\z@
			\strut ##1\strut}%
	}%
\def\FV@Space {%
	\nobreak\hskip\z@ plus\fontdimen3\font minus\fontdimen4\font
	\discretionary{\copy\Wrappedvisiblespacebox}{\Wrappedafterbreak}
	{\kern\fontdimen2\font}%
}%

\Wrappedbreaksatspecials
\OriginalVerbatim[#1,codes*=\Wrappedbreaksatpunct]%
}
\makeatother
\definecolor{cellbackground}{HTML}{F7F7F7}
\definecolor{cellborder}{HTML}{CFCFCF}

\makeatletter
\def\PY@reset{\let\PY@it=\relax \let\PY@bf=\relax%
	\let\PY@ul=\relax \let\PY@tc=\relax%
	\let\PY@bc=\relax \let\PY@ff=\relax}
\def\PY@tok#1{\csname PY@tok@#1\endcsname}
\def\PY@toks#1+{\ifx\relax#1\empty\else%
	\PY@tok{#1}\expandafter\PY@toks\fi}
\def\PY@do#1{\PY@bc{\PY@tc{\PY@ul{%
				\PY@it{\PY@bf{\PY@ff{#1}}}}}}}
\def\PY#1#2{\PY@reset\PY@toks#1+\relax+\PY@do{#2}}

\expandafter\def\csname PY@tok@w\endcsname{\def\PY@tc##1{\textcolor[rgb]{0.73,0.73,0.73}{##1}}}
\expandafter\def\csname PY@tok@c\endcsname{\let\PY@it=\textit\def\PY@tc##1{\textcolor[rgb]{0.25,0.50,0.50}{##1}}}
\expandafter\def\csname PY@tok@cp\endcsname{\def\PY@tc##1{\textcolor[rgb]{0.74,0.48,0.00}{##1}}}
\expandafter\def\csname PY@tok@k\endcsname{\let\PY@bf=\textbf\def\PY@tc##1{\textcolor[rgb]{0.00,0.50,0.00}{##1}}}
\expandafter\def\csname PY@tok@kp\endcsname{\def\PY@tc##1{\textcolor[rgb]{0.00,0.50,0.00}{##1}}}
\expandafter\def\csname PY@tok@kt\endcsname{\def\PY@tc##1{\textcolor[rgb]{0.69,0.00,0.25}{##1}}}
\expandafter\def\csname PY@tok@o\endcsname{\def\PY@tc##1{\textcolor[rgb]{0.40,0.40,0.40}{##1}}}
\expandafter\def\csname PY@tok@ow\endcsname{\let\PY@bf=\textbf\def\PY@tc##1{\textcolor[rgb]{0.67,0.13,1.00}{##1}}}
\expandafter\def\csname PY@tok@nb\endcsname{\def\PY@tc##1{\textcolor[rgb]{0.00,0.50,0.00}{##1}}}
\expandafter\def\csname PY@tok@nf\endcsname{\def\PY@tc##1{\textcolor[rgb]{0.00,0.00,1.00}{##1}}}
\expandafter\def\csname PY@tok@nc\endcsname{\let\PY@bf=\textbf\def\PY@tc##1{\textcolor[rgb]{0.00,0.00,1.00}{##1}}}
\expandafter\def\csname PY@tok@nn\endcsname{\let\PY@bf=\textbf\def\PY@tc##1{\textcolor[rgb]{0.00,0.00,1.00}{##1}}}
\expandafter\def\csname PY@tok@ne\endcsname{\let\PY@bf=\textbf\def\PY@tc##1{\textcolor[rgb]{0.82,0.25,0.23}{##1}}}
\expandafter\def\csname PY@tok@nv\endcsname{\def\PY@tc##1{\textcolor[rgb]{0.10,0.09,0.49}{##1}}}
\expandafter\def\csname PY@tok@no\endcsname{\def\PY@tc##1{\textcolor[rgb]{0.53,0.00,0.00}{##1}}}
\expandafter\def\csname PY@tok@nl\endcsname{\def\PY@tc##1{\textcolor[rgb]{0.63,0.63,0.00}{##1}}}
\expandafter\def\csname PY@tok@ni\endcsname{\let\PY@bf=\textbf\def\PY@tc##1{\textcolor[rgb]{0.60,0.60,0.60}{##1}}}
\expandafter\def\csname PY@tok@na\endcsname{\def\PY@tc##1{\textcolor[rgb]{0.49,0.56,0.16}{##1}}}
\expandafter\def\csname PY@tok@nt\endcsname{\let\PY@bf=\textbf\def\PY@tc##1{\textcolor[rgb]{0.00,0.50,0.00}{##1}}}
\expandafter\def\csname PY@tok@nd\endcsname{\def\PY@tc##1{\textcolor[rgb]{0.67,0.13,1.00}{##1}}}
\expandafter\def\csname PY@tok@s\endcsname{\def\PY@tc##1{\textcolor[rgb]{0.73,0.13,0.13}{##1}}}
\expandafter\def\csname PY@tok@sd\endcsname{\let\PY@it=\textit\def\PY@tc##1{\textcolor[rgb]{0.73,0.13,0.13}{##1}}}
\expandafter\def\csname PY@tok@si\endcsname{\let\PY@bf=\textbf\def\PY@tc##1{\textcolor[rgb]{0.73,0.40,0.53}{##1}}}
\expandafter\def\csname PY@tok@se\endcsname{\let\PY@bf=\textbf\def\PY@tc##1{\textcolor[rgb]{0.73,0.40,0.13}{##1}}}
\expandafter\def\csname PY@tok@sr\endcsname{\def\PY@tc##1{\textcolor[rgb]{0.73,0.40,0.53}{##1}}}
\expandafter\def\csname PY@tok@ss\endcsname{\def\PY@tc##1{\textcolor[rgb]{0.10,0.09,0.49}{##1}}}
\expandafter\def\csname PY@tok@sx\endcsname{\def\PY@tc##1{\textcolor[rgb]{0.00,0.50,0.00}{##1}}}
\expandafter\def\csname PY@tok@m\endcsname{\def\PY@tc##1{\textcolor[rgb]{0.40,0.40,0.40}{##1}}}
\expandafter\def\csname PY@tok@gh\endcsname{\let\PY@bf=\textbf\def\PY@tc##1{\textcolor[rgb]{0.00,0.00,0.50}{##1}}}
\expandafter\def\csname PY@tok@gu\endcsname{\let\PY@bf=\textbf\def\PY@tc##1{\textcolor[rgb]{0.50,0.00,0.50}{##1}}}
\expandafter\def\csname PY@tok@gd\endcsname{\def\PY@tc##1{\textcolor[rgb]{0.63,0.00,0.00}{##1}}}
\expandafter\def\csname PY@tok@gi\endcsname{\def\PY@tc##1{\textcolor[rgb]{0.00,0.63,0.00}{##1}}}
\expandafter\def\csname PY@tok@gr\endcsname{\def\PY@tc##1{\textcolor[rgb]{1.00,0.00,0.00}{##1}}}
\expandafter\def\csname PY@tok@ge\endcsname{\let\PY@it=\textit}
\expandafter\def\csname PY@tok@gs\endcsname{\let\PY@bf=\textbf}
\expandafter\def\csname PY@tok@gp\endcsname{\let\PY@bf=\textbf\def\PY@tc##1{\textcolor[rgb]{0.00,0.00,0.50}{##1}}}
\expandafter\def\csname PY@tok@go\endcsname{\def\PY@tc##1{\textcolor[rgb]{0.53,0.53,0.53}{##1}}}
\expandafter\def\csname PY@tok@gt\endcsname{\def\PY@tc##1{\textcolor[rgb]{0.00,0.27,0.87}{##1}}}
\expandafter\def\csname PY@tok@err\endcsname{\def\PY@bc##1{\setlength{\fboxsep}{0pt}\fcolorbox[rgb]{1.00,0.00,0.00}{1,1,1}{\strut ##1}}}
\expandafter\def\csname PY@tok@kc\endcsname{\let\PY@bf=\textbf\def\PY@tc##1{\textcolor[rgb]{0.00,0.50,0.00}{##1}}}
\expandafter\def\csname PY@tok@kd\endcsname{\let\PY@bf=\textbf\def\PY@tc##1{\textcolor[rgb]{0.00,0.50,0.00}{##1}}}
\expandafter\def\csname PY@tok@kn\endcsname{\let\PY@bf=\textbf\def\PY@tc##1{\textcolor[rgb]{0.00,0.50,0.00}{##1}}}
\expandafter\def\csname PY@tok@kr\endcsname{\let\PY@bf=\textbf\def\PY@tc##1{\textcolor[rgb]{0.00,0.50,0.00}{##1}}}
\expandafter\def\csname PY@tok@bp\endcsname{\def\PY@tc##1{\textcolor[rgb]{0.00,0.50,0.00}{##1}}}
\expandafter\def\csname PY@tok@fm\endcsname{\def\PY@tc##1{\textcolor[rgb]{0.00,0.00,1.00}{##1}}}
\expandafter\def\csname PY@tok@vc\endcsname{\def\PY@tc##1{\textcolor[rgb]{0.10,0.09,0.49}{##1}}}
\expandafter\def\csname PY@tok@vg\endcsname{\def\PY@tc##1{\textcolor[rgb]{0.10,0.09,0.49}{##1}}}
\expandafter\def\csname PY@tok@vi\endcsname{\def\PY@tc##1{\textcolor[rgb]{0.10,0.09,0.49}{##1}}}
\expandafter\def\csname PY@tok@vm\endcsname{\def\PY@tc##1{\textcolor[rgb]{0.10,0.09,0.49}{##1}}}
\expandafter\def\csname PY@tok@sa\endcsname{\def\PY@tc##1{\textcolor[rgb]{0.73,0.13,0.13}{##1}}}
\expandafter\def\csname PY@tok@sb\endcsname{\def\PY@tc##1{\textcolor[rgb]{0.73,0.13,0.13}{##1}}}
\expandafter\def\csname PY@tok@sc\endcsname{\def\PY@tc##1{\textcolor[rgb]{0.73,0.13,0.13}{##1}}}
\expandafter\def\csname PY@tok@dl\endcsname{\def\PY@tc##1{\textcolor[rgb]{0.73,0.13,0.13}{##1}}}
\expandafter\def\csname PY@tok@s2\endcsname{\def\PY@tc##1{\textcolor[rgb]{0.73,0.13,0.13}{##1}}}
\expandafter\def\csname PY@tok@sh\endcsname{\def\PY@tc##1{\textcolor[rgb]{0.73,0.13,0.13}{##1}}}
\expandafter\def\csname PY@tok@s1\endcsname{\def\PY@tc##1{\textcolor[rgb]{0.73,0.13,0.13}{##1}}}
\expandafter\def\csname PY@tok@mb\endcsname{\def\PY@tc##1{\textcolor[rgb]{0.40,0.40,0.40}{##1}}}
\expandafter\def\csname PY@tok@mf\endcsname{\def\PY@tc##1{\textcolor[rgb]{0.40,0.40,0.40}{##1}}}
\expandafter\def\csname PY@tok@mh\endcsname{\def\PY@tc##1{\textcolor[rgb]{0.40,0.40,0.40}{##1}}}
\expandafter\def\csname PY@tok@mi\endcsname{\def\PY@tc##1{\textcolor[rgb]{0.40,0.40,0.40}{##1}}}
\expandafter\def\csname PY@tok@il\endcsname{\def\PY@tc##1{\textcolor[rgb]{0.40,0.40,0.40}{##1}}}
\expandafter\def\csname PY@tok@mo\endcsname{\def\PY@tc##1{\textcolor[rgb]{0.40,0.40,0.40}{##1}}}
\expandafter\def\csname PY@tok@ch\endcsname{\let\PY@it=\textit\def\PY@tc##1{\textcolor[rgb]{0.25,0.50,0.50}{##1}}}
\expandafter\def\csname PY@tok@cm\endcsname{\let\PY@it=\textit\def\PY@tc##1{\textcolor[rgb]{0.25,0.50,0.50}{##1}}}
\expandafter\def\csname PY@tok@cpf\endcsname{\let\PY@it=\textit\def\PY@tc##1{\textcolor[rgb]{0.25,0.50,0.50}{##1}}}
\expandafter\def\csname PY@tok@c1\endcsname{\let\PY@it=\textit\def\PY@tc##1{\textcolor[rgb]{0.25,0.50,0.50}{##1}}}
\expandafter\def\csname PY@tok@cs\endcsname{\let\PY@it=\textit\def\PY@tc##1{\textcolor[rgb]{0.25,0.50,0.50}{##1}}}

\def\PYZus{\char`\_}

\def\PYZca{\char`\^}

\def\PYZlt{\char`\<}

\def\PYZsh{\char`\#}

\def\PYZhy{\char`\-}

\def\PYZdq{\char`\"}

\makeatother

\newsiamremark{remark}{Remark}
\newsiamremark{hypothesis}{Hypothesis}
\crefname{hypothesis}{Hypothesis}{Hypotheses}
\newsiamthm{claim}{Claim}

\title{An automatic system to detect equivalence between iterative algorithms 
}

\author{Shipu Zhao\thanks{Cornell University (\email{sz533@cornell.edu}, \email{udell@cornell.edu}).}
	\and Laurent Lessard\thanks{Northeastern University (\email{l.lessard@northeastern.edu}).} 
	\and Madeleine Udell\footnotemark[1]}

\usepackage{amsopn}

\RequirePackage[normalem]{ulem} 
\RequirePackage{color}\definecolor{RED}{rgb}{1,0,0}\definecolor{BLUE}{rgb}{0,0,1} 

\ifpdf
\hypersetup{
  pdftitle={Main},
  pdfauthor={}
}
\fi

\begin{document}

\maketitle

\begin{abstract}
	When are two algorithms the same?
	How can we be sure a recently proposed algorithm is novel,
	and not a minor twist on an existing method?
	In this paper, we present a framework for reasoning about equivalence
	between a broad class of iterative algorithms,
	with a focus on algorithms designed for convex optimization.  
	We propose several notions of what it means for two algorithms to be equivalent,
	and provide computationally tractable means to detect equivalence.
	Our main definition, oracle equivalence, states that two algorithms
	are equivalent if they result in the same sequence of calls to the function oracles
	(for suitable initialization).
	Borrowing from control theory,
	we use state-space realizations to represent algorithms
	and characterize algorithm equivalence via transfer functions.
	Our framework can also identify and characterize some algorithm transformations
	including permutations of the update equations,
	repetition of the iteration,
	and conjugation of some of the function oracles in the algorithm.
	To support the paper, we have developed a software package named \lin{}
	that implements the framework to identify other
	iterative algorithms that are equivalent to an input algorithm.
	More broadly, this framework and software advances the goal of making mathematics searchable. 
\end{abstract}


\section{Introduction}\label{intro}

Large-scale optimization problems in machine learning, signal processing, and imaging have fueled ongoing interest in iterative optimization algorithms.
New optimization algorithms are regularly proposed in order to
capture more complicated models, reduce computational burdens,
or obtain stronger performance and convergence guarantees.

However, the \emph{novelty} of an algorithm can be difficult to establish
because algorithms can be written in different equivalent forms.
For example, \cref{algo_i1} was originally proposed by Popov~\cite{popov1980modification}
in the context of solving saddle point problems.
This method was later generalized by Chiang et al.~\cite[\S4.1]{chiang2012online}
in the context of online optimization.
\Cref{algo_i2} is a reformulation of \cref{algo_i1}
adapted for use in generative adversarial networks (GANs)~\cite{gidel2018a}.
\Cref{algo_i3} is an adaptation of \emph{Optimistic Mirror Descent}~\cite{OMD_rakhlin}
used by Daskalakis et al.~\cite{daskalakis2018training} and also used to train GANs.
Finally, \cref{algo_i4} was proposed by Malitsky~\cite{malitsky2015projected}
for solving monotone variational inequality problems.
(Some of these algorithms were originally proposed in conjunction with
projections or other operations that make them more distinct.)
In all four algorithms, the vectors $x^k_1$ and $x^k_2$ are algorithm states,
$\eta$ is a tunable parameter,
and $F^k(\cdot)$ is the gradient of the loss function at time step $k$.
\mnote{check}

\vspace{-1em}
\noindent\hfil
\begin{minipage}[t]{0.45\textwidth}
	\begin{algorithm}[H]
		\centering
		\caption{(Modified Arrow--Hurwicz)}
		\label{algo_i1}
		\begin{algorithmic}
			\FOR{$k=1, 2,\dots$}
			\STATE{$x^{k+1}_1 = x^k_1 - \eta F^k(x^k_2)$}
			\STATE{$x^{k+1}_2 = x^{k+1}_1 - \eta F^k(x^k_2)$}
			\ENDFOR
		\end{algorithmic}
	\end{algorithm}
\end{minipage}
\hfil
\begin{minipage}[t]{0.45\textwidth}
	\begin{algorithm}[H]
		\centering
		\caption{(Extrapolation from the past)}
		\label{algo_i2}
		\begin{algorithmic}
			\FOR{$k=1, 2,\dots$}
			\STATE{$x^k_2 = x^k_1 - \eta F^{k-1}(x^{k-1}_2)$}
			\STATE{$x^{k+1}_1 = x^k_1 - \eta F^k(x^{k}_2)$}
			\ENDFOR
		\end{algorithmic}
	\end{algorithm}
\end{minipage}
\hfil

\noindent\hfill
\begin{minipage}[t]{0.45\textwidth}
	\begin{algorithm}[H]
		\centering
		\caption{(Optimistic Mirror Descent)}
		\label{algo_i3}
		\begin{algorithmic}
			\FOR{$k=1, 2,\dots$}
			\STATE{$x^{k+1}_2 =  x^k_2 - 2\eta F^k(x^k_2) + \eta F^{k-1}(x^{k-1}_2)$}
			\ENDFOR
		\end{algorithmic}
	\end{algorithm}
\end{minipage}
\hfill
\begin{minipage}[t]{0.45\textwidth}
	\begin{algorithm}[H]
		\centering
		\caption{(Reflected Gradient Method)}
		\label{algo_i4}
		\begin{algorithmic}
			\FOR{$k=1, 2,\dots$}
			\STATE{$x^{k+1}_1 =  x^k_1 - \eta F^k( 2x^k_1 - x^{k-1}_1 )$}
			\ENDFOR
		\end{algorithmic}
	\end{algorithm}
\end{minipage}
\hfill
\vspace{1em}

\Crefrange{algo_i1}{algo_i4} are equivalent in the sense that when suitably initialized,
the sequences $(x^k_1)_{k\ge 0}$ and $(x^k_2)_{k\ge 0}$ are identical for all four algorithms.\footnote{
	In their original formulations, \cref{algo_i1,algo_i2,algo_i4}
	included projections onto convex constraint sets.
	We assume an unconstrained setting here for illustrative purposes.
	Some of the equivalences no longer hold in the constrained case.}
Although these particular equivalences are not difficult to verify
and many have been explicitly pointed out in the literature,
for example in~\cite{gidel2018a}, algorithm equivalence is not always immediately apparent.

One famous example concerns the relations between
the Chambolle-Pock method, Douglas-Rachford splitting, and the alternating directions method of multipliers (ADMM):
indeed, showing the connection between Chambolle-Pock and Douglas-Rachford requires a full page of mathematics in \cite{chambolle2011first}.
In contrast, our analysis supports a single coherent view of these algorithms
that can be summarized in a commutative diagram (\cref{fig11}).

In this paper, we present a framework for reasoning about algorithm equivalence,
with the ultimate goal of making the analysis and design of algorithms more principled and streamlined.
This includes:
\begin{itemize}
	\item A universal way of representing algorithms, inspired by methods from control theory.
	\item Several definitions of what it means for algorithms to be equivalent.
	\item A computationally efficient way to verify whether two algorithms are equivalent.
\end{itemize}
Briefly, our method is to parse each algorithm to a standard form
as a linear system in feedback with a nonlinearity;
to compute the transfer function of each linear system;
and to check, using a computer algebra system,
if there are parameter values that make the transfer functions equal.

We must point out a tension in our terminology:
the notion of algorithm equivalence we define below is rather broad,
which is in order to discover interesting connections between algorithms.
As a consequence, equivalent algorithms (in our terminology)
can nevertheless be extremely useful for different tasks: for example,
writing one algorithm in different ways can yield different generalizations,
different interpretations,
different computational complexity,
and different numerical stability.
On the other hand, equivalent algorithms will share many properties,
such as convergence, stability, and fixed points.

We also present a software package implementing this framework named \lin{}\footnote{
	Named after Carl Linnaeus, a botanist and zoologist who invented the modern system of naming organisms.},
for the classification and taxonomy of iterative algorithms.
The software is a search engine, where the input is an algorithm described using natural syntax,
and the output is a canonical form for the algorithm along with any known names
and pointers to relevant literature.
The approach described in this paper allows \lin{} to search over
first-order optimization algorithms
such as gradient descent with acceleration,
ADMM,
and the extragradient method. 
As the database in \lin{} grows,
it will help algorithm researchers understand
and efficiently discover connections between algorithms.
More generally, \lin{} advances the goal of making mathematics searchable.

This paper is organized as follows.
In \cref{relatedwork}, we briefly summarize existing literature related to our work.
In \cref{example}, we introduce three examples of equivalent algorithms that motivate our framework.
In \cref{preliminary}, we briefly review important background on linear systems and optimization
used throughout the paper.
We formally define two notions of algorithm equivalence,
\emph{oracle equivalence} and \emph{shift equivalence},
in \cref{equivalence} and discuss how to characterize them
via transfer functions in \cref{charac-oracle,charac-shift}.
Certain transformations can also be identified and characterized with our
framework including \emph{algorithm repetition}, repeating an algorithm multiple times,
and \emph{conjugation}, a transformation using conjugate function oracles.
These are discussed in \cref{repe,conjugation} respectively.
In \cref{package}, we briefly introduce our
software package \lin{} for the classification of iterative algorithms.

\section{Related work}\label{relatedwork}
%

A variety of existing work advances the goal of making mathematics searchable.
This work is too diverse to survey here.
As an example, consider the On-Line Encyclopedia of Integer Sequences:
given a sub-sequence or a keyword, the encyclopedia will find a matching sequence
and return useful information such as mathematical motivation for the sequence and links to other literature \cite{integer-search}.
As a very different example, recent work in deep learning has led to new language models,
such as GPT3, that can generate code snippets, including
machine learning models, javascript applications, and SQL queries \cite{NEURIPS2020_1457c0d6, gpt3sandbox, openai-api, gpt3news}.
As these models are trained from large corpuses of data, we might view such models
as implementing a generalized search.

Within the optimization literature, several standard forms have been proposed to represent
problems and algorithms.
For example, the CVX* modeling languages represent (disciplined) convex optimization problems in a
standard conic form,
building up the representations of complex problems
from a few basic functions and a small set of composition rules \cite{cvx, gb08, udell2014convex, diamond2016cvxpy, shen2017disciplined}.
This paper builds on a foundation developed by Lessard et al. \cite{doi:10.1137/15M1009597}
that represents first-order algorithms as linear systems in feedback with a nonlinearity.
Lessard et al.\ use this representation to analyze convergence properties of an algorithm with integral quadratic constraints.
Our work extends theirs with the insight that such representations can be computed automatically
by a computer.

There are rich connections between many first-order methods for convex optimization.
These algorithms are surveyed in a recent textbook by Ryu and Yin,
which summarizes and unifies several operator splitting methods for convex optimization \cite{ryuyinconvex}.
Many of these connections are well known to experts, but the connections have
traditionally been complex to explain, communicate, or even remember.
For example, Boyd et al.~\cite{MAL-016} write,
``There are also a number of other algorithms distinct from but inspired
by ADMM. For instance, Fukushima \cite{applicationadmm} applies ADMM to a dual
problem formulation, yielding a `dual ADMM' algorithm, which is
shown in \cite{reformulationadmm} to be equivalent to the `primal Douglas-Rachford' method
discussed in \cites[\S3.5.6]{phdthesis}.''
As another example, Chambolle and Pock in~\cite{chambolle2011first}
propose a new primal-dual splitting algorithm and demonstrate
that transformations of their algorithm can yield
Douglas-Rachford splitting and ADMM, using a full page of mathematics
to sketch the connection.
Using our framework, the (many!) relations between
the Chambolle-Pock method, Douglas-Rachford splitting, and ADMM
can be established precisely and conveyed efficiently in a commutative diagram;
see \cref{conjugation} and \cref{fig11} in particular.

\section{Motivating examples}\label{example}
To explain what we mean by algorithm equivalence,
we introduce three motivating examples in this section.
Each provides a different view of how two algorithms might be equivalent.

\noindent
\hfil
\begin{minipage}{0.46\textwidth}
	\begin{algorithm}[H]
		\centering
		\caption{}
		\label{algo1}
		\begin{algorithmic}
			\FOR{$k=0, 1, 2,\ldots$}
			\STATE{$x^{k+1}_1 = 2x^k_1 - x^k_2 - \frac{1}{10} \nabla f(2x^k_1 - x^k_2)$}
			\STATE{$x^{k+1}_2 = x^k_1$}
			\ENDFOR
		\end{algorithmic}
	\end{algorithm}
\end{minipage}
\hfil
\begin{minipage}{0.46\textwidth}
	\begin{algorithm}[H]
		\centering
		\caption{}
		\label{algo2}
		\begin{algorithmic}
			\FOR{$k=0, 1, 2,\ldots$}
			\STATE{$\xi^{k+1}_1 = \xi^k_1 - \xi^k_2 - \frac{1}{5} \nabla f(\xi^k_1)$}
			\STATE{$\xi^{k+1}_2 = \xi^k_2 + \frac{1}{10} \nabla f(\xi^k_1)$}
			\ENDFOR
		\end{algorithmic}
	\end{algorithm}
\end{minipage}
\hfil
\vspace{1em}

The first example consists of \cref{algo1,algo2}.
These algorithms are equivalent in a strong sense:
when suitably initialized,
we may transform the iterates of \cref{algo1} by the invertible linear map
$\xi^k_1 = 2x^k_1 - x^k_2, \xi^k_2 = - x^k_1+x^k_2$
to yield the iterates of \cref{algo2}.
We say that the sequences $(x^k_1)_{k\ge 0}$ and $(x^k_2)_{k\ge 0}$ are
\emph{equivalent}
to sequences $(\xi^k_1)_{k\ge 0}$ and $(\xi^k_2)_{k\ge 0}$
\emph{up to an invertible linear transformation}.

\vspace{-1em}
\noindent
\hfil
\begin{minipage}[t]{0.46\textwidth}
	\begin{algorithm}[H]
		\centering
		\caption{}
		\label{algo3}
		\begin{algorithmic}
			\FOR{$k=0, 1, 2,\ldots$}
			\STATE{${x}^{k+1}_1 = 3{x}^k_1 - 2x^k_2 + \frac{1}{5} \nabla f(-x^k_1 + 2x^k_2)$}
			\STATE{$x^{k+1}_2 = x^k_1$}
			\ENDFOR
		\end{algorithmic}
	\end{algorithm}
\end{minipage}
\hfil
\begin{minipage}[t]{0.46\textwidth}
	\begin{algorithm}[H]
		\centering
		\caption{}
		\label{algo4}
		\begin{algorithmic}
			\FOR{$k=0, 1, 2,\ldots$}
			\STATE{$\xi^{k+1} = \xi^k - \frac{1}{5} \nabla f(\xi^k)$}
			\ENDFOR
		\end{algorithmic}
	\end{algorithm}
\end{minipage}
\hfil
\vspace{1em}

The second example consists of \cref{algo3,algo4}.
These algorithms do not even have the same number of state variables,
so these algorithms are \emph{not} equivalent up to an invertible linear transformation.
But when suitably initialized,
we may transform the iterates of \cref{algo3} by the linear map
$\xi^k = -x^k_1 +2 x^{k}_2$
to yield the iterates of \cref{algo4}.
This transformation is linear but not invertible.
Instead, notice that the sequence of calls to the gradient oracle are identical:
the algorithms satisfy \emph{oracle equivalence},
a notion we will define formally later in this paper.

\vspace{-1em}
\noindent
\hfil
\begin{minipage}[t]{0.46\textwidth}
	\begin{algorithm}[H]
		\centering
		\caption{}
		\label{algo5}
		\begin{algorithmic}
			\FOR{$k=0, 1, 2,\ldots$}
			\STATE{$x^{k+1}_1 = \textnormal{prox}_{f}(x^k_3)$}
			\STATE{$x^{k+1}_2 = \textnormal{prox}_{g}(2x^{k+1}_1 - x^k_3)$}
			\STATE{$x^{k+1}_3 = x^k_3 + x^{k+1}_2 - x^{k+1}_1$}
			\ENDFOR
		\end{algorithmic}
	\end{algorithm}
\end{minipage}
\hfil
\begin{minipage}[t]{0.46\textwidth}
	\begin{algorithm}[H]
		\centering
		\caption{}
		\label{algo6}
		\begin{algorithmic}
			\FOR{$k=0, 1, 2,\ldots$}
			\STATE{$\xi^{k+1}_1 = \textnormal{prox}_{g}(- \xi^k_1 + 2\xi^k_2) + \xi^k_1 - \xi^k_2$}
			\STATE{$\xi^{k+1}_2 = \textnormal{prox}_{f}(\xi^{k+1}_1)$}
			\ENDFOR
		\end{algorithmic}
	\end{algorithm}
\end{minipage}
\hfil
\vspace{1em}

The third example consists of \cref{algo5,algo6}.
With suitable initialization,
they will generate the same sequence of
calls to the proximal operator, ignoring the very first call to one of the oracles.
Specifically, \cref{algo6} is initialized as $\xi^0_1 = x^0_3$, $\xi^0_2 = x^1_1$
and the first call to $\text{prox}_f$ in \cref{algo5} is ignored.
We will say they are equivalent up to a prefix or shift: they satisfy \emph{shift equivalence}.

Generalizing from these motivating examples,
we will call algorithms equivalent when they generate an identical sequence
(e.g., of states or oracle calls) up to some transformations,
with suitable initialization.
To make our ideas formal, we need a few definitions and some ideas from control theory.
We will then revisit those motivating examples and define algorithm equivalence.

\section{Preliminaries}\label{preliminary}

We let $\R^n$ denote the standard Euclidean space of $n$-dimensional vectors, and
use boldface lowercase symbols denote semi-infinite sequences of vectors,
which we index using superscripts.
For example, we may write $\bx \defeq (x^0, x^1, \dots)$,
where $x^k \in \R^n$ for each $k\geq 0$.
Subscripts index components or subvectors:
for example, we may write $x = \sbmat{x_1 \\ x_2} \in \R^n$,
where $x_1\in\R^{n_1}$ and $x_2\in\R^{n - n_1}$.

\subsection{Optimization}\label{opt}

\titleparagraph{Optimization problem, objective, and constraints}
An optimization problem is identified by an objective function and a constraint set.
The objective may be written as the sum of several functions, and the constraint
set may be the intersection of several sets.
As an example, in the optimization problem \cref{eqp1}~\cite{MAL-016}
\beq \label{eqp1}
\ba{ll}
\mbox{minimize} & f(x) + g(z) \\
\mbox{subject to} & Ax + Bz = c,
\ea
\eeq
the objective function is $f(x)+g(z)$ and the constraint set is $\{(x,z): Ax + Bz = c\}$.

\titleparagraph{Oracles}
We assume an oracle model of optimization:
we can only access an optimization problem
by querying oracles
at discrete query points~\cites[\S4]{boyd_vandenberghe_2004}[\S1]{MAL-050}[\S1]{nesterov2018lectures}.
Oracles might include the gradient or proximal operator of a function,
or projection onto a constraint set~\cites[\S6]{doi:10.1137/1.9781611974997}[\S2]{fenchel1953convex}[\S1]{OPT-003}.
Each query to the oracle returns an output such as
the function value, gradient, or proximal operator.
For example, the oracles for problem \cref{eqp1} might include the
gradients or proximal operators of $f$ and $g$,
and projection onto the hyperplane $\{(x,z): Ax + Bz = c\}$.

\subsection{Algorithms}\label{algorithm}

Detecting equivalence between \emph{any} pair of algorithms is beyond the scope of this paper.
Instead, we restrict our attention to equivalence
between iterative linear time invariant optimization algorithms.
In the following section, we provide some intuition and define each of these terms.
Further formalism of these terms will be provided
in the next subsection on control theory.

\titleparagraph{Iterative algorithms}
Given an optimization problem and an initial point $x^0 \in \mathcal{X}$,
an \emph{iterative algorithm} $\mathcal{A}$ generates a sequence of points
$\bx \defeq (x^k)_{k \geq 0}$
by repeated application of the map $\mathcal{A}: \mathcal X \to \mathcal X$.
(We do not distinguish the algorithm from its associated map.)
Hence, $x^{k+1} = \mathcal{A}(x^k)$ for  $k \geq 0$.
We call $x^k$ the \emph{state} of the algorithm at \emph{time} $k$.
We make two important simplifying assumptions when treating algorithms.

First, suppose the operator $\mathcal{A}$ calls each different oracle \emph{exactly once}.
(We will see how to extend our ideas to more complex algorithms later.)
This assumption forbids trivial repetition,
such as $\mathcal{A}' \defeq \mathcal{A} \circ \mathcal{A}$.
Second, we consider algorithms that are \emph{time-invariant}.
In general, one could envision an algorithm $\mathcal{A}^k$ that changes at each timestep.
Such time-varying algorithms are common in practice:
for example, gradient-based methods with diminishing stepsizes.
We view time-varying algorithms as a scheme for switching between
different time-invariant algorithms.
Since our aim is to reason about algorithm equivalence,
we restrict our attention to time-invariant algorithms.
A nice benefit of this restriction is
that we can define algorithm equivalence independently of the choice of initial point.

The formulation $x^{k+1} = \mathcal{A}(x^k)$ is general enough
to include algorithms with multiple timesteps.
For example consider \cref{algo_i4}:
$x_1^{k+1} = x_1^k - \eta F (2 x_1^k - x_1^{k-1})$.
If we define the new state $x_2^k \defeq x_1^{k-1}$
and let $x^k \defeq \sbmat{x_1^k \\ x_2^k}$,
then we may rewrite the algorithm as
\begin{equation}\label{eq:algdemo}
	x^{k+1}
	= \bmat{x_1^{k+1} \\ x_2^{k+1} }
	= \bmat{x_1^k - \eta F( 2x_1^k - x_2^k) \\ x_1^k}
	= \mathcal{A}\left( \bmat{x_1^k\\x_2^k} \right)
	= \mathcal{A}(x^k).
\end{equation}

The algorithm $\mathcal{A}$ contains a combination of oracle calls and state updates.
Define $y^k$ and $u^k$ to be the input and output of the oracles called at time $k$, respectively.
Now, write three separate equations for the state update, oracle input, and oracle output.
Applying this to \cref{eq:algdemo}, we obtain:
\begin{subequations}\label{eq:algdemo0}
	\begin{align}
		\bmat{x_1^{k+1} \\ x_2^{k+1}} &= \bmat{ 1 & 0 \\ 1 & 0} \bmat{x_1^k \\ x_2^k} + \bmat{-\eta \\ 0} u^k  \hspace{-1cm}&&\hspace{-1cm}  \text{(state update)},\\
		y^k &= \bmat{2 & -1} \bmat{x_1^k \\ x_2^k} \hspace{-1cm}&&\hspace{-1cm} \text{(oracle input)}, \\
		u^k &= F(y^k) \hspace{-1cm}&&\hspace{-1cm} \text{(oracle output)}.
	\end{align}
\end{subequations}

\titleparagraph{Oracle sequence}

We have defined an algorithm $\mathcal{A}$ as a map $\mathcal X \to \mathcal X$.
In optimization, it is also conventional to write
an algorithm as a sequence of update equations,
that are executed sequentially on a computer to implement the map.
When this sequence of updates is executed, we may record the
sequence of states or the sequence of oracle calls (oracle and its input pairs),
which we call the oracle sequence.
There may be several ways of writing the algorithm as a sequence of updates,
which may produce different state sequences or oracle sequences.
We are not aware of any practical algorithm for optimization
that may be written to produce two different oracle sequences.
Hence we will assume for now that the oracle sequence produced by an algorithm
is unique.
\footnote{
	This assumption eliminates the possibility that some oracles
	may be permuted 
	without changing the state sequence:
	e.g.,
	
	\vspace{-1em}
	\noindent\hfil
	\begin{minipage}[t]{0.42\textwidth}
		\begin{algorithm}[H]
			\centering
			\caption{}
			\begin{algorithmic}
				\FOR{$k=1, 2,\dots$}
				\STATE{$x^{k+1}_1 = \mathcal A_1(x^k_1)$}
				\STATE{$x^{k+1}_2 = \mathcal A_2(x^k_2)$}
				\STATE{$x^{k+1}_3 = \mathcal A_3(x^{k+1}_1, x^{k+1}_2)$}
				\ENDFOR
			\end{algorithmic}
		\end{algorithm}
	\end{minipage}
	\hfil
	\begin{minipage}[t]{0.42\textwidth}
		\begin{algorithm}[H]
			\centering
			\caption{}
			\begin{algorithmic}
				\FOR{$k=1, 2,\dots$}
				\STATE{$x^{k+1}_2 = \mathcal A_2(x^k_2)$}
				\STATE{$x^{k+1}_1 = \mathcal A_1(x^k_1)$}
				\STATE{$x^{k+1}_3 = \mathcal A_3(x^{k+1}_1, x^{k+1}_2)$}
				\ENDFOR
			\end{algorithmic}
		\end{algorithm}
	\end{minipage}
	\hfil
	\vspace{1em}
	
	\noindent Here, the algorithm may be equally well written with the oracle sequence $(\mathcal A_1, \mathcal A_2, \mathcal A_3)$
	as with the oracle sequence $(\mathcal A_2, \mathcal A_1, \mathcal A_3)$.
	But again, we are not aware of any concrete examples of optimization algorithms
	with this structure.
}
We will revisit this assumption later in the paper (\cref{charac-shift}) to see
how our ideas extend to more complex (not-yet-discovered) algorithms.

\titleparagraph{Linear algorithms}

The equations \cref{eq:algdemo0} have the general \emph{linear} form
\begin{subequations}\label{eqp2}
	\begin{align}
		x^{k+1}  & = Ax^k + Bu^k, \label{eqp2a}\\
		y^k & = Cx^k + Du^k, \label{eqp2b} \\
		u^k & = \phi (y^k) \label{eqp2c}.
	\end{align}
\end{subequations}
We say that a time-invariant algorithm is \emph{linear} if it can be written
in the form of \cref{eqp2}, where $x^k$ is the algorithm state and $\phi$ is the set of oracles.
Here $\phi$ can be any nonlinear map,
including a map with internal state.
For example, the oracle $\phi$ corresponding to the subgradient $\partial f$ of a
nondifferentiable function $f$ might make a choice to ensure
the output is unique and consistent, for example,
by selecting the subgradient of minimum norm;
the oracle $\phi$ corresponding to a stochastic gradient might
include an internal random seed that
ensures the output is unique and deterministic, given the seed.

In the rest of the paper, unless specifically noted,
our discussion is limited to linear algorithms.
We will see that the class of linear algorithms includes commonly used algorithms,
such as accelerated methods, proximal methods, operator splitting methods, and more \cite{hu2020analysis, doi:10.1137/15M1009597}.


The general form \cref{eqp2} represents a convenient parameterization
of linear algorithms in terms of matrices $(A,B,C,D)$,
but it is only a starting point.
For example, \crefrange{algo_i1}{algo_i4} have different $(A,B,C,D)$ parameters
despite being equivalent algorithms.
In the next section, we show how tools from control theory can be brought
to bear on these sorts of representations.

\titleparagraph{Remark}
For an arbitrary state-space realization $(A, B, C, D)$,
the corresponding algorithmic sequence may not exist or may not be unique.
However, any implementable practical algorithm, written as a sequence of update equations
has a corresponding algorithmic sequence that exists and is unique:
it is obtained by performing the steps indicated in the update equations
and recording the values of $x$, $u$, and $y$.

\subsection{Control theory}\label{control}
This subsection provides a brief overview of relevant methods and terminology from control theory.
More detail can be found in standard references such as \cite[Ch.~1--3]{antsaklis2006linear}
and \cite[Ch.~1,2,5]{williams2007linear}.

\titleparagraph{Algorithms as linear systems}
Let $\bu$ denote the entire sequence of $u^k$ and
$\by$ denote the entire sequence of $y^k$.
The equations in \cref{eqp2} can be separated into two parts.
Equations \cref{eqp2a} and \cref{eqp2b} define a map $\bH$
from $\bu$ to $\by$ compactly as $\by = \bH\bu$,
while \cref{eqp2c} defines a map $\bPhi$ from $\by$ to $\bu$ as $\bu = \bPhi\by$, where $\bPhi = \mathrm{diag}\{\phi,\phi,\dots\}$.
We can represent these algebraic relations visually via the \emph{block-diagram} shown in \cref{figp1}.

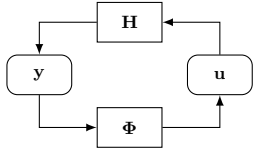
\begin{figure}[tbhp]
	\centering
	\begin{tikzpicture}[>=latex]
		
		\node[scale = 0.75][box9] at (0, 0) (algo) {$\bH$};
		\node[scale = 0.75][box9] at (0, -1.4) (oracle) {$\bPhi$};
		\node[scale = 0.75][box10] at (1.2, -0.7) (input) {$\bu$};
		\node[scale = 0.75][box10] at (-1.2, -0.7) (output) {$\by$};
		
		\draw[->]  (algo) -| (output);
		\draw[->]  (output) |- (oracle);
		\draw[->]   (oracle) -| (input);
		\draw[->]  (input) |- (algo);
		
	\end{tikzpicture}
	\caption{Block-diagram representation of an algorithm. This is equivalent to the pair of equations $\by = \bH \bu$ and $\bu = \bPhi \by$. }
	\label{figp1}
\end{figure}

Consider map $\bH$ defined by \cref{eqp2a} and \cref{eqp2b}.
For simplicity, we assume that $x^0 = 0$.
As we eliminate $\{x^1, \dots, x^k\}$ from \cref{eqp2a} and \cref{eqp2b},
map $\bH$ can be represented as a semi-infinite matrix,

\begin{equation}\label{eqp3}
	\begin{bmatrix}
		y^0 \\ y^1 \\ y^2 \\ y^3 \\ \vdots
	\end{bmatrix}
	=
	\underbrace{\begin{bmatrix}
			D & 0 & 0 & 0 &\cdots \\
			CB & D & 0 & 0 &\cdots \\
			CAB & CB & D & 0 &\cdots \\
			C(A)^2B & CAB & CB & D & \cdots \\
			\vdots & \ddots & \ddots & \ddots & \ddots
	\end{bmatrix}}_\bH
	\begin{bmatrix}
		u^0 \\ u^1 \\ u^2 \\ u^3 \\ \vdots
	\end{bmatrix}.
\end{equation}

In control theory, map $\bH$ is considered as a (discrete-time) \emph{system} that
maps a sequence of \emph{inputs} $\bu$ to a sequence of \emph{outputs} $\by$.
Map $\bH$ is linear since it can be represented as a semi-infinite matrix.
The matrix representation is lower-triangular and it indicates $\bH$ is \emph{causal}.
Further, $\bH$ is time-invariant because the matrix representation is (block) \emph{Toeplitz},
which means that $\bH$ is (block) constant along diagonals from top-left to bottom right.
Thus, $\bH$ is a \emph{causal linear time-invariant system}.
For the rest of this paper, we will work with such systems
and we will refer to such systems as \emph{linear systems}.

Further, to combine maps $\bH$ and $\bPhi$ together,
a linear algorithm in the form of \cref{eqp2} can be regarded as
a linear system connected in feedback with a nonlinearity shown by \cref{figp1}.
At time $k$, $u^k$ is the input and $y^k$ is the output of the system.
Nonlinear feedback $\phi$ represents the set of oracles
such as the gradient or subgradient of a convex function
and it maps the output $y^k$ to the input $u^k$.

\titleparagraph{State-space realization}
Reconsider equations \cref{eqp2a} and \cref{eqp2b}.
They correspond to the \emph{state-space realization} of system $\bH$.
In control theory, a state-space realization is characterized
by an internal sequence of \emph{states} $\bx$
that evolves according to a difference equation with parameters $(A,B,C,D)$:
\begin{equation}\label{eqp4}
	\begin{aligned}
		x^{k+1}  & = Ax^k + Bu^k, \\
		y^k & = Cx^k + Du^k,
	\end{aligned}
	\qquad
	\text{or equivalently, }
	\bmat{x^{k+1} \\ y^k} = L \bmat{ x^k \\ u^k},
	\text{ where }
	L = \bmat{A & B \\ C & D}.
\end{equation}
%
%
%
Here, $u^k \in \R^m$, $y^k \in \R^p$, and $x^k \in \R^n$.
The parameters $(A,B,C,D)$ are matrices of compatible dimensions,
so  $A\in \R^{n\times n}$, $B \in \R^{n\times m}$, $C \in \R^{p\times n}$, and $D \in \R^{p\times m}$.
The state-space realization corresponding to the system $\bH$ can also be characterized by omitting all vectors and writing the block matrix $L$ shown in \cref{eqp4} (right), which is the map from $(x^k,u^k)$ to $(x^{k+1},y^k)$.

In this paper, we rely on such formalism that represents algorithms as linear systems
using a state-space realization as \cref{eqp4} for each algorithm,
following~\cite{hu2020analysis, doi:10.1137/15M1009597}.
The state-space realization $L$ represents the linear part of an algorithm and
map $\phi$ represents the nonlinear part.
Moreover, we have $\mathcal{A} = (L, \phi)$.
In this way, we can unroll \cref{figp1} in time to obtain
the block-diagram shown in \cref{figp2}.
Each dashed box in \cref{figp2} represents map $\mathcal{A}$ for each iteration.

\begin{figure}[tbhp]
	\centering
	\begin{tikzpicture}[>=latex]
		
		\node[scale = 0.75][box6] at (0,-0.2) (algo1) {$L$};
		\node[scale = 0.75] at (0,0) (ref1) {};
		\node[scale = 0.75] at (0,-0.4) (ref2) {};
		\node[scale = 0.75][box3, left of = ref1 , node distance = 8em] (state0) {$x^{k-1}$};
		\node[scale = 0.75][left of = ref1, node distance = 13em] (init1) {\ldots};
		\node[scale = 0.75][box3, right of = ref1 , node distance = 8em] (state1) {$x^{k}$};
		\node[scale = 0.75][box] at (0, -1) (oracle1) {$\phi$};
		\node[scale = 0.75][box3, right of = oracle1, node distance = 6em] (output1) {$y^{k-1}$};
		\node[scale = 0.75, left of= oracle1, node distance = 6em][box3] (input1) {$u^{k-1}$};
		\node[scale = 0.75][box6, right of = algo1, node distance = 16em] (algo2) {$L$};
		\node[scale = 0.75][box8] at (0, -0.55) (a1) {};
		
		\draw[->]  (init1) -- (state0);
		\draw[->]  (state0) -- (ref1-|algo1.west);
		\draw[->]  (state1-|algo1.east) -- (state1);
		\draw[->]   (input1) |- (ref2-|algo1.west);
		\draw[->]  (output1) -- (oracle1);
		\draw[->]  (oracle1) -- (input1);
		\draw[->]  (ref2-|algo1.east) -| (output1) ;
		\draw[->]  (state1) -- (state1-|algo2.west);
		
		\node[scale = 0.75][box3, right of = state1 , node distance = 16em] (state2) {$x^{k+1}$};
		\node[scale = 0.75][box, right of = oracle1, node distance = 16em] (oracle2) {$\phi$};
		\node[scale = 0.75][box3, right of = oracle2, node distance = 6em] (output2) {$y^{k}$};
		\node[scale = 0.75, left of= oracle2, node distance = 6em][box3] (input2) {$u^{k}$};
		\node[scale = 0.75][box6, right of = algo2, node distance = 16em] (algo3) {$L$};
		\node[scale = 0.75][box8, right of = a1, node distance = 16em] (a2) {};
		
		\draw[->]  (state1-|algo2.east) -- (state2);
		\draw[->]  (input2) |- (ref2-|algo2.west);
		\draw[->]  (output2) -- (oracle2);
		\draw[->]  (oracle2) -- (input2);
		\draw[->]  (ref2-|algo2.east) -| (output2);
		\draw[->]  (state2) -- (state1-|algo3.west);
		
		\node[scale = 0.75][right of = state2, node distance = 16em] (end1) {\ldots};
		\node[scale = 0.75][box, right of = oracle2, node distance = 16em] (oracle3) {$\phi$};
		\node[scale = 0.75][box3, right of = oracle3, node distance = 6em] (output3) {$y^{k+1}$};
		\node[scale = 0.75, left of= oracle3, node distance = 6em][box3] (input3) {$u^{k+1}$};
		\node[scale = 0.75][box8, right of = a2, node distance = 16em] (a3) {};
		
		\draw[->]  (state1-|algo3.east) -- (end1);
		\draw[->]  (input3) |- (ref2-|algo3.west);
		\draw[->]  (output3) -- (oracle3);
		\draw[->]  (oracle3) -- (input3);
		\draw[->]  (ref2-|algo3.east) -| (output3);
		
	\end{tikzpicture}
	\caption{Unrolled-in-time block-diagram representation of an algorithm.}
	\label{figp2}
\end{figure}
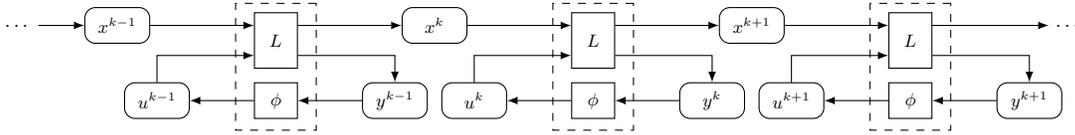



\titleparagraph{Impulse response and transfer function}
From \cref{eqp3}, without the assumption that $x^0 = 0$, we can obtain
\begin{equation}\label{eqp6}
	y^k= C(A)^kx^0+ \sum_{j=0}^{k-1}C(A)^{k-(j+1)}Bu^j + Du^k.
\end{equation}
The output $y^k$ is the sum of $C(A)^kx^0$, which is due to the initial condition $x^0$, and
$\sum_{j=0}^{k-1}C(A)^{k-(j+1)}Bu^j + Du^k$, which is due to the inputs $\{u^0,\dots,u^k\}$.
The compact form $\by = \bH\bu$ and its matrix representation \cref{eqp3}
omit the first term that depends on $x^0$.
These representations are formally equivalent to the state-space model
only when the state is initialized at $x^0 = 0$.
However, linearity of $\bH$ allows the two contributions to be studied separately:
\[
(\text{total response}) = \underbrace{(\text{zero input response})}_{\text{set $u^k=0$ for $k \geq 0$}}
\,+\, \underbrace{(\text{zero state response})}_{\text{set $x^0=0$}}.
\]
This decomposition is analogous to writing the general solution
to a linear differential (or difference) equation as the sum of a
homogeneous solution (due to initial conditions only)
and a particular solution (due to the non-homogeneous terms only).
We will characterize linear systems by their input-output map.
The input-output map depends only on the zero state response,
which allows us to avoid details about initialization.
For simplicity, we denote the entries in the matrix representation of $\bH$ in \cref{eqp3} as
\begin{equation}\label{eqp7}
	H^k = \begin{cases}
		D 		  & k = 0 \\
		C(A)^{k-1}B & k \geq 1
	\end{cases}.
\end{equation}
To study the zero state response, recall from \cref{eqp3} that
\begin{equation}\label{eqp8}
	y^k = H^k u^0 + H^{k-1} u^1 + \cdots + H^1 u^{k-1} + H^0u^k.
\end{equation}
The sequence $(H^k)_{k\geq 0}$ is called the \emph{impulse response} of $\bH$, because it corresponds
to 
the impulsive input $u^0=1$ and $u^j=0$ for $j \geq 1$.

A convenient way to represent $\bH$ is via the use of a \emph{transfer function}.
To this end,
we can represent $\by$ and $\bu$ as generating functions in the variable $z^{-1}$.
Equating powers of $z^{-1}$, we have:
\begin{equation}\label{eqp9}
	\underbrace{\left( y^0 + y^1 z^{-1} + y^{2} z^{-2} + \cdots \right)}_{\hat y(z)} =
	\underbrace{\left( H^0 + H^1 z^{-1} + H^2 z^{-2} + \cdots \right)}_{\hat H(z)}
	\underbrace{\left( u^0 + u^1 z^{-1} + u^{2} z^{-2} + \cdots \right)}_{\hat u(z)}.
\end{equation}
We can recover \cref{eqp8} by expanding the multiplication in \cref{eqp9}
and grouping terms with the same power of $z^{-1}$.
So when written as generating functions, the output is related to the input via multiplication.
The functions $\hat y$ and $\hat u$ are the $z$-transforms of the sequences $\by$ and $\bu$, respectively,
and $\hat H$ is called the \emph{transfer function}. If $p\geq 2$ or $m\geq 2$ (the $H^k$ are matrices),
then $\hat H$ is called the \emph{transfer matrix}.

Substituting \cref{eqp7} into the definition of the transfer function,
we can write a compact form for the formal power series $\hat H$,
which converges on some appropriate set:
\begin{equation}\label{eqp10}
	\hat H(z) = \left[\begin{array}{c|c}
		A & B\\
		\hline
		C & D
	\end{array}\right] = D+\sum_{k=1}^{\infty}C(A)^{k-1}Bz^{-k} = C(zI - A)^{-1}B +D.
\end{equation}
The transfer function $\hat H(z) = C(zI - A)^{-1}B + D$ can be directly computed
from the state-space matrices $(A,B,C,D)$.
Moreover, $\hat H(z)$ is a matrix whose entries are rational functions of $z$.
Hence the transfer function provides a computationally efficient way
to uniquely characterize the input-output map of a system.
We will use the block notation with solid lines to indicate transfer function as in \cref{eqp10}.

\titleparagraph{Linear transformations of state-space realizations}
Consider a linear transformation of the states $x^k$ in \cref{eqp4}.
Specifically, suppose $Q \in \R^{n\times n}$ is invertible,
and define $\tilde{x}^k = Qx^k$ for each $k$.
The new state-space realization in terms of the new variables $\tilde x^k$ is

\begin{equation}\label{eqp11}
	\begin{aligned}
		\tilde x^{k+1}  & = QAQ^{-1} \tilde x^k + QB u^k \\
		y^k & = CQ^{-1} \tilde x^k + Du^k
	\end{aligned},
	\qquad
	\tilde L =
	\bmat{ QAQ^{-1} & QB \\
		CQ^{-1}  &  D }.
\end{equation}
It is straightforward to check that $\bH$ and $\tilde \bH$ have the same transfer function.
Therefore,
whether we apply the linear system $\bH$ or $\tilde \bH$,
the same input sequence $\bu$ will produce the same output sequence $\by$,
although the respective states $x^k$ and $\tilde x^k$ will generally be different.
So although the state-space realization
$(A,B,C,D)$ depend on the coordinates used to represent states $x^k$,
the transfer function is invariant under linear transformations.

This invariance is the key to understanding when two optimization algorithms
are the same, 
even if they look different as written.
For example, this idea alone suffices to show that \cref{algo1,algo2} are equivalent.
\mnote{Better way to connect this idea to optimization?}

\titleparagraph{Minimal realizations}
Every set of appropriately-sized state-space parameters $(A,B,C,D)$
produces a transfer matrix whose entries are rational functions of $z$.
Closer inspection of the formula $\hat H(z) = C(zI-A)^{-1}B + D$
reveals that $\hat H(z) \to D$ as $z\to\infty$.
Therefore, the rational entries of $\hat H(z)$ must be \emph{proper}:
the degree of the numerator cannot exceed the degree of the denominator.
Moreover, the degree of the common denominator of all
entries of $\hat H(z)$ cannot exceed $n$ (the size of the matrix $A$).
Further, given any transfer matrix $\hat H(z)$ whose entries are proper,
there exists at least one realization $(A,B,C,D)$ whose transfer function is $\hat H(z)$.
Any realization of $\hat H(z)$ for which the size of $A$ is as small as possible
is called \emph{minimal}.
All minimal realizations of $\hat H(z)$ are related by
an invertible state transformation via a suitably
chosen invertible matrix $Q$, as in \cref{eqp11}.

Realizations can be non-minimal when
the transfer function has factors
that cancel from both the numerator and denominator.
For example, the following pair of state-space equations both have the same transfer function:
\begin{align*}
	\hat H(z) = & \left[\begin{array}{c|c}
		1 & 1 \\ \hline
		1 & 0
	\end{array}\right] =  1 \cdot (z-1)^{-1} \cdot 1 = \frac{1}{z-1},\\
	%
	\hat H(z) = &\left[\begin{array}{cc|c}
		1 & 2 & 1 \\
		0 & 3 & 0 \\ \hline
		1 & 6 & 0
	\end{array}\right] =
	\begin{bmatrix} 1 & 6\end{bmatrix} \begin{bmatrix}z-1 & -2 \\ 0 & z-3\end{bmatrix}^{-1}\begin{bmatrix}1\\ 0\end{bmatrix} =  \frac{z-3}{z^2-4z+3} = \frac{1}{z-1}.
\end{align*}
We can detect when two optimization algorithms are equivalent,
even when one has additional (redundant) state variables,
by computing their minimal realizations.
This strategy shows that \cref{algo3,algo4} are equivalent.

\titleparagraph{Inverse of state-space realization}
Consider a state-space system $\bH$ with
realization \cref{eqp4} and for which $m=p$ (input and output dimension are the same).
Is it possible to find a state-space system $\bH^{-1}$ that maps $\by$ back to $\bu$? It turns out this is possible
if and only if $D$ is invertible.
In this case, the transfer function of $\bH^{-1}$ is $\hat H^{-1}(z)$,
a matrix whose entries are rational functions of $z$.
One possible state-space realization of the inverse system $\bH^{-1}$ is
\[
\hat H^{-1}(z) =
\left[\begin{array}{c|c}
	A & B \\ \hline
	C & D
\end{array}\right]^{-1}
=
\left[\begin{array}{c|c}
	A-BD^{-1}C & BD^{-1} \\ \hline
	-D^{-1}C & D^{-1}
\end{array}\right].
\]
This explicit realization can be obtained by applying the matrix inversion lemma to \cref{eqp10}. We can extend this idea to partial inverses of linear systems.
Suppose the input sequence $\bu$ is partitioned as
\[
\bu \defeq (u^0, u^1, \dots) = \left( \bmat{u^0_1 \\ u^0_2}, \bmat{u^1_1 \\ u^1_2}, \dots \right)	= \bmat{\bu_1 \\ \bu_2},\quad\text{where } u^k_1 \in \R^{m_1}, u^k_2 \in \R^{m_2}\text{ for all }k\geq 0
\]
and similarly for $\by$.
The matrix $D$ and transfer matrix $\hat H(z)$ can also be partitioned conformally as
\begin{equation}\label{eqp12}
	D = \begin{bmatrix}
		D_{11} & D_{12} \\
		D_{21} & D_{22}
	\end{bmatrix} \quad\text{and}\quad \hat H(z) = \begin{bmatrix}
		\hat H_{11}(z) & \hat H_{12}(z) \\
		\hat H_{21}(z) & \hat H_{22}(z)
	\end{bmatrix},\quad\text{where }D_{ij} \in \R^{p_i \times m_j}\text{ and similarly for }\hat H(z).
\end{equation}
If $D_{11}$ is invertible, we can partially invert $\bH$ with respect to $\bu_1$ and $\by_1$
to form a new system $\bH'$ that maps $(\by_1,\bu_2)\mapsto (\bu_1,\by_2)$.
The transfer function $\hat H'(z)$ of the new system $\bH'$ satisfies
\begin{equation}\label{eqp13}
	\hat H'(z) = \left[\begin{array}{c c}
		\hat H_{11}^{-1}(z) & -\hat H_{11}^{-1}(z)\hat H_{12}(z) \\
		\hat H_{21}(z)\hat H_{11}^{-1}(z) & \hat H_{22}(z) - \hat H_{21}(z)\hat H_{11}^{-1}(z)\hat H_{12}(z)
	\end{array}\right].
\end{equation}

A detailed proof of \cref{eqp13} is presented in \cref{proof-inversesystem}. Note that if $D_{22}$ is invertible,
we can perform a similar partial inverse with respect to the second component.
When an optimization algorithm
is related to another by conjugation of one of the function oracles,
their transfer functions are related by (possibly partial) inversion.

\section{Algorithm equivalence}\label{equivalence}
We are now ready to revisit the motivating examples and formally define algorithm equivalence.

\subsection{Assumptions}\label{assumptions}\

We now formally state the assumptions that we have discussed informally in \cref{preliminary}.
We assume all algorithms throughout the paper satisfy these assumptions unless specifically noted.
\begin{assumption}\label{assump:linearalgo}
	The algorithm is causal, time-invariant, and linear.
\end{assumption}
Any algorithm satisfying \cref{assump:linearalgo}
can be implemented as a sequence of update equations (because it is causal) and
can be written in form \cref{eqp2} (because it is linear and time-invariant).

\begin{assumption}\label{assump:oracle}
	Given an oracle $\phi$, the oracle sequence produced by the algorithm is unique.
\end{assumption}
\Cref{assump:oracle} follows if the output of $\phi$ is
deterministic. It also follows if $\phi$ has internal state but
is deterministic given the sequence of inputs to $\phi$ so far.

Two algorithms can only produce the same sequences if called on the same set of oracles
(or on compatible oracles, for example, related by convex conjugacy). We say that
two algorithms are \emph{comparable} if they use the same or compatible oracles.
\begin{assumption}\label{assump:compare}
	When we compare two algorithms to detect equivalence or other relations,
	we assume that they are comparable.
\end{assumption}
We will discuss several kinds of compatible oracles in the sequel.

\subsection{Oracle equivalence}\label{oracle-equ}\

In the first motivating example,
the algorithms have the same number of states,
and the state sequences are equivalent up
to an invertible linear transformation.
We call these algorithms \emph{state-equivalent}.

In the second motivating example,
the state sequence of \cref{algo3} can be transformed into the
state sequence of \cref{algo4} with a linear transformation.
However, unlike the first motivating example,
the linear transformation is not invertible;
indeed, \cref{algo4} uses fewer state variables than \cref{algo3}.
Instead, recall that
the sequence of calls to the gradient oracle are identical for \cref{algo3,algo4}.
Hence these algorithms are \emph{oracle-equivalent}.

\begin{definition}\label{def1} Two algorithms are oracle-equivalent
	on a set of optimization problems if, for any problem in the set and
	for all possible oracles,
	there exist initializations for both algorithms
	such that the two algorithms generate the same oracle sequence.
\end{definition}

Oracle-equivalent algorithms generate identical sequence regardless of oracles.
For example, if two oracle-equivalent algorithms both call oracle $\nabla f$ and generate identical oracle sequence,
they will still produce identical oracle sequence if we replace oracle $\nabla f$ to $\nabla g$ or every other possible oracle.
Further, oracle equivalence is a symmetric relation.
Notice that if the oracle sequences (that is, the oracles and their arguments $y^k$) are the same,
then the oracles produce the same inputs $u^k$ for the linear systems of each algorithm.
Hence, as shown in \cref{fig5},
oracle-equivalent algorithms have matching input $\bu$ and output $\by$ sequences.
The solid double-sided arrow indicates
the sequences $y^k$ and $\tilde y^k$ are identical,
and the sequences $u^k$ and $\tilde u^k$ are identical.

\begin{figure}[tbhp]
	\centering
	\begin{tikzpicture}[>=latex]
		
		\node[scale = 0.75][box6] at (0,1.9) (algo1) {$L$};
		\node[scale = 0.75] at (0,1.7) (ref2) {};
		\node[scale = 0.75] at (0,2.1) (ref1) {};
		\node[scale = 0.75][box3, left of = ref1 , node distance = 8em] (state0) {$x^{k-1}$};
		\node[scale = 0.75][left of = ref1, node distance = 13em] (init1) {\ldots};
		\node[scale = 0.75][box3, right of = ref1, node distance = 8em] (state1) {$x^{k}$};
		\node[scale = 0.75][box] at (0, 1.1) (oracle1) {$\phi$};
		\node[scale = 0.75][box3, right of = oracle1, node distance = 6em] (output1) {$y^{k-1}$};
		\node[scale = 0.75, left of= oracle1, node distance = 6em][box3] (input1) {$u^{k-1}$};
		\node[scale = 0.75][box6, right of = algo1, node distance = 16em] (algo2) {$L$};
		
		\draw[->]  (init1) -- (state0);
		\draw[->]  (state0) -- (ref1-|algo1.west);
		\draw[->]  (state1-|algo1.east) -- (state1);
		\draw[->]   (input1) |- (ref2-|algo1.west);
		\draw[->]  (output1) -- (oracle1);
		\draw[->]  (oracle1) -- (input1);
		\draw[->]  (ref2-|algo1.east) -| (output1);
		\draw[->]  (state1) -- (state1-|algo2.west);
		
		\node[scale = 0.75][box3, right of = state1, node distance = 16em] (state2) {$x^{k+1}$};
		\node[scale = 0.75][box, right of = oracle1, node distance = 16em] (oracle2) {$\phi$};
		\node[scale = 0.75][box3, right of = oracle2, node distance = 6em] (output2) {$y^{k}$};
		\node[scale = 0.75, left of= oracle2, node distance = 6em][box3] (input2) {$u^{k}$};
		\node[scale = 0.75][box6, right of = algo2, node distance = 16em] (algo3) {$L$};
		
		\draw[->]  (state1-|algo2.east) -- (state2);
		\draw[->]  (input2) |- (ref2-|algo2.west);
		\draw[->]  (output2) -- (oracle2);
		\draw[->]  (oracle2) -- (input2);
		\draw[->]  (ref2-|algo2.east) -| (output2);
		\draw[->]  (state2) -- (state1-|algo3.west);
		
		\node[scale = 0.75][right of = state2 , node distance = 16em] (end1) {\ldots};
		\node[scale = 0.75][box, right of = oracle2, node distance = 16em] (oracle3) {$\phi$};
		\node[scale = 0.75][box3, right of = oracle3, node distance = 6em] (output3) {$y^{k+1}$};
		\node[scale = 0.75, left of= oracle3, node distance = 6em][box3] (input3) {$u^{k+1}$};
		
		\draw[->]  (state1-|algo3.east) -- (end1);
		\draw[->]  (input3) |- (ref2-|algo3.west);
		\draw[->]  (output3) -- (oracle3);
		\draw[->]  (oracle3) -- (input3);
		\draw[->]  (ref2-|algo3.east) -| (output3);
		
		\node[scale = 0.75][box6] at (0,-0.8) (algo10) {$\tilde{L}$};
		\node[scale = 0.75] at (0,-1) (ref10) {};
		\node[scale = 0.75] at (0,-0.6) (ref20) {};
		\node[scale = 0.75][box3, left of = ref10, node distance = 8em] (state00) {$\tilde{x}^{k-1}$};
		\node[scale = 0.75][left of = ref10, node distance = 13em] (init10) {\ldots};
		\node[scale = 0.75][box3, right of = ref10, node distance = 8em] (state10) {$\tilde{x}^{k}$};
		\node[scale = 0.75][box] at (0, 0) (oracle10) {$\phi$};
		\node[scale = 0.75][box3, right of = oracle10, node distance = 6em] (output10) {$\tilde{y}^{k-1}$};
		\node[scale = 0.75, left of= oracle10, node distance = 6em][box3] (input10) {$\tilde{u}^{k-1}$};
		\node[scale = 0.75][box6, right of = algo10, node distance = 16em] (algo20) {$\tilde{L}$};
		
		\draw[->]  (init10) -- (state00);
		\draw[->]  (state00) -- (ref10-|algo10.west);
		\draw[->]  (state10-|algo10.east) -- (state10);
		\draw[->]   (input10) |- (ref20-|algo10.west);
		\draw[->]  (output10) -- (oracle10);
		\draw[->]  (oracle10) -- (input10);
		\draw[->]  (ref20-|algo10.east) -| (output10) ;
		\draw[->]  (state10) -- (state10-|algo20.west);
		
		\node[scale = 0.75][box3, right of = state10 , node distance = 16em] (state20) {$\tilde{x}^{k+1}$};
		\node[scale = 0.75][box, right of = oracle10, node distance = 16em] (oracle20) {$\phi$};
		\node[scale = 0.75][box3, right of = oracle20, node distance = 6em] (output20) {$\tilde{y}^{k}$};
		\node[scale = 0.75, left of= oracle20, node distance = 6em][box3] (input20) {$\tilde{u}^{k}$};
		\node[scale = 0.75][box6, right of = algo20, node distance = 16em] (algo30) {$\tilde{L}$};
		
		\draw[->]  (state10-|algo20.east) -- (state20);
		\draw[->]  (input20) |- (ref20-|algo20.west);
		\draw[->]  (output20) -- (oracle20);
		\draw[->]  (oracle20) -- (input20);
		\draw[->]  (ref20-|algo20.east) -| (output20);
		\draw[->]  (state20) -- (state10-|algo30.west);
		
		\node[scale = 0.75][right of = state20, node distance = 16em] (end10) {\ldots};
		\node[scale = 0.75][box, right of = oracle20, node distance = 16em] (oracle30) {$\phi$};
		\node[scale = 0.75][box3, right of = oracle30, node distance = 6em] (output30) {$\tilde{y}^{k+1}$};
		\node[scale = 0.75, left of= oracle30, node distance = 6em][box3] (input30) {$\tilde{u}^{k+1}$};
		
		\draw[->]  (state10-|algo30.east) -- (end10);
		\draw[->]  (input30) |- (ref20-|algo30.west);
		\draw[->]  (output30) -- (oracle30);
		\draw[->]  (oracle30) -- (input30);
		\draw[->]  (ref20-|algo30.east) -| (output30);
		
		
		\draw[-{Straight Barb[left]}] ($(output1.south) + (0.02,0)$) -- ($(output10.north) + (0.02,0)$);
		\draw[-{Straight Barb[left]}] ($(output10.north) + (-0.02,0)$) -- ($(output1.south) + (-0.02,0)$);
		
		\draw[-{Straight Barb[left]}] ($(output2.south) + (0.02,0)$) -- ($(output20.north) + (0.02,0)$);
		\draw[-{Straight Barb[left]}] ($(output20.north) + (-0.02,0)$) -- ($(output2.south) + (-0.02,0)$);
		
		\draw[-{Straight Barb[left]}] ($(output3.south) + (0.02,0)$) -- ($(output30.north) + (0.02,0)$);
		\draw[-{Straight Barb[left]}] ($(output30.north) + (-0.02,0)$) -- ($(output3.south) + (-0.02,0)$);
		
		\draw[-{Straight Barb[left]}] ($(input1.south) + (0.02,0)$) -- ($(input10.north) + (0.02,0)$);
		\draw[-{Straight Barb[left]}] ($(input10.north) + (-0.02,0)$) -- ($(input1.south) + (-0.02,0)$);
		
		\draw[-{Straight Barb[left]}] ($(input2.south) + (0.02,0)$) -- ($(input20.north) + (0.02,0)$);
		\draw[-{Straight Barb[left]}] ($(input20.north) + (-0.02,0)$) -- ($(input2.south) + (-0.02,0)$);
		
		\draw[-{Straight Barb[left]}] ($(input3.south) + (0.02,0)$) -- ($(input30.north) + (0.02,0)$);
		\draw[-{Straight Barb[left]}] ($(input30.north) + (-0.02,0)$) -- ($(input3.south) + (-0.02,0)$);
		
	\end{tikzpicture}
	\caption{Unrolled block-diagram representation of oracle equivalence.}
	\label{fig5}
\end{figure}
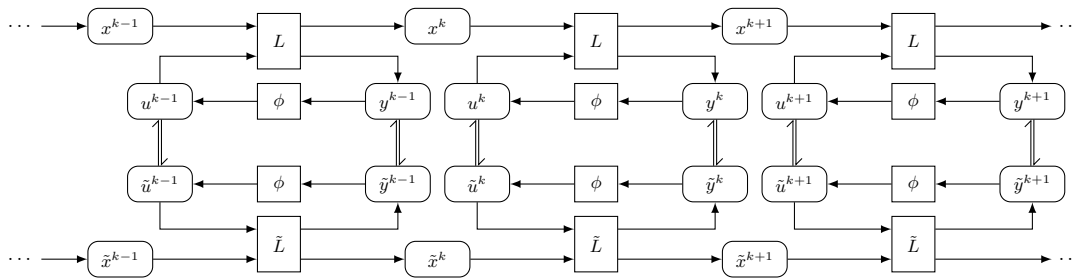

Further, since oracle-equivalent algorithms have identical input and output sequences,
many analytical properties of interest, 
particularly those pertaining to algorithm convergence or robustness, are preserved.
For example, suppose the target problem is to minimize $f(x)$ with $x\in R^n$,
with solution $x^\star$ and corresponding objective value $f(x^\star)$.
Further suppose $f$ is convex and differentiable
with oracle $\nabla f$.
If two algorithms are oracle-equivalent,
the sequence of gradients $\left \| \nabla f(x) \right \|$,
distance to the solution $\left \| x - x^\star \right \|$,
and objective function values $\| f(x) - f(x^\star)\|$ evolve identically,
so they have the same worst-case convergence, etc:
the gradient sequence and objective value are controlled by the oracle sequence.
\mnote{Laurent, it's obvious that gradient sequence and obj value are controlled by the oracle sequence,
	but why is distance in state space also controlled? State space could be different.}
Moreover, even if the oracle is noisy
(e.g., suffers from additive or multiplicative noise, or even adversarial noise),
from the point of view of the oracle, the algorithms are indistinguishable
and any analytical property that involves only the oracle sequence will be the same.

\subsection{Shift equivalence}\label{shift-equ}\

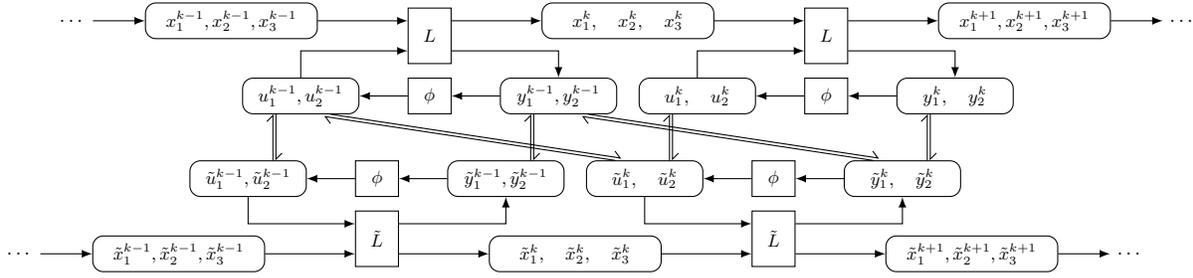
\begin{figure}[tbhp]
	\centering
	\begin{tikzpicture}[>=latex]
		
		\node[scale = 0.75][box6] at (0.35,1.9) (algo1) {$L$};
		\node[scale = 0.75] at (0.35,1.7) (ref2) {};
		\node[scale = 0.75] at (0.35,2.1) (ref1) {};
		\node[scale = 0.75][box2, left of = ref1 , node distance = 10em] (state0) {$x^{k-1}_1, x^{k-1}_2, x^{k-1}_3$};
		\node[scale = 0.75][left of = ref1, node distance = 18em] (init1) {\ldots};
		\node[scale = 0.75][box2, right of = ref1, node distance = 10em] (state1) {$x^{k}_1,\quad x^{k}_2,\quad x^{k}_3$};
		\node[scale = 0.75][box] at (0.35, 1.1) (oracle1) {$\phi$};
		\node[scale = 0.75][box4, right of = oracle1, node distance = 6.5em] (output1) {$y^{k-1}_1, y^{k-1}_2$};
		\node[scale = 0.75, left of= oracle1, node distance = 6.5em][box4] (input1) {$u^{k-1}_1, u^{k-1}_2$};
		\node[scale = 0.75][box6, right of = algo1, node distance = 20em] (algo2) {$L$};
		
		\draw[->]  (init1) -- (state0);
		\draw[->]  (state0) -- (ref1-|algo1.west);
		\draw[->]  (state1-|algo1.east) -- (state1);
		\draw[->]   (input1) |- (ref2-|algo1.west);
		\draw[->]  (output1) -- (oracle1);
		\draw[->]  (oracle1) -- (input1);
		\draw[->]  (ref2-|algo1.east) -| (output1);
		\draw[->]  (state1) -- (state1-|algo2.west);
		
		\node[scale = 0.75][box2, right of = state1, node distance = 20em] (state2) {$x^{k+1}_1, x^{k+1}_2, x^{k+1}_3$};
		\node[scale = 0.75][box, right of = oracle1, node distance = 20em] (oracle2) {$\phi$};
		\node[scale = 0.75][box4, right of = oracle2, node distance = 6.5em] (output2) {$y^{k}_1, \quad y^k_2$};
		\node[scale = 0.75, left of= oracle2, node distance = 6.5em][box4] (input2) {$u^{k}_1, \quad u^k_2$};
		
		\draw[->]  (state1-|algo2.east) -- (state2);
		\draw[->]  (input2) |- (ref2-|algo2.west);
		\draw[->]  (output2) -- (oracle2);
		\draw[->]  (oracle2) -- (input2);
		\draw[->]  (ref2-|algo2.east) -| (output2);
		\node[scale = 0.75][right of = state2, node distance = 8em] (end1) {\ldots};
		\draw[->]  (state2) -- (end1);
		
		\node[scale = 0.75][box6] at (-0.35,-0.8) (algo10) {$\tilde{L}$};
		\node[scale = 0.75] at (-0.35,-1) (ref10) {};
		\node[scale = 0.75] at (-0.35,-0.6) (ref20) {};
		\node[scale = 0.75][box2, left of = ref10, node distance = 10em] (state00) {$\tilde{x}^{k-1}_1, \tilde{x}^{k-1}_2, \tilde{x}^{k-1}_3$};
		\node[scale = 0.75][left of = ref10, node distance = 18em] (init10) {\ldots};
		\node[scale = 0.75][box2, right of = ref10, node distance = 10em] (state10) {$\tilde{x}^{k}_1,\quad \tilde{x}^{k}_2,\quad \tilde{x}^{k}_3$};
		\node[scale = 0.75][box] at (-0.35, 0) (oracle10) {$\phi$};
		\node[scale = 0.75][box4, right of = oracle10, node distance = 6.5em] (output10) {$\tilde{y}^{k-1}_1, \tilde{y}^{k-1}_2$};
		\node[scale = 0.75, left of= oracle10, node distance = 6.5em][box4] (input10) {$\tilde{u}^{k-1}_1, \tilde{u}^{k-1}_2$};
		\node[scale = 0.75][box6, right of = algo10, node distance = 20em] (algo20) {$\tilde{L}$};
		
		\draw[->]  (init10) -- (state00);
		\draw[->]  (state00) -- (ref10-|algo10.west);
		\draw[->]  (state10-|algo10.east) -- (state10);
		\draw[->]   (input10) |- (ref20-|algo10.west);
		\draw[->]  (output10) -- (oracle10);
		\draw[->]  (oracle10) -- (input10);
		\draw[->]  (ref20-|algo10.east) -| (output10) ;
		\draw[->]  (state10) -- (state10-|algo20.west);
		
		\node[scale = 0.75][box2, right of = state10 , node distance = 20em] (state20) {$\tilde{x}^{k+1}_1, \tilde{x}^{k+1}_2, \tilde{x}^{k+1}_3$};
		\node[scale = 0.75][box, right of = oracle10, node distance = 20em] (oracle20) {$\phi$};
		\node[scale = 0.75][box4, right of = oracle20, node distance = 6.5em] (output20) {$\tilde{y}^{k}_1, \quad \tilde{y}^{k}_2$};
		\node[scale = 0.75, left of= oracle20, node distance = 6.5em][box4] (input20) {$\tilde{u}^{k}_1, \quad \tilde{u}^{k}_2$};
		
		\draw[->]  (state10-|algo20.east) -- (state20);
		\draw[->]  (input20) |- (ref20-|algo20.west);
		\draw[->]  (output20) -- (oracle20);
		\draw[->]  (oracle20) -- (input20);
		\draw[->]  (ref20-|algo20.east) -| (output20);
		\node[scale = 0.75][right of = state20, node distance = 8em] (end10) {\ldots};
		\draw[->]  (state20) -- (end10);
		
		
		\draw[-{Straight Barb[left]}] ($(output1.south) + (-0.33,0)$) -- ($(output10.north) + (0.37,0)$);
		\draw[-{Straight Barb[left]}] ($(output10.north) + (0.33,0)$) -- ($(output1.south) + (-0.37,0)$);
		
		\draw[-{Straight Barb[left]}] ($(output1.south) + (0.39,0)$) -- ($(output20.north) + (-0.31,+0.03)$);
		\draw[-{Straight Barb[left]}] ($(output20.north) + (-0.39,0)$) -- ($(output1.south) + (0.31,-0.03)$);
		
		\draw[-{Straight Barb[left]}] ($(output2.south) + (-0.33,0)$) -- ($(output20.north) + (0.37,0)$);
		\draw[-{Straight Barb[left]}] ($(output20.north) + (0.33,0)$) -- ($(output2.south) + (-0.37,0)$);
		
		\draw[-{Straight Barb[left]}] ($(input1.south) + (-0.33,0)$) -- ($(input10.north) + (0.37,0)$);
		\draw[-{Straight Barb[left]}] ($(input10.north) + (0.33,0)$) -- ($(input1.south) + (-0.37,0)$);
		
		\draw[-{Straight Barb[left]}] ($(input1.south) + (0.39,0)$) -- ($(input20.north) + (-0.31,+0.03)$);
		\draw[-{Straight Barb[left]}] ($(input20.north) + (-0.39,0)$) -- ($(input1.south) + (0.31,-0.03)$);
		
		\draw[-{Straight Barb[left]}] ($(input2.south) + (-0.33,0)$) -- ($(input20.north) + (0.37,0)$);
		\draw[-{Straight Barb[left]}] ($(input20.north) + (0.33,0)$) -- ($(input2.south) + (-0.37,0)$);
		
	\end{tikzpicture}
	\caption{Unrolled block-diagram representation of shift equivalence.}
	\label{fig6}
\end{figure}

Now consider \cref{algo5,algo6} from the third motivating example.
They are not oracle-equivalent.
However, their input and output sequences become
identical after shifting \cref{algo5} one step backward:
these algorithms are \emph{shift-equivalent}.

\begin{definition}\label{def2}
	Two algorithms are shift-equivalent
	on a set of problems if,
	for any problem in the set and
	for all possible oracles,
	there exist initializations for both algorithms
	such that the oracle sequences match up to a prefix.
\end{definition}

Shift equivalence can also be interpreted as oracle equivalence up to a shift.
We depict shift equivalence graphically in \cref{fig6}.
Conversely, oracle equivalence can be regarded as a special case of shift equivalence,
where the oracle sequences match without any shift.
Besides, similar as oracle equivalence, shift equivalence is also symmetric.

\subsection{Discussion}\label{discussion-equ}\

\titleparagraph{One algorithm, many interpretations}
Is it useful to have many different forms of an algorithm, if
all the forms are (oracle- or shift-)equivalent?
Yes: different rewritings of one algorithm
often yield different (``physical'') intuition.
For example,
\cref{algo_i1} uses the current
loss function for extrapolation~\cite{vasilyev2010extragradient};
while \cref{algo_i2} seems to extrapolate
from the previous loss function~\cite{censor2011subgradient}.
Equivalent algorithms can differ in memory usage, computational efficiency,
or numerical stability.
For example, implementations of
\cref{algo_i3,algo_i4} lead to different memory usage~\cite{daskalakis2018training, malitsky2015projected}.
In each time step $k$, \cref{algo_i3} needs to store $x_2^{k}, x_2^{k+1}$ and $F^k(\cdot)$,
but \cref{algo_i4} only needs to store $x_1^{k}$ and $x_1^{k+1}$ in memory.
These different rewritings also naturally yield different generalizations,
for example, by projecting different state variables.

\titleparagraph{Limitations}
Do these formal notions of equivalence capture everything an
optimization expert might mean by ``equivalent algorithms''?
No: an example is shown in \cref{algop3}.
\Cref{algop3,algo4} are related by a nonlinear state transformation,
$x^k = \textnormal{exp}(\xi^{k})$. However, none of the equivalences we have
discussed capture this example.
The difficulty is that \cref{algop3} is a nonlinear algorithm,
while all of our machinery for detecting algorithm equivalence requires linearity.
While notions of nonlinear equivalence are certainly interesting,
in this paper we will define only those types of equivalence
that our framework can detect.

\begin{algorithm}[H]
	\centering
	\caption{}
	\label{algop3}
	\begin{algorithmic}
		\FOR{$k=1, 2,\dots$}
		\STATE{$x^{k+1} = x^k \textnormal{exp} ( - \frac{1}{5} \nabla f(\textnormal{log} x^k))$}
		\ENDFOR
	\end{algorithmic}
\end{algorithm}

\section{A characterization of oracle equivalence}\label{charac-oracle}
In this section, we will discuss how to characterize oracle equivalence
via transfer functions.
Recall that oracle equivalence, introduced in \cref{equivalence},
characterizes an algorithm by its oracle sequence.
This sequence is uniquely determined by
the initialization of the algorithm (which we ignore)
and the input-output map of the linear system representing the algorithm.
While the state-space realization of two equivalent algorithms may differ,
from \cref{control}, recall that the transfer function of a linear system
uniquely characterizes the system as an input-output map.
Fortunately, using \cref{eqp10},
we can directly calculate the transfer function
from the state-space realization of an algorithm;
and we can use equality of transfer functions
to check if two algorithms are equivalent.
This machinery allows us to avoid the issue of initialization
(or of the optimization problem!) entirely,
as we can check algorithm equivalence without ever producing a sequence of iterates.

More formally,
consider two oracle-equivalent algorithms
with the same number of oracle calls in each iteration.
From \cref{oracle-equ}, we know that for every optimization problem,
and for all possible oracles,
there exist initializations for both algorithms
so that the oracle sequence of the two algorithms is the same.
Concretely, by picking the initializations of both algorithms appropriately,
we can ensure that the first output of the linear systems match.
Hence (since the oracles are the same), the first input of the linear systems match,
and so the second output of the linear systems match, etc.
By induction, for each possible sequence of input $\bu$,
they produce identical sequences of output $\by$.
Then from \cref{control}, the algorithms must have identical impulse responses
and consequently identical transfer functions.
In light of the previous discussion,
we have proved the following proposition,
since each step in the reasoning above is necessary and sufficient.
We defer a detailed mathematical proof to \cref{proof-oracleequivalence}.

\begin{proposition}\label{prop1}
	Algorithms with the same oracle calls in each iteration are oracle-equivalent
	if and only if they have identical transfer functions.
\end{proposition}

Importantly, oracle-equivalent algorithms have the same transfer function,
even if they have a different number of state variables.
But any realization of the algorithm must have
at least as many state variables as
the minimal realization of the linear system.

\titleparagraph{Remark}
It is meaningless to compare algorithms with different oracle calls,
as two algorithms are oracle equivalent if
there exist initializations
for both algorithms such that they generate the same oracle sequence.
Hence throughout this section, we make \cref{assump:compare}:
when we compare two algorithms, we assume both algorithms use the same set of oracles.
In this case, by \cref{control}, we can always initialize both algorithms at zero to satisfy the requirement of oracle equivalence.
For any algorithm that involves constant terms in its state-space realization, we can affinely transform
it into an equivalent state-space realization without constant terms.
Under this affine transformation,
zero still satisfies the requirements of initialization for oracle-equivalence.
This justifies our approach to characterize oracle equivalence with transfer functions and ignore the initializations.
Further, from \cref{eqp6}, the effect of initialization diminishes as time step goes to infinity,
thus, asymptotically initialization does not affect the behavior of an algorithm such as convergence properties.

Oracle-equivalent algorithms have identical oracle sequences
and hence converge to the same fixed point (if they converge).
Suppose algorithm $\mathcal{A}_1: \mathcal X \to \mathcal X$ with (nonlinear) oracle $\phi: \mathcal X \to \mathcal X$
and state-space realization $(A_1, B_1, C_1, D_1)$,
converges to a fixed point $(y^\star, u^\star, x^\star)$ that satisfies
\begin{equation}\label{eq36}
	\begin{aligned}
		x^\star & = A_1 x^\star + B_1 u^\star \\
		y^\star & = C_1 x^\star + D_1 u^\star \\
		u^\star & = \phi(y^\star).
	\end{aligned}
\end{equation}
If algorithm $\mathcal{A}_2$ is oracle-equivalent to $\mathcal{A}_1$,
$\mathcal{A}_2$ converges to a fixed point $(y^\star, u^\star, \tilde x^\star)$
that has the same output and input as the fixed point of $\mathcal{A}_1$;
however, the state $\tilde x^\star$ may not be the same,
or even have the same dimension.

Further, if there is an invertible linear map $Q$ between the states of $\mathcal{A}_1$ and $\mathcal{A}_2$ and
$(y^\star, u^\star, x^\star)$ is a fixed point of $\mathcal{A}_1$,
then $(y^\star, u^\star, Qx^\star)$ is a fixed point of $\mathcal{A}_2$.
We can use this fact to derive a relation between the
state-space realizations of the two algorithms:
the fixed point equation for $\mathcal{A}_2$ can be written as
\begin{equation}\label{eq37}
	\begin{aligned}
		Qx^\star & = QA_1Q^{-1} Qx^\star + QB_1 u^\star \\
		y^\star & = C_1Q^{-1} Qx^\star + D_1 u^\star \\
		u^\star & = \phi(y^\star),
	\end{aligned}
\end{equation}
which shows that the state-space realization
of $\mathcal{A}_2$ is
\begin{equation}\label{eq2}
	\bmat{
		QA_1Q^{-1} & QB_1 \\
		C_1Q^{-1}  &  D_1},
\end{equation}
which can be obtained by \cref{eqp11}.

\subsection{Motivating examples: proof of equivalence}
Now, we will revisit the first and second motivating examples
and apply \cref{prop1} to show equivalence.
We perform the computation using the gradient oracle ($\nabla f$) as the oracle
to compute the state-space realizations and transfer functions.

\titleparagraph{\Cref{algo1,algo2}}
The state-space realization and transfer function of \cref{algo1} are shown as
\begin{displaymath}
	\hat H_1(z) = \left[\begin{array}{c c|c}
		2 & -1 &-\frac{1}{10}\\
		1 & 0 & 0  \\
		\hline
		2 & -1 & 0
	\end{array}\right] =
	\left[\begin{array}{c c} 2 & -1 \end{array} \right] \left(zI - \left[\begin{array}{c c} 2 & -1 \\
		1 & 0 \end{array}\right]\right)^{-1} \left[\begin{array}{c} -\frac{1}{10}\\  0 \end{array} \right] = \frac{-2z + 1}{10(z-1)^2}.
\end{displaymath}
The state-space realization and the transfer function of \cref{algo2} are
\begin{displaymath}
	\hat H_2(z) = \left[\begin{array}{c c|c}
		1 & -1 &-\frac{1}{5}\\
		0 & 1 & \frac{1}{10}  \\
		\hline
		1 & 0 & 0
	\end{array}\right] =
	\left[\begin{array}{c c} 1 & 0 \end{array} \right] \left(zI - \left[\begin{array}{c c} 1 & -1 \\
		0 & 1 \end{array}\right]\right)^{-1} \left[\begin{array}{c} -\frac{1}{5}\\  \frac{1}{10} \end{array} \right] = \frac{-2z + 1}{10(z-1)^2}.
\end{displaymath}
Hence we see \cref{algo1,algo2}
have the same transfer function,
so by \cref{prop1} they are oracle-equivalent.
In fact, since the algorithms have the same number of state variables,
there exists an invertible linear transformation
\[
Q = \bmat{ 2 & -1 \\ -1 & 1}
\]
to convert the state-space realization of \cref{algo1}
to the state-space realization of \cref{algo2} following \cref{eqp11}.

\titleparagraph{\Cref{algo3,algo4}}
The state-space realization and transfer function of \cref{algo3} are
\begin{displaymath}
	\hat H_3(z) = \left[\begin{array}{c c|c}
		3 & -2 & \frac{1}{5}\\
		1 & 0 & 0  \\
		\hline
		-1 & 2 & 0
	\end{array}\right] =
	\left[\begin{array}{c c} -1 & 2 \end{array} \right] \left(zI - \left[\begin{array}{c c} 3 & -2 \\
		1 & 0 \end{array}\right]\right)^{-1} \left[\begin{array}{c} \frac{1}{5}\\  0 \end{array} \right] = -\frac{1}{5(z-1)}.
\end{displaymath}
The state-space realization and transfer function of \cref{algo4} are
\begin{displaymath}
	\hat H_4(z) = \left[\begin{array}{c|c}
		1 & -\frac{1}{5}\\
		\hline
		1 &  0
	\end{array}\right] =
	\left[\begin{array}{c} 1 \end{array} \right] \left(zI - \left[\begin{array}{c} 1 \end{array}\right]\right)^{-1}
	\left[\begin{array}{c} -\frac{1}{5} \end{array} \right] = -\frac{1}{5(z-1)}.
\end{displaymath}
\Cref{algo3,algo4} have the same transfer function,
so by \cref{prop1} they are oracle-equivalent.
On the other hand, they have different numbers of states.
Consider the invertible linear transformation
\begin{displaymath}
	Q = \left[ \begin{array}{c c} -1 & 2 \\ 0 & 1 \end{array}\right].
\end{displaymath}
Applying $Q$ to the state-space realization of \cref{algo3} leads to
\begin{displaymath}
	\left[\begin{array}{c c:c}
		1 & 0 & -\frac{1}{5}\\
		-1 & 2 & 0  \\
		\hdashline
		1 & 0 & 0
	\end{array}\right],
\end{displaymath}
where we have used dashed lines to demarcate the blocks in the state-space realization.
This has the same minimal realization as \cref{algo4} by \cref{control}. 
\begin{displaymath}
	\left[\begin{array}{c:c}
		1 & -\frac{1}{5}\\
		\hdashline
		1 & 0
	\end{array}\right].
\end{displaymath}
Note that the state-space realization of \cref{algo4} is a minimal realization.
This shows the reason why \cref{algo3,algo4} are equivalent
even if they have different numbers of states.

Now we show how the sausage was made.
\Cref{algo3} was designed by starting with the more complex \emph{Triple momentum algorithm}
\cref{algo13}~\cite{doi:10.1137/15M1009597,tmm} and choosing parameters of the algorithm
so its transfer function matched \cref{algo4}.

\renewcommand{\thealgorithm}{6.\arabic{algorithm}}
\setcounter{algorithm}{0}

\begin{algorithm}[H]
	\caption{Triple momentum algorithm}
	\label{algo13}
	\begin{algorithmic}
		\FOR{$k=0, 1, 2,\ldots$}
		\STATE{$x^{k+1}_1 = (1+\beta)x^k_1 - \beta x^k_2 - \alpha \nabla f((1+\eta)x^k_1 - \eta x^k_2)$}
		\STATE{$x^{k+1}_2 = x^k_1 $}
		\ENDFOR
	\end{algorithmic}
\end{algorithm}

The state-space realization and transfer function of \cref{algo13} are
\begin{equation}\label{eq7}
	\hat H_7(z) = \left[\begin{array}{c c|c}
		1+\beta & -\beta & -\alpha \\
		1 & 0 & 0  \\
		\hline
		1+\eta & -\eta & 0
	\end{array}\right]
	= -\frac{\alpha((\eta+1)z-\eta)}{(z-1)(z-\beta)}.
\end{equation}
We now demand that \cref{eq7}(right) must equal the transfer function of \cref{algo4}
for all values of $z$, resulting in the equations
\begin{equation}\label{eq3}
	\begin{aligned}
		& 5\alpha(\eta+1) = 1 \\
		& 5\alpha\eta = \beta.
	\end{aligned}
\end{equation}
We solve for the parameters $\alpha$, $\eta$ and $\beta$
to find a solution $\alpha =-\frac{1}{5}$, $\beta = 2$ and $\eta = -2$ to \cref{eq3}
that corresponds to \cref{algo3}.
Other solutions exist:
for example, $\alpha = 1$, $\beta = -4$ and $\eta = -\frac{4}{5}$ solves \cref{eq3}
and yields another (different!) algorithm equivalent to \cref{algo4}.

\section{A characterization of shift equivalence}\label{charac-shift}
We can also characterize shift equivalence using transfer functions.
Suppose an algorithm uses more than one oracle,
and the call to the second oracle depends on the value of the first.
Take \cref{algo5} as example:
at iteration $k$,
the first update equation calls the oracle $\textnormal{prox}_{f}$ to compute $x^{k+1}_1 = \textnormal{prox}_{f}(x^k_3)$,
and the second update equation calls the oracle $\textnormal{prox}_{g}$ to compute $x^{k+1}_2 = \textnormal{prox}_{g}(2x^{k+1}_1 - x^k_3)$.
This second update relies on the value of $x^{k+1}_1$.
Imagine now that we reorder the update equations by some permutation.
Generally this change produces an entirely different algorithm.
But if the permutation is a \emph{cyclic} permutation,
the order of the oracle calls is preserved.
In the example of \cref{algo5}, we could start with the update equation
$x^{k+1}_2 = \textnormal{prox}_{g}(2x^{k+1}_1 - x^k_3)$
and produce exactly the same sequence of oracle calls (after the first)
by initializing $x^{k+1}_1$ and $x^k_3$ appropriately.
This new algorithm is shift-equivalent to \cref{algo5} by \cref{def2}.

\Cref{algo5} has three update equations, and so there are two other
algorithms that may be produced by cyclic permutations of \cref{algo5},
shown below as \cref{algo5_1,algo5_2}.

\renewcommand{\thealgorithm}{7.\arabic{algorithm}}
\setcounter{algorithm}{0}
\begin{minipage}[t]{0.46\textwidth}
	\begin{algorithm}[H]
		\centering
		\caption{}
		\label{algo5_1}
		\begin{algorithmic}
			\FOR{$k=0, 1, 2,\ldots$}
			\STATE{$x^{k+1}_2 = \textnormal{prox}_{g}(2x^{k}_1 - x^k_3)$}
			\STATE{$x^{k+1}_3 = x^k_3 + x^{k+1}_2 - x^{k}_1$}
			\STATE{$x^{k+1}_1 = \textnormal{prox}_{f}(x^{k+1}_3)$}
			\ENDFOR
		\end{algorithmic}
	\end{algorithm}
\end{minipage}
\hfil
\begin{minipage}[t]{0.46\textwidth}
	\begin{algorithm}[H]
		\centering
		\caption{}
		\label{algo5_2}
		\begin{algorithmic}
			\FOR{$k=0, 1, 2,\ldots$}
			\STATE{$x^{k+1}_3 = x^k_3 + x^{k}_2 - x^{k}_1$}
			\STATE{$x^{k+1}_1 = \textnormal{prox}_{f}(x^{k+1}_3)$}
			\STATE{$x^{k+1}_2 = \textnormal{prox}_{g}(2x^{k+1}_1 - x^{k+1}_3)$}
			\ENDFOR
		\end{algorithmic}
	\end{algorithm}
\end{minipage}
\hfil
\vspace{1em}

Both are shift-equivalent to \cref{algo5},
but \cref{algo5_2} is also oracle-equivalent to \cref{algo5}.
(We will revisit and formally prove this result later.)
It is easy to see why:
the oracles $\textnormal{prox}_{f}$ and $\textnormal{prox}_{g}$
are called in the same order in \cref{algo5,algo5_2},
but in the opposite order in \cref{algo5_1}.

We introduce notation to generalize this idea to more complex algorithms.
Consider an algorithm $\mathcal{A}$ that consists of $m$ update equations and
makes $n$ sequential oracle calls in each iteration.
We insist that no update equation may contain more than one oracle call,
so $m \ge n$.
At iteration $k$, the algorithm generates states
$x^k_1, \ldots, x^k_m$,
outputs $y^k_1, \ldots, y^k_n$, and inputs $u^k_1, \ldots, u^k_n$, respectively.
Consider any permutation $\tilde \pi$ of the sequence $(m) = (1, \ldots, m)$.
We call algorithm $\mathcal{B} = P_{\tilde \pi}\mathcal{A}$ a \emph{permutation} of algorithm $\mathcal{A}$
if $\mathcal{B}$ performs the update equations of $\mathcal{A}$ in the order $\tilde \pi$ at each iteration.
The algorithms $\mathcal{A}$ and $\mathcal{B}$
are shift-equivalent
if and only if $\tilde \pi$ is a cyclic permutation of $(m)$.

\begin{proposition}\label{prop3}
	An algorithm and any of its cyclic permutations are shift-equivalent.
	Any two shift-equivalent algorithms are equivalent to cyclic permutations of each other.
\end{proposition}

\begin{proof}
	We provide a proof sketch here,
	and defer a detailed proof to \cref{proof-cyclicpermutation}.
	Let us name the oracle calls of the original algorithm $\mathcal A$
	so that the oracles are called in order $(n)$.
	
	Cyclic permutation implies shift equivalence.
	Suppose $\mathcal{B} = P_{\tilde \pi}\mathcal{A}$ where $\tilde \pi$ is a cyclic permutation of $(m)$.
	The permutation of update equations may reorder the
	oracle calls within one iteration,
	so that the oracle calls in algorithm $\mathcal{B}$
	follow a cyclic permutation $\pi$ of $(n)$ (possibly, the identity).
	Hence $\mathcal{A}$ and $\mathcal{B}$ are shift-equivalent.
	(If the permutation is the identity, then the algorithms are also oracle-equivalent.)
	
	Shift equivalence implies cyclic permutation.
	Suppose algorithms $\mathcal A$ and $\mathcal B$ are shift-equivalent.
	If they are also oracle-equivalent, then they can be written using the same set of update equations.
	If they are not oracle-equivalent, we can always find a cyclic permutation of
	the update equations of $\mathcal A$ that produces the same oracle sequence as $\mathcal B$.
	Therefore $\mathcal A$ and $\mathcal B$
	are equivalent to cyclic permutations of each other.
	(In the first case, the permutation is the identity.)
\end{proof}


\subsection{Reordering oracle calls}\label{odg}

Most optimization algorithms proceed by sequential updates,
each of which depends on the previous update.
However, for completeness, we consider a more general class of
equivalences that arises for algorithms whose oracle updates
have a more complex dependency structure.
We may express the order of oracle calls at each iteration using a directed graph,
where the graph has edge from oracle $i$ to oracle $j$
if oracle call $j$ depends on the result of oracle call $i$ (within the same iteration).
In other words, within the iteration we must call oracle $i$ before oracle $j$.
We call this directed graph the \emph{oracle dependence graph} (ODG) of the algorithm.

An example is provided below as \cref{algo6_1}.
Note that we are not aware of any practical algorithm for optimization
with this ODG.
It is constructed only for illustration.

\begin{minipage}[t]{0.46\textwidth}
	\begin{algorithm}[H]
		\centering
		\caption{}
		\label{algo6_1}
		\begin{algorithmic}
			\FOR{$k=0, 1, 2,\ldots$}
			\STATE{$x^{k+1}_1 = x^k_4 - t\nabla f(x^k_4)$}
			\STATE{$x^{k+1}_2 = x^{k+1}_1 - t\nabla g(x^{k+1}_1)$}
			\STATE{$x^{k+1}_3 = x^{k+1}_1 - t\nabla h(x^{k+1}_1)$}
			\STATE{$x^{k+1}_4 = \textnormal{prox}_{tf}(\frac{1}{2}x^{k+1}_2 + \frac{1}{2}x^{k+1}_3)$}
			\ENDFOR
		\end{algorithmic}
	\end{algorithm}
\end{minipage}
\hfil
\begin{minipage}[t]{0.46\textwidth}
	\begin{algorithm}[H]
		\centering
		\caption{}
		\label{algo6_2}
		\begin{algorithmic}
			\FOR{$k=0, 1, 2,\ldots$}
			\STATE{$x^{k+1}_1 = x^k_4 - t\nabla f(x^k_4)$}
			\STATE{$x^{k+1}_3 = x^{k+1}_1 - t\nabla h(x^{k+1}_1)$}
			\STATE{$x^{k+1}_2 = x^{k+1}_1 - t\nabla g(x^{k+1}_1)$}
			\STATE{$x^{k+1}_4 = \textnormal{prox}_{tf}(\frac{1}{2}x^{k+1}_2 + \frac{1}{2}x^{k+1}_3)$}
			\ENDFOR
		\end{algorithmic}
	\end{algorithm}
\end{minipage}
\hfil
\vspace{1em}

\Cref{fig7} expresses the dependency of oracle calls within each iteration of \cref{algo6_1}.
At each iteration, oracle calls 2 ($\nabla g$) and 3 ($\nabla h$) depends
on the result of oracle call 1 ($\nabla f$); oracle call 4 ($\textnormal{prox}_{tf}$)
depends on the results of oracle calls 1, 2, and 3.

\begin{figure}[tbhp]
	\centering
	\begin{tikzpicture}
		[->,>=stealth',shorten >=1pt,auto,node distance=1.5cm,semithick]
		
		\node[state] (A) [minimum size=0.7cm] {$1$};
		\node[state] (B) [below right of=A] [minimum size=0.7cm] {$2$};
		\node[state] (C) [below left of=A] [minimum size=0.7cm] {$3$};
		\node[state] (D) [below left of=B] [minimum size=0.7cm] {$4$};
		\path (A)  edge node {} (B);
		\path (A)  edge node {} (C);
		\path (B)  edge node {} (D);
		\path (C)  edge node {} (D);
		\path (A)  edge node {} (D);
	\end{tikzpicture}
	\caption{Directed graph representing dependency of oracle calls in \cref{algo6_1}.}
	\label{fig7}
\end{figure}
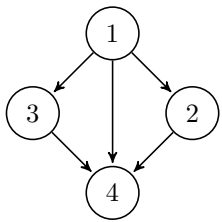

An algorithm is always written as a sequence of update equations.
But some algorithms might have a directed graph that may be written as a sequence (with all edges pointing forward)
in more than one way,
and so can be implemented as a
sequence of oracle calls in more than one way.
For illustration, consider \cref{algo6_1,algo6_2}.
At each iteration, the oracle calls of \cref{algo6_1,algo6_2}
are identical:
that is,
calls to oracles $\nabla f$, $\nabla g$, $\nabla h$, and $\textnormal{prox}_{tf}$ are identical.
The only difference is that the
oracle calls $\nabla g$ and $\nabla h$ are swapped in the oracle sequence at each iteration.
Notice that the state-space realizations of these algorithms \emph{still}
have the same transfer function
(after swapping the second and third columns and rows),
consistent with the fact that \cref{algo6_1,algo6_2}
share the same directed graph of oracle calls (\cref{fig7}).

We know of no practical optimization algorithm like this.
However, were one to be discovered, we would suggest an expanded definition of
oracle equivalence: two algorithms are oracle-equivalent if there exists a
way of writing each algorithm as a sequence of updates so that
both algorithms have the same sequence of oracle calls.
The transfer function still identifies algorithms that are oracle-equivalent in this
expanded sense.

The oracle calls in an algorithm at each iteration are always written in sequential form.
This sequential form is lost in the state-space realization of the algorithm.
However, the order (dependency) of oracle calls is encoded in the $D$ matrix of the state-space realization.
In this sense, the $D$ matrix encodes the adjacency matrix of the directed graph.
We have $D_{ij} \neq 0$ if and only if
oracle call $i$ depends on the results of oracle call $j$ at each iteration.
For example, in the state-space realization of \cref{algo6_1},
the $D$ matrix is
\begin{displaymath}
	\left[\begin{array}{cccc}
		0 & 0 & 0 & 0 \\
		-t & 0 & 0 & 0 \\
		-t & 0 & 0 & 0 \\
		-t & -\frac{1}{2}t & -\frac{1}{2}t & 0
	\end{array}\right].
\end{displaymath}
In light of this discussion, we can strengthen \cref{prop3} to \cref{prop3_1}.
\begin{proposition}\label{prop3_1}
	An algorithm and any of its cyclic permutations are shift-equivalent;
	further, if they share the same $D$ matrix in their state-space realizations,
	they are also oracle-equivalent.
	Any two shift-equivalent algorithms are equivalent to cyclic permutations of each other.
\end{proposition}
If an algorithm contains $m$ update equations and $n$ oracle calls at each iteration ($m \ge n$),
there are $m$ possible cyclic permutations on the update equations.
According to the $D$ matrix in the state-space realization,
we can group the $m$ cyclic permutations into $n$ distinct equivalent classes.
Algorithms within each equivalence class are oracle-equivalent and shift-equivalent,
while algorithms in different equivalent classes are only shift-equivalent.
The $n$ distinct equivalence classes correspond to the
$n$ cyclic permutations of the original order of oracle calls $(n)$.


\subsection{Characterization of cyclic permutation}\

In the remainder of this paper,
let us restrict our attention to algorithms for which
a (cyclic) permutation of the algorithm
changes the update order of oracle calls within one iteration,
or in other words, changes the $D$ matrix in the state-space realization.
In this way, we call algorithm $\mathcal{B} = P_{\pi}\mathcal{A}$ a permutation of algorithm $\mathcal{A}$
if $\mathcal{B}$ performs the update equations of $\mathcal{A}$ in a different order
such that the update order of oracle calls of $\mathcal B$
is $\pi$ at each iteration.

Suppose $\mathcal{A}$ has state-space realization $(A, B, C, D)$,
and $\mathcal{B} = P_{\pi}\mathcal{A}$ where
$\pi = (j+1,\ldots, n, 1,\ldots, j)$ for $1< j < n$ is a cyclic permutation of $(n)$.
We will show how to recognize this relationship between the algorithms
using their transfer functions.
Partition the oracle calls into two parts, $(1, \ldots, j)$ and $(j+1, \ldots, n)$,
and partition the input and output sequences in the same way:
$\bar{\bu}_1$, $\bar{\bu}_2$ for inputs and $\bar{\by}_1$, $\bar{\by}_2$ for outputs.
The state-space realization $L_\mathcal{A}$ and transfer function $\hat H_{\mathcal{A}}(z)$
can also be partitioned accordingly as
\begin{equation}\label{eq8}
	L_\mathcal{A} = \left[\begin{array}{c:c c}
		A & B_1 & B_2 \\
		\hdashline
		C_1 & D_{11} & D_{12} \\
		C_2 & D_{21} & D_{22}
	\end{array}\right],
\end{equation}
\begin{displaymath}
	\hat H_{\mathcal{A}}(z) = \left[\begin{array}{c c}
		C_1(zI - A)^{-1}B_1 + D_{11} & C_1(zI - A)^{-1}B_2 + D_{12} \\
		C_2(zI - A)^{-1}B_1 + D_{21} & C_2(zI - A)^{-1}B_2 + D_{22} \\ \end{array}\right]
	= \left[\begin{array}{c c}
		\hat H_{11}(z) & \hat H_{12}(z) \\
		\hat H_{21}(z) & \hat H_{22}(z) \\
	\end{array}\right].
\end{displaymath}

Now we can say how the transfer function of an algorithm is related to
that of its cyclic permutation. Recall that by \cref{assump:compare},
when we compare transfer functions to detect shift equivalence (or cyclic permutations),
both algorithms call the same set of oracles in each iteration.

\begin{proposition}\label{prop4}
	Instate notation as in \cref{eq8} and assume $D_{12} = 0$. Then $\mathcal{B}$ is equivalent to $P_{\pi}\mathcal{A}$
	if and only if the transfer function of $\mathcal{B}$ satisfies
	\begin{equation}\label{eq11}
		\hat H_{\mathcal{B}}(z) = \left[\begin{array}{c c}
			\hat H_{11}(z) & z\hat H_{12}(z) \\
			\hat H_{21}(z)/z & \hat H_{22}(z) \\
		\end{array}\right].
	\end{equation}
\end{proposition}

\begin{proof}
	We provide a proof sketch here and defer a detailed proof to \cref{proof-shiftequivalence}. The state-space realization of $P_{\pi}\mathcal{A}$ is
	\begin{equation}\label{eq12}
		\left[\begin{array}{c c :c c}
			A & B_1 & 0 & B_2 \\
			0 & 0 & I & 0\\
			\hdashline
			C_1A & C_1B_1 & D_{11} & C_1B_2 \\
			C_2 & D_{21} & 0 & D_{22}
		\end{array}\right].
	\end{equation}
	From the state-space realization, we may compute the transfer function as
	\begin{equation*}
		\hat H_{\mathcal{B}}(z) = \left[\begin{array}{c c}
			C_1(zI - A)^{-1}B_1 + D_{11} & zC_1(zI - A)^{-1}B_2 \\
			C_2(zI - A)^{-1}B_1/z + D_{21}/z & C_2(zI - A)^{-1}B_2 + D_{22} \\
		\end{array}\right] = \left[\begin{array}{c c}
			\hat H_{11}(z) & z\hat H_{12}(z) \\
			\hat H_{21}(z)/z & \hat H_{22}(z) \\
		\end{array}\right].
	\end{equation*}
	Finally, note two algorithms are equivalent if and only if they have identical transfer functions by \cref{prop1}.
\end{proof}

We have assumed that $D_{12} = 0$ for algorithm $\mathcal{A}$.
This assumption is quite weak.
In fact, $D_{12}$ must be $0$
for any algorithm $\mathcal{A}$ that can be represented as a causal linear time-invariant system.
Here, causal means that we can implement the algorithm by calling
state update equations sequentially.
To see this,
suppose the state update equations have been arranged in this order,
and use \cref{eqp3}
to write down the matrix representation of the infinite dimensional map $\bH$
that maps input $\bu$ to output $\by$
corresponding to $\mathcal{A}$ as \cref{eq18}:
\begin{equation}\label{eq18}
	\bH =
	\begin{bmatrix}
		D_{11} & D_{12} & 0 & 0 & 0 & 0 &\cdots \\
		D_{21} & D_{22} & 0 & 0 & 0 & 0 &\cdots \\
		C_1B_1 & C_1B_2 & D_{11} & D_{12} & 0 & 0 &\cdots \\
		C_2B_1 & C_2B_2 & D_{21} & D_{22} & 0 & 0 &\cdots \\
		C_1AB_1 & C_1AB_2& C_1B_1 & C_1B_2 & D_{11} & D_{12} & \cdots \\
		C_2AB_1 & C_2AB_2 & C_2B_1 & C_2B_2 & D_{21} & D_{22} & \cdots \\
		\vdots & \ddots & \ddots & \ddots & \ddots & \ddots & \ddots
	\end{bmatrix}.
\end{equation}
We can see that map $\bH$ is (block) Toeplitz.
Further, if algorithm $\mathcal{A}$ is causal,
map $\bH$ must be lower-triangular,
and so $D_{12}$ must be $0$.

By causality, at each iteration
the former oracle calls must be independent with the latter oracle calls while
the latter calls can depend on the former calls.
This indicates that there are no directed cycles in the directed graph
representing oracle calls at each iteration for any causal algorithm.
In other words, the graph is a directed acyclic graph (DAG).
This is consistent with the fact that
any causal algorithm has a lower-triangular $D$ matrix
(lower-triangular adjacency matrix of the directed graph).

Note that algorithms are not always written with
state update equations ordered causally: for example,
the state-space realization \cref{eq12} has a non-zero $D_{12}$ block.
However, we may reorder these equations so that each equation depends
only on previously-computed quantities to reveal
that the iteration is causal; after this rearrangement, the new $D_{12}$ block is 0.
We discuss permutations further in \cref{timedelay}.

The fixed points of an algorithm and its cyclic permutations are
the same up to a permutation, as stated by \cref{prop-shiftfixedpoint}.
\begin{proposition}\label{prop-shiftfixedpoint}
	If algorithm $\mathcal{A}$ converges to a fixed point $(\bar{y}_1^\star, \bar{y}_2^\star, \bar{u}_1^\star, \bar{u}_2^\star, x^\star)$,
	then its cyclic permutation $P_{\pi}\mathcal{A}$ converges to fixed point
	$(\bar{y}_2^\star, \bar{y}_1^\star, \bar{u}_2^\star, \bar{u}_1^\star, x^\star)$.
\end{proposition}
A detailed proof is provided in \cref{proof-shiftfixedpoint}.

\subsection{Applications: proof of shift equivalence}\label{app-shift}

\titleparagraph{\Cref{algo5,algo6}}
Now, we can revisit \cref{algo5,algo6} in the third motivating example
and show that they are shift-equivalent.
Here the oracles of \cref{algo5,algo6} are $\text{prox}_f$ and $\text{prox}_g$.
The transfer function of \cref{algo5} is
\begin{displaymath}
	\hat H_5(z) = \left[\begin{array}{ c c c | c c }
		0 & 0 & 0 & 1 & 0 \\
		0 & 0 & 0 & 0 & 1 \\
		0 & 0 & 1 & -1 & 1 \\
		\hline
		0 & 0 & 1 & 0 & 0 \\
		0 & 0 & -1 & 2 & 0 \\
	\end{array}\right] = \left[\begin{array}{ c c }
		-\frac{1}{z-1} & \frac{1}{z-1}  \\
		\frac{2z-1}{z-1} & -\frac{1}{z-1}  \\
	\end{array}\right].
\end{displaymath}
The 
transfer functions of \cref{algo6} is
\begin{displaymath}
	\hat H_6(z) = \left[\begin{array}{ c c | c c }
		1 & -1 & 0 & 1 \\
		0 & 0 & 1 & 0  \\
		\hline
		1 & -1 & 0 & 1  \\
		-1 & 2 & 0 & 0  \\
	\end{array}\right] = \left[\begin{array}{ c c }
		-\frac{1}{z-1} & \frac{z}{z-1}  \\
		\frac{2z-1}{z(z-1)} & -\frac{1}{z-1}  \\
	\end{array}\right].
\end{displaymath}
From \cref{prop3,prop4},
we know that they are shift-equivalent and equivalent up to a cyclic permutation.

\titleparagraph{\Cref{algo5_1,algo5_2}}
Here we revisit \cref{algo5_1,algo5_2}
at the beginning of this chapter
and show their relations with \cref{algo5}.
The oracles are $\text{prox}_f$ and $\text{prox}_g$.
The transfer function of \cref{algo5_1} is
\begin{displaymath}
	\hat H_8(z) = \left[\begin{array}{ c c c | c c }
		0 & 0 & 0 & 1 & 0 \\
		0 & 0 & 0 & 0 & 1 \\
		-1 & 0 & 1 & 0 & 1 \\
		\hline
		-1 & 0 & 1 & 0 & 1 \\
		2 & 0 & -1 & 0 & 0 \\
	\end{array}\right] = \left[\begin{array}{ c c }
		-\frac{1}{z-1} & \frac{z}{z-1}  \\
		\frac{2z-1}{z(z-1)} & -\frac{1}{z-1}  \\
	\end{array}\right].
\end{displaymath}
The transfer function of \cref{algo5_2} is
\begin{displaymath}
	\hat H_9(z) = \left[\begin{array}{ c c c | c c }
		0 & 0 & 0 & 1 & 0 \\
		0 & 0 & 0 & 0 & 1 \\
		-1 & 1 & 1 & 0 & 0 \\
		\hline
		-1 & 1 & 1 & 0 & 0 \\
		1 & -1 & -1 & 2 & 0 \\
	\end{array}\right] = \left[\begin{array}{ c c }
		-\frac{1}{z-1} & \frac{1}{z-1}  \\
		\frac{2z-1}{z-1} & -\frac{1}{z-1}  \\
	\end{array}\right].
\end{displaymath}
From \cref{prop3,prop4},
we know that \cref{algo5,algo5_1} are shift-equivalent and equivalent up to a cyclic permutation.
From \cref{prop1}, we know \cref{algo5,algo5_2}
are oracle-equivalent, thus they are also shift-equivalent.

\vspace{-1em}
\noindent
\hfil
\begin{minipage}[t]{0.4\textwidth}
	\begin{algorithm}[H]
		\centering
		\caption{Douglas-Rachford splitting}
		\label{algo7}
		\begin{algorithmic}
			\FOR{$k=0, 1, 2,\ldots$}
			\STATE{$x^{k+1}_1 = \textnormal{prox}_{tf}(x^k_3)$}
			\STATE{$x^{k+1}_2 = \textnormal{prox}_{t(g \circ L)}(2x^{k+1}_1 - x^k_3)$}
			\STATE{$x^{k+1}_3 = x^k_3 + x^{k+1}_2 - x^{k+1}_1$}
			\ENDFOR
		\end{algorithmic}
	\end{algorithm}
\end{minipage}
\hfil
\begin{minipage}[t]{0.54\textwidth}
	\begin{algorithm}[H]
		\centering
		\caption{ADMM}
		\label{algo8}
		\begin{algorithmic}
			\FOR{$k=0, 1, 2,\ldots$}
			\STATE{$\xi^{k+1}_1 = \textnormal{argmin}_{\xi}\{g(\xi)+ \frac{\rho}{2}
				\left \| A\xi + B\xi^{k}_2 -c + \xi^{k}_3 \right \|^2  \} $}
			\STATE{$\xi^{k+1}_2 = \textnormal{argmin}_{\xi}\{f(\xi)+ \frac{\rho}{2}
				\left \| A\xi^{k+1}_1 + B\xi -c + \xi^{k}_3 \right \|^2  \} $}
			\STATE{$\xi^{k+1}_3 = \xi^{k}_3 + A\xi^{k+1}_1 + B\xi^{k+1}_2 - c$}
			\ENDFOR
		\end{algorithmic}
	\end{algorithm}
\end{minipage}
\hfil
\vspace{1em}

\titleparagraph{Douglas-Rachford splitting and ADMM}
Consider a last example of algorithm permutation: Douglas-Rachford splitting (DR)
(\cref{algo7}~\cite{douglas1956numerical, eckstein1992douglas})
and the alternating direction method of multipliers (ADMM) (\cref{algo8}~\cite[\S8]{ryuyinconvex}).
Suppose that linear operator $L$ is invertible, $A = L^{-1}$, $B = -I$, and $c = 0$ in \cref{eqp1}.
Then both DR and ADMM solve problem \cref{eqp1}~\cite{MAL-016, wen2010alternating, lions1979splitting},
and the update equations of ADMM can be simplified as \cref{algo8_1}.
\begin{algorithm}
	\centering
	\caption{Simplified ADMM}
	\label{algo8_1}
	\begin{algorithmic}
		\FOR{$k=0, 1, 2,\ldots$}
		\STATE{$\xi^{k+1}_1 = L\textnormal{prox}_{\frac{1}{\rho}(g \circ L)}(\xi^k_2 - \xi^k_3)$}
		\STATE{$\xi^{k+1}_2 = \textnormal{prox}_{\frac{1}{\rho}f}(L^{-1}\xi^{k+1}_1 + \xi^k_3)$}
		\STATE{$\xi^{k+1}_3 = \xi^{k}_3 + L^{-1}\xi^{k+1}_1  - \xi^{k+1}_2$}
		\ENDFOR
	\end{algorithmic}
\end{algorithm}
Further, we assume $\rho = 1/t$ in ADMM.
We will compute the transfer function of both algorithms
using $\textnormal{prox}_{tf}$
and $\textnormal{prox}_{t(g \circ L)}$ as the oracles.
The transfer function of DR is
\begin{equation}\label{eq19}
	\hat H_{10}(z) = \left[\begin{array}{ c c c | c c }
		0 & 0 & 0 & I & 0 \\
		0 & 0 & 0 & 0 & I \\
		0 & 0 & I & -I & I \\
		\hline
		0 & 0 & I & 0 & 0 \\
		0 & 0 & -I & 2I & 0 \\
	\end{array}\right] = \left[\begin{array}{ c c }
		-\frac{1}{z-1}I & \frac{1}{z-1}I  \\
		\frac{2z-1}{z-1}I & -\frac{1}{z-1}I  \\
	\end{array}\right]
\end{equation}
and the transfer function of ADMM is
\begin{equation}\label{eq20}
	\hat H_{11}(z) = \left[\begin{array}{ c c c | c c }
		0 & 0 & 0 & 0 & L \\
		0 & 0 & 0 & I & 0 \\
		0 & 0 & I & -I & I \\
		\hline
		0 & 0 & I & 0 & I \\
		0 & I & -I & 0 & 0 \\
	\end{array}\right] =  \left[\begin{array}{ c c }
		-\frac{1}{z-1}I & \frac{z}{z-1}I  \\
		\frac{2z-1}{z(z-1)}I & -\frac{1}{z-1}I  \\
	\end{array}\right].
\end{equation}
From \cref{prop3,prop4},
we know that DR and ADMM (with $\rho = 1/t$) are shift-equivalent
and that DR is equivalent to a cyclic permutation of ADMM.
In fact, it is also possible to write the state-space realization for each algorithm
using the gradient (or subgradient)
of $f$ and $g$ as the oracle.
The transfer functions depend on the choice of oracle,
but in either case, we obtain the same results:
the algorithms are shift-equivalent.
We discuss the details further in \cref{dradmmsubgradient}.
We can write the state-space realizations of DR and ADMM using the (sub)gradients as oracles
in \cref{dradmmsubgradient}: the corresponding $D_{12}$ blocks are still zero and thus still
satisfy causality.

\section{Algorithm repetition}\label{repe}
In previous sections,
we have defined equivalence between algorithms with the same number of oracle calls in each iteration.
This section considers how to identify relations between two algorithms
when the number of oracles in each iteration differs.
For example, we would like to detect when one algorithm consists of another, simpler algorithm,
repeated twice or more, possibly with changes to variables or shifts that obscure the relation.

Consider an algorithm $\mathcal{A}$. Given a problem and an initialization,
the algorithm will generate state sequence $(x^k_\mathcal{A})_{k\ge 0}$, input sequence
$(u^k_\mathcal{A})_{k\ge 0}$, and output sequence $(y^k_\mathcal{A})_{k\ge 0}$, respectively.
Specifically, the update at time step $k$ can be written
as $x^{k+1}_\mathcal{A} = \mathcal{A}(x^k_\mathcal{A})$.
Suppose we have another algorithm $\mathcal{B}$ such that $\mathcal{B} = \mathcal A^2$:
repeating $\mathcal A$ twice gives the same result as $\mathcal B$.
We call $\mathcal B$ a repetition of $\mathcal A$.

Just as in the previous sections,
algorithm repetition can be characterized by the transfer function.
Here, \cref{assump:compare} ensures the algorithms
compared call the same set of oracles,
although the number of times each oracle is called may be different.

\begin{proposition}\label{prop6}
	Suppose $\mathcal{A}$ has state-space realization $(A, B, C, D)$.
	Then $\mathcal{B}$ is equivalent to $\mathcal{A}^2$ if and only if its transfer function has the form
	\begin{equation}\label{eq21}
		\left[\begin{array}{c c}
			C(zI - A^2)^{-1}AB +D & C(zI - A^2)^{-1}B \\
			CA(zI - A^2)^{-1}AB+CB & CA(zI - A^2)^{-1}B+D
		\end{array}\right].
	\end{equation}
\end{proposition}

Detailed proof of \cref{prop6} is provided in \cref{proof-repetition}.

\renewcommand{\thealgorithm}{8.\arabic{algorithm}}
\setcounter{algorithm}{0}

\vspace{-1em}
\noindent
\hfil
\begin{minipage}[t]{0.46\textwidth}
	\begin{algorithm}[H]
		\centering
		\caption{Gradient method}
		\label{algo9}
		\begin{algorithmic}
			\FOR{$k=0, 1, 2,\ldots$}
			\STATE{$x^{k+1} = x^k - t\nabla f(x^k)$}
			\ENDFOR
		\end{algorithmic}
	\end{algorithm}
\end{minipage}
\hfil
\begin{minipage}[t]{0.46\textwidth}
	\begin{algorithm}[H]
		\centering
		\caption{Repetition of gradient method}
		\label{algo10}
		\begin{algorithmic}
			\FOR{$k=0, 1, 2,\ldots$}
			\STATE{$\xi^{k+1}_2 = \xi^k_1 - t\nabla f(\xi^k_1)$}
			\STATE{$\xi^{k+1}_1 = \xi^{k+1}_2 - t\nabla f(\xi^{k+1}_2)$}
			\ENDFOR
		\end{algorithmic}
	\end{algorithm}
\end{minipage}
\hfil
\vspace{1em}

One example of repetition consists the gradient method \cref{algo9} and its repetition \cref{algo10}.
Both call the same set of oracles ($\nabla f$).
The transfer functions of each algorithm are computed as $\hat H_{12}(z)$ and $\hat H_{13}(z)$ respectively:

\[
\hat H_{12}(z) = \left[\begin{array}{ c | c }
	1 &-t \\
	\hline
	1 & 0  \\
\end{array}\right] = - \frac{t}{z-1},
\qquad
\hat H_{13}(z) = \left[\begin{array}{ c | c c }
	1 & -t & -t \\
	\hline
	1 & 0 & 0   \\
	1 & -t & 0  \\
\end{array}\right] = \left[\begin{array}{ c c }
	- \frac{t}{z-1} & - \frac{t}{z-1}  \\
	- \frac{tz}{z-1} & - \frac{t}{z-1}  \\
\end{array}\right].
\]

\Cref{prop6} reveals how the transfer function changes when an algorithm is
repeated twice.
In fact, we can identify an algorithm that has been repeated arbitrarily many times.
Suppose algorithm $\mathcal{C}$ is $\mathcal{A}$ repeated $n \geq 1$ times:
$\mathcal{C} = \mathcal{A}^n$.

\begin{proposition}\label{prop7}
	Suppose $\mathcal{A}$ has state-space realization $(A, B, C, D)$.
	Then $\mathcal{C}$ is equivalent to $\mathcal{A}^n$ for $n\geq 1$
	if and only if $\mathcal{C}$ has a transfer function given by \cref{eq26}.
\end{proposition}

\begin{proof}
	Sufficiency. We can represent $\mathcal{C}$ with state-space realization
	\begin{equation}\label{eq25}
		\left[\begin{array}{c :c c c c c}
			A^n & A^{n-1}B & \dots & \dots & AB &B  \\
			\hdashline
			C & D & 0 & 0 & \dots & 0  \\
			CA & CB & D & 0 & \dots & 0 \\
			\vdots & \vdots & \ddots & \ddots & \ddots & \vdots \\
			CA^{n-1} & CA^{n-2}B  &\dots & \dots & CB & D  \\
		\end{array}\right].
	\end{equation}
	Note that $(zI - A^n)^{-1}A^l = A^l(zI - A^n)^{-1}$ for any $n$ and $l$.
	Let $ \tilde{C} = C(zI - A^n)^{-1}$,
	and compute the transfer function of $\mathcal C$:
	\begin{equation}\label{eq26}
		\left[\begin{array}{c c c c c c}
			\tilde{C}A^{n-1}B + D & \tilde{C}A^{n-2}B & \dots & \dots & \tilde{C}AB & \tilde{C}B  \\
			\tilde{C}A^{n}B + CB & \tilde{C}A^{n-1}B + D & \dots & \dots & \tilde{C}A^2B & \tilde{C}AB  \\
			\vdots & \vdots & \ddots & \ddots & \vdots & \vdots \\
			\tilde{C}A^{2n-2}B + CA^{n-2}B & \tilde{C}A^{2n-3}B + CA^{n-3}B & \dots & \dots &
			\tilde{C}A^nB + CB & \tilde{C}A^{n-1}B + D  \\
		\end{array}\right].
	\end{equation}
	
	Necessity is provided by \cref{prop1} since the transfer function uniquely
	characterizes an equivalence class of algorithms.
\end{proof}

\titleparagraph{Remark}
\Cref{prop6} is a special case of \cref{prop7} when $n = 2$.
The dimension of transfer function of $\mathcal{C}$ is $n$ times
the dimension of transfer function of $\mathcal{A}$.
Similarly, the dimension of input and output of $\mathcal{C}$ is $n$ times
the dimension of the input and output of $\mathcal{A}$.
At time step $k$, we have $y^k_\mathcal{C} =
(y^{nk}_\mathcal{A}, \dots, y^{(n+1)k-1}_\mathcal{A})$ and
$u^k_\mathcal{C} = (u^{nk}_\mathcal{A}, \dots, u^{(n+1)k-1}_\mathcal{A})$.

Just as for oracle equivalence and cyclic permutations,
the fixed points of an algorithm and its repetitions are related,
as shown in \cref{prop13}.

\begin{proposition}\label{prop13}
	If algorithm $\mathcal{A}$ converges to a fixed point $(y^\star, u^\star, x^\star)$,
	then its repetition $\mathcal{A}^n$ for $n\geq 1$ converges to fixed point $(y', u', x^\star)$,
	with $y' = y^\star \bigotimes \mathbbm{1}^n$ and $u' = u^\star \bigotimes \mathbbm{1}^n$.
	Here $\bigotimes$ is the Kronecker product and
	$\mathbbm{1}^n$ is an $n$ dimensional vector whose entries are all ones.
\end{proposition}
Detailed proof is provided in \cref{proof-repetitionfixedpoint}.
Since $\mathcal{A}^n$ repeats $\mathcal{A}$ $n$ times,
the input and output of the fixed point of $\mathcal{A}^n$
are obtained by repeating the input and output
on the corresponding fixed point of $\mathcal{A}$ $n$ times.

Repetition gives us many more ways to combine algorithms into complex and
unwieldly (but convergent) new methods.
We can repeat a sequence of iterations from different algorithms
and regard them together as a new algorithm.
Suppose we choose $n$ algorithms $\mathcal{A}_1, \dots, \mathcal{A}_n$
with state-space realizations $(A_1, B_1, C_1, D_1), \dots, (A_n, B_n, C_n, D_n)$
and run one iteration of each as a single iteration of our new monster algorithm.
For simplicity, suppose the state-space realization matrices $A_i, B_i, C_i, D_i$
for each algorithm $\mathcal A_i$
have the same dimensions as all others $i=1,\ldots,n$.
(Otherwise the result is harder to write down, but still straightforward to compute.)
Then we can represent the resulting monster algorithm with transfer function
\begin{equation}\label{eq27}
	\left[\begin{array}{c | c c c c c}
		\prod_{i = n}^{1}A_i &\prod_{i = n}^{2}A_iB_{1} & \dots & \dots & A_nB_{n-1} & B_n  \\
		\hline
		C_1 & D_1 & 0 & 0 & \dots & 0  \\
		C_2A_1 & C_2B_1 & D_2 & 0 & \dots & 0 \\
		\vdots & \vdots & \ddots & \ddots & \ddots & \vdots \\
		C_n\prod_{i = n-1}^{1}A_i & C_n\prod_{i = n-1}^{2}A_iB_{1} &\dots & \dots & C_nB_{n-1} & D_n  \\
	\end{array}\right].
\end{equation}

Hence one way to develop a new optimization algorithm
would be to combine existing algorithms into a new monster algorithm
with similar convergence properties but (perhaps) new exciting interpretations.
For example, we could combine gradient descent with the proximal point method
to derive a proximal gradient method for minimizing $f(x)$: $\prox_f(x - \nabla f(x))$.
(We are not aware of any published optimization algorithms that have been constructed in this way.)

Using our software, it would be easy to detect such algorithm surgery
by searching over all pairs (or trios, etc) of known algorithms.
This combinatorial search is still not too expensive,
since the list of known algorithms is still rather small,
and the number of algorithms that makes up a monster algorithm is limited by the
number of oracle calls at each iteration of the monster algorithm.

\section{Algorithm conjugation}\label{conjugation}
In this section, we introduce one last algorithm transformation,
conjugation, which alters the oracle calls
but results in algorithms that still bear a family resemblance.

In convex optimization,
algorithm conjugation naturally relates some oracles to others \cites{ryu2016primer}[\S2]{ryuyinconvex}:
for example, when $f^*(y) = \sup_x \{x^Ty - f(x)\}$ is the Fenchel conjugate of $f$ \cite[\S3]{fenchel1953convex},
\begin{itemize}
	\item $(\partial f)^{-1} = \partial f^*$, and \mnote{What's the name for this equality?}
	\item \emph{Moreau's identity.} $I - \prox_f = \prox_{f^*}$.
\end{itemize}
We can rewrite any algorithm in terms of different, also easily computable,
oracles using these identities.
Consider a simple example: we will obfuscate
the proximal gradient method (\cref{algo11} \cites[\S10]{doi:10.1137/1.9781611974997}{doi:10.1137/080716542})
by rewriting it in terms of the conjugate of the original oracle $\prox_g$,
using Moreau's identity, as \cref{algo12}~\cite{moreau:hal-01867187}.

\renewcommand{\thealgorithm}{9.\arabic{algorithm}}
\setcounter{algorithm}{0}

\vspace{-1em}
\noindent
\hfil
\begin{minipage}[t]{0.42\textwidth}
	\begin{algorithm}[H]
		\centering
		\caption{Proximal gradient method}
		\label{algo11}
		\begin{algorithmic}
			\FOR{$k=0, 1, 2,\ldots$}
			\STATE{$x^{k+1} = \textnormal{prox}_{tg}(x^k - t\nabla f(x^k))$}
			\ENDFOR
		\end{algorithmic}
	\end{algorithm}
\end{minipage}
\hfil
\begin{minipage}[t]{0.5\textwidth}
	\begin{algorithm}[H]
		\centering
		\caption{Conjugate of proximal gradient method}
		\label{algo12}
		\begin{algorithmic}
			\FOR{$k=0, 1, 2,\ldots$}
			\STATE{$\xi^{k+1} = \xi^k - t\nabla f(\xi^k) -
				t\textnormal{prox}_{\frac{1}{t}g^*}(\frac{1}{t}(\xi^k - t\nabla f(\xi^k)))$}
			\ENDFOR
		\end{algorithmic}
	\end{algorithm}
\end{minipage}
\hfil
\vspace{1em}

The transfer function of the algorithm changes when we rewrite the algorithm to call a different oracle,
such as calling $\prox_{f^*}$ instead of $\prox_f$.
Yet the sequence of states is preserved!
Similarly, when we rewrite an algorithm to call $\partial f^*$ instead of $\partial f$,
the resulting algorithm is related to the original algorithm by swapping the input and output sequences.
\mnote{Define enough first that this statement is clear!}
We say that algorithm $\mathcal B = \C_\kappa \mathcal A$ is a conjugate of algorithm $\mathcal A$ if
algorithm $\mathcal B$ results from rewriting algorithm $\mathcal A$
to use the conjugates of the oracles in set $\kappa \subseteq [n]$,
where $[n] = \{1, \ldots, n\}$ is the set of oracle indices for algorithm $\mathcal A$.
Interestingly, conjugation preserves the state sequence but not the oracle sequence.
We will also call two algorithms conjugates if they are oracle-equivalent to a conjugate pair.
Our goal in this section is to describe how to identify conjugate algorithms.

For simplicity in the remainder of this section, we suppose that all oracles are
(sub)gradients. To detect equivalence of algorithms involving prox using methods presented here,
we may write
the state-space realization of the algorithm in terms of (sub)gradients:
\[
u = \prox_f(y) \quad \iff \quad y \in u + \partial f(u).
\]
In fact, our software uses this method to check algorithm conjugation.

Restricting to (sub)gradients, we see from the identity $(\partial f)^{-1} = \partial f^*$
that algorithm conjugation swaps the input and output of an algorithm:
the algorithm after conjugation
takes the output of the original algorithm as input and
produces the input of the original one as output.
As shown in \cref{fig10}, the input sequence of the algorithm after conjugation
is the original output sequence and the output sequence in the algorithm after conjugation
is the original input sequence.

\begin{figure}[tbhp]
	\centering
	\begin{tikzpicture}[>=latex]
		
		\node[scale = 0.75][box6] at (0,1.9) (algo1) {$L$};
		\node[scale = 0.75] at (0,1.7) (ref2) {};
		\node[scale = 0.75] at (0,2.1) (ref1) {};
		\node[scale = 0.75][box3, left of = ref1 , node distance = 8em] (state0) {$x^{k-1}$};
		\node[scale = 0.75][left of = ref1, node distance = 13em] (init1) {\ldots};
		\node[scale = 0.75][box3, right of = ref1, node distance = 8em] (state1) {$x^{k}$};
		\node[scale = 0.75][box] at (0, 1.1) (oracle1) {$\phi$};
		\node[scale = 0.75][box3, right of = oracle1, node distance = 6em] (output1) {$y^{k-1}$};
		\node[scale = 0.75, left of= oracle1, node distance = 6em][box3] (input1) {$u^{k-1}$};
		\node[scale = 0.75][box6, right of = algo1, node distance = 16em] (algo2) {$L$};
		
		\draw[->]  (init1) -- (state0);
		\draw[->]  (state0) -- (ref1-|algo1.west);
		\draw[->]  (state1-|algo1.east) -- (state1);
		\draw[->]   (input1) |- (ref2-|algo1.west);
		\draw[->]  (output1) -- (oracle1);
		\draw[->]  (oracle1) -- (input1);
		\draw[->]  (ref2-|algo1.east) -| (output1);
		\draw[->]  (state1) -- (state1-|algo2.west);
		
		\node[scale = 0.75][box3, right of = state1, node distance = 16em] (state2) {$x^{k+1}$};
		\node[scale = 0.75][box, right of = oracle1, node distance = 16em] (oracle2) {$\phi$};
		\node[scale = 0.75][box3, right of = oracle2, node distance = 6em] (output2) {$y^{k}$};
		\node[scale = 0.75, left of= oracle2, node distance = 6em][box3] (input2) {$u^{k}$};
		\node[scale = 0.75][box6, right of = algo2, node distance = 16em] (algo3) {$L$};
		
		\draw[->]  (state1-|algo2.east) -- (state2);
		\draw[->]  (input2) |- (ref2-|algo2.west);
		\draw[->]  (output2) -- (oracle2);
		\draw[->]  (oracle2) -- (input2);
		\draw[->]  (ref2-|algo2.east) -| (output2);
		\draw[->]  (state2) -- (state1-|algo3.west);
		
		\node[scale = 0.75][right of = state2 , node distance = 16em] (end1) {\ldots};
		\node[scale = 0.75][box, right of = oracle2, node distance = 16em] (oracle3) {$\phi$};
		\node[scale = 0.75][box3, right of = oracle3, node distance = 6em] (output3) {$y^{k+1}$};
		\node[scale = 0.75, left of= oracle3, node distance = 6em][box3] (input3) {$u^{k+1}$};
		
		\draw[->]  (state1-|algo3.east) -- (end1);
		\draw[->]  (input3) |- (ref2-|algo3.west);
		\draw[->]  (output3) -- (oracle3);
		\draw[->]  (oracle3) -- (input3);
		\draw[->]  (ref2-|algo3.east) -| (output3);
		
		\node[scale = 0.75][box6] at (0,-0.8) (algo10) {$\tilde{L}$};
		\node[scale = 0.75] at (0,-1) (ref10) {};
		\node[scale = 0.75] at (0,-0.6) (ref20) {};
		\node[scale = 0.75][box3, left of = ref10 , node distance = 8em] (state00) {$\tilde{x}^{k-1}$};
		\node[scale = 0.75][left of = ref10, node distance = 13em] (init10) {\ldots};
		\node[scale = 0.75][box3, right of = ref10, node distance = 8em] (state10) {$\tilde{x}^{k}$};
		\node[scale = 0.75][box] at (0, 0) (oracle10) {$\phi^{-1}$};
		\node[scale = 0.75][box3, right of = oracle10, node distance = 6em] (output10) {$\tilde{y}^{k-1}$};
		\node[scale = 0.75, left of= oracle10, node distance = 6em][box3] (input10) {$\tilde{u}^{k-1}$};
		\node[scale = 0.75][box6, right of = algo10, node distance = 16em] (algo20) {$\tilde{L}$};
		
		\draw[->]  (init10) -- (state00);
		\draw[->]  (state00) -- (ref10-|algo10.west);
		\draw[->]  (state10-|algo10.east) -- (state10);
		\draw[->]   (input10) |- (ref20-|algo10.west);
		\draw[->]  (output10) -- (oracle10);
		\draw[->]  (oracle10) -- (input10);
		\draw[->]  (ref20-|algo10.east) -| (output10);
		\draw[->]  (state10) -- (state10-|algo20.west);
		
		\node[scale = 0.75][box3, right of = state10 , node distance = 16em] (state20) {$\tilde{x}^{k+1}$};
		\node[scale = 0.75][box, right of = oracle10, node distance = 16em] (oracle20) {$\phi^{-1}$};
		\node[scale = 0.75][box3, right of = oracle20, node distance = 6em] (output20) {$\tilde{y}^{k}$};
		\node[scale = 0.75, left of= oracle20, node distance = 6em][box3] (input20) {$\tilde{u}^{k}$};
		\node[scale = 0.75][box6, right of = algo20, node distance = 16em] (algo30) {$\tilde{L}$};
		
		\draw[->]  (state10-|algo20.east) -- (state20);
		\draw[->]  (input20) |- (ref20-|algo20.west);
		\draw[->]  (output20) -- (oracle20);
		\draw[->]  (oracle20) -- (input20);
		\draw[->]  (ref20-|algo20.east) -| (output20);
		\draw[->]  (state20) -- (state10-|algo30.west);
		
		\node[scale = 0.75][right of = state20, node distance = 16em] (end10) {\ldots};
		\node[scale = 0.75][box, right of = oracle20, node distance = 16em] (oracle30) {$\phi^{-1}$};
		\node[scale = 0.75][box3, right of = oracle30, node distance = 6em] (output30) {$\tilde{y}^{k+1}$};
		\node[scale = 0.75, left of= oracle30, node distance = 6em][box3] (input30) {$\tilde{u}^{k+1}$};
		
		\draw[->]  (state10-|algo30.east) -- (end10);
		\draw[->]  (input30) |- (ref20-|algo30.west);
		\draw[->]  (output30) -- (oracle30);
		\draw[->]  (oracle30) -- (input30);
		\draw[->]  (ref20-|algo30.east) -| (output30);
		
		
		\draw[-{Straight Barb[left]}] ($(input1.south) + (0.02,0)$) -- ($(output10.north) + (0.02,0.02)$);
		\draw[-{Straight Barb[left]}] ($(output10.north) + (-0.02,0)$) -- ($(input1.south) + (-0.02,-0.02)$);
		
		\draw[-{Straight Barb[left]}] ($(input2.south) + (0.02,0)$) -- ($(output20.north) + (0.02,0.02)$);
		\draw[-{Straight Barb[left]}] ($(output20.north) + (-0.02,0)$) -- ($(input2.south) + (-0.02,-0.02)$);
		
		\draw[-{Straight Barb[left]}] ($(input3.south) + (0.02,0)$) -- ($(output30.north) + (0.02,0.02)$);
		\draw[-{Straight Barb[left]}] ($(output30.north) + (-0.02,0)$) -- ($(input3.south) + (-0.02,-0.02)$);
		
		\draw[-{Straight Barb[right]}] ($(output1.south) + (-0.02,0)$) -- ($(input10.north) + (-0.02,0.02)$);
		\draw[-{Straight Barb[right]}] ($(input10.north) + (0.02,0)$) -- ($(output1.south) + (0.02,-0.02)$);
		
		\draw[-{Straight Barb[right]}] ($(output2.south) + (-0.02,0)$) -- ($(input20.north) + (-0.02,0.02)$);
		\draw[-{Straight Barb[right]}] ($(input20.north) + (0.02,0)$) -- ($(output2.south) + (0.02,-0.02)$);
		
		\draw[-{Straight Barb[right]}] ($(output3.south) + (-0.02,0)$) -- ($(input30.north) + (-0.02,0.02)$);
		\draw[-{Straight Barb[right]}] ($(input30.north) + (0.02,0)$) -- ($(output3.south) + (0.02,-0.02)$);
	\end{tikzpicture}
	\caption{Unrolled block-diagram representation of algorithm conjugation.}
	\label{fig10}
\end{figure}
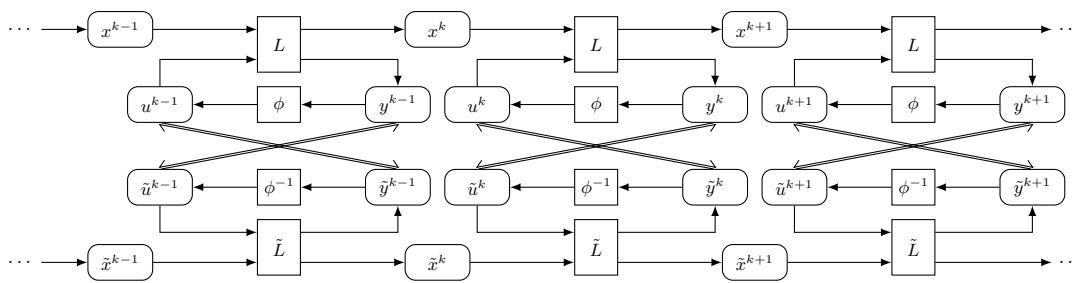

First, let's introduce a bit of standard notation.
Suppose an algorithm $\mathcal{A}$ contains $n$ oracle calls in each iteration.
The cardinality of a subset $\kappa \subseteq [n]$ is
$ \left | \kappa \right |$ and the complement is $\bar{\kappa}=[n] \setminus \kappa$.
For any matrix $M\in \R^{n\times n}$,
$M[\kappa, \nu]$ is the sub-matrix of $M$ whose rows and columns
are indexed by $\kappa$ and $\nu \subseteq [n]$, respectively.
We write $M[\kappa, \kappa]$ as $M[\kappa]$ for simplicity.
For $i \in [n]$, the conjugation operator $\C_{i}$ conjugates oracle $i$:
it replaces the $i$th oracle by its inverse.
The operator $\C_{\kappa}$ conjugates all oracles in the set $\kappa \subseteq [n]$
to produce the conjugate algorithm $\C_{\kappa}\mathcal{A}$.

\begin{proposition}\label{prop8}
	Suppose $\mathcal{A}$ has state-space realization $(A, B, C, D)$ and transfer function $\hat H(z)$,
	and $D[\kappa]$ is invertible.
	Then $\mathcal{B}$ is equivalent to $\C_{\kappa}\mathcal{A}$ if and
	only if the transfer function $\hat H'(z)$ of $\mathcal{B}$ satisfies
	\begin{equation}\label{eq28}
		P{\hat H'(z)}P^{T}=
		\begin{bmatrix}
			\hat H[\kappa]^{-1}(z) & -\hat H[\kappa]^{-1}(z)\hat H[\kappa, \bar{\kappa}](z)\\
			\hat H[\bar{\kappa}, \kappa](z)\hat H[\kappa]^{-1}(z) &
			\hat H[\bar{\kappa}](z)-\hat H[\bar{\kappa}, \kappa](z)\hat H[\kappa]^{-1}(z)\hat H[\kappa, \bar{\kappa}](z)
		\end{bmatrix}.
	\end{equation}
	Here $P$ is a permutation matrix that swaps rows and columns so indices in $\kappa$ come first:
	\begin{equation}\label{eq29}
		P\hat{H}(z)P^{T}=
		\begin{bmatrix}
			\hat H[\kappa](z) & \hat H[\kappa, \bar{\kappa}](z)\\
			\hat H[\bar{\kappa}, \kappa](z) & \hat H[\bar{\kappa}](z)
		\end{bmatrix}.
	\end{equation}
\end{proposition}
\begin{proof}
	Sufficiency. Without loss of generality, suppose the  oracles $\kappa = \{1,\ldots,|\kappa|\}$ appear first,
	\begin{displaymath}
		\hat H(z) = \begin{bmatrix}
			\hat H[\kappa](z) & \hat H[\kappa, \bar{\kappa}](z)\\
			\hat H[\bar{\kappa}, \kappa](z) & \hat H[\bar{\kappa}](z)
		\end{bmatrix}, \qquad D = \begin{bmatrix}
			D[\kappa] & D[\kappa, \bar{\kappa}]\\
			D[\bar{\kappa}, \kappa] & D[\bar{\kappa}] \end{bmatrix},
	\end{displaymath}
	and consequently the permutation matrix $P$ is the identity.
	We obtain the desired results from \cref{eqp13}
	by setting $D_{11} = D[\kappa]$, $\hat H_{11}(z) = \hat H[\kappa](z)$,
	$\hat H_{12}(z) = \hat H[\kappa, \bar{\kappa}](z)$, $\hat H_{21}(z) = \hat H[\bar{\kappa}, \kappa](z)$,
	and $\hat H_{22}(z) = \hat H[\bar{\kappa}](z)$.
	
	Necessity is provided by \cref{prop1} as the transfer function uniquely
	characterizes an equivalence class of algorithms.
	
\end{proof}

From \cref{prop8}, the transfer function $\hat H(z)$ of algorithm $\mathcal{A}$
is partially inverted when the algorithm is conjugated by $\C_{\kappa}$.
The new transfer function $\hat H'(z)$ results from applying the Sweep operator with indices $\kappa$ to $\hat H(z)$
\cite{10.2307/2683825, TSATSOMEROS2000151}.
If we consider the input and output sequences for each oracle separately,
for any oracle in $\kappa$, the input sequence corresponding to $\C_{\kappa}\mathcal{A}$ is
the original output sequence in $\mathcal{A}$
and the output sequence corresponding to $\C_{\kappa}\mathcal{A}$
is the original input sequence in $\mathcal{A}$.
The input and output sequences of oracles in $[n] \setminus \kappa$
remain unchanged in the new algorithm $\C_{\kappa}\mathcal{A}$.
Here, \cref{assump:compare} ensures the algorithms
compared call either same oracles or their corresponding conjugate oracles and
in each iteration the number of oracle calls are the same.

\Cref{prop8} assumes that $D[\kappa]$ is invertible.
In fact, 
$\C_{\kappa}\mathcal A$ is a causal algorithm
if and only if $D[\kappa]$ is invertible.
We need not condition on causality in the proposition,
since any algorithm that can be written down as a set of update equations
is necessarily causal.

Now we consider two special cases: conjugating
1) a single oracle, or
2) all of the oracles.

\begin{corollary}\label{coro1}
	Consider algorithm $\mathcal{A}$ with state-space realization $(A, B, C, D)$ and transfer
	function $\hat H(z) \in \R^{n\times n}$.
	\begin{enumerate}[(a)]
		\item Suppose $D_{kk} \neq 0$ for any $k \in [n]$. Then
		the new transfer function $\hat H'(z)$ of $\C_{k}\mathcal{A}$ can be expressed entrywise as
		\begin{equation}\label{eq33}
			h'_{ij}(z) =\begin{cases}
				1/h_{kk}(z) & i=k,~j=k
				\\
				-h_{kj}(z)/h_{kk}(z)& i=k,~j\neq k
				\\
				h_{ik}(z)/h_{kk}(z)& i\neq k,~j = k
				\\
				h_{ij}(z)-h_{ik}(z)h_{kj}(z)/h_{kk}(z)& i\neq k,~j\neq k,
			\end{cases}
		\end{equation}
		as $h_{ij}(z)$ and $h'_{ij}(z)$ $1\leq i,j \leq n$ denote the entries of $\hat H(z)$ and $\hat H'(z)$ respectively.
		
		\item	Suppose $D$ is invertible. Then the transfer
		function $\hat H'(z)$ of $\C_{[n]}\mathcal{A}$ satisfies $\hat H'(z) = \hat H^{-1}(z)$.
		
	\end{enumerate}
	
\end{corollary}

%

\paragraph{Proximal gradient}
Now we can revisit \cref{algo11,algo12} and
show that they are conjugate.
The transfer functions of \cref{algo11,algo12} are computed
as $\hat H_{14}(z)$ and $\hat H_{15}(z)$ below.
Note that the state-space realizations are written in terms of (sub)gradients.
From \cref{coro1}, they are conjugate
with respect to the second oracle.

\begin{displaymath}
	\hat H_{14}(z) = \left[\begin{array}{ c c }
		-\frac{t}{z-1} & -\frac{t}{z-1}  \\
		-\frac{tz}{z-1} & -\frac{tz}{z-1}  \\
	\end{array}\right], \qquad
	\hat H_{15}(z) = \left[\begin{array}{ c c }
		0 & \frac{1}{z}  \\
		-1 & -\frac{z-1}{tz}  \\
	\end{array}\right]
\end{displaymath}

\begin{algorithm}[H]
	\caption{Chambolle-Pock method}
	\centering
	\label{algo14}
	\begin{algorithmic}
		\FOR{$k=0, 1, 2,\ldots$}
		\STATE{$x^{k+1}_1 = \textnormal{prox}_{\tau f}(x^k_1 - \tau M^T x^k_2)$}
		\STATE{$x^{k+1}_2 = \textnormal{prox}_{\sigma g^*}(x^k_2 + \sigma M (2x^{k+1}_1 - x^k_1))$}
		\ENDFOR
	\end{algorithmic}
\end{algorithm}

\paragraph{DR and Chambolle-Pock}
Another important example is the relation between DR (\cref{algo7})
and the primal-dual optimization method proposed
by Chambolle and Pock (\cref{algo14}~\cites{chambolle2011first}{o2018equivalence}).
Note that \cref{algo7} has parameter $t$ and linear operator $L$,
and \cref{algo14} has parameters $\tau$ and $\sigma$ and linear operator $M$.
Let $M = L$ so that \cref{algo7,algo14} solve the same problem.
Further suppose that $M$ is invertible and $MM^T = \delta I$ for any $\delta > 0$.
By \cref{coro1}, we know that they are conjugate with respect to the second oracle
if $\tau = t$ and $\sigma = 1/(\delta t)$.
So DR and the Chambolle-Pock method
(when the parameter value $\tau = t$ and $\sigma = 1/(\delta t)$) are conjugate.
The transfer functions of \cref{algo7,algo14} are provided below as $\hat H_{10}(z)$ and $\hat H_{16}(z)$ respectively.
We will say more about how to discover the correct parameter restriction in \cref{package}.

\begin{displaymath}
	\hat H_{10}(z) = \begin{bmatrix}
		-\frac{tz}{z-1}I & -\frac{t}{z-1}L^T\\
		\frac{t(1-2z)}{z-1}L & -\frac{tz}{z-1}LL^T
	\end{bmatrix}, \quad
	\hat H_{16}(z) \xrightarrow[\sigma = \frac{1}{\delta t}, \tau = t]{M=L, LL^T = \delta I}
	\begin{bmatrix}
		\frac{t(1-z)}{z}I & \frac{1}{z}L^T(LL^T)^{-1}\\
		\frac{1-2z}{z}(LL^T)^{-1}L & \frac{1-z}{tz}(LL^T)^{-1}
	\end{bmatrix}
\end{displaymath}

In order to test equivalence, all algorithms must use the same set of oracles.
This requirement becomes tricky when algorithms are written in terms of
an argmin: what is the oracle?
To resolve this issue, we compute the state-space realization of every algorithm
in this section using the subgradient as the oracle.
All these subgradient oracles are associated with proximal operators,
and so they are unique-valued, even though subgradients are generally set-valued:
the input-output pairs match those returned by the proximal operator.
(These subgradient oracles are used for the analysis but need not be computed explicitly.)


The fixed points
of an algorithm and its conjugate are related as stated in \cref{prop14}.
\begin{proposition}\label{prop14}
	If an algorithm $\mathcal{A}$ converges to a fixed point
	$(y[\kappa]^\star, y[\bar{\kappa}]^\star, u[\kappa]^\star, u[\bar{\kappa}]^\star, x^\star)$,
	then its conjugate $\C_{\kappa}\mathcal A$ converges to fixed point
	$(u[\kappa]^\star, y[\bar{\kappa}]^\star, y[\kappa]^\star, u[\bar{\kappa}]^\star, x^\star)$.
\end{proposition}
For simplicity, detailed proof is provided in \cref{proof-conjugation}.
Intuitively, as we invert the input-output map of $u[\kappa]$ and $y[\kappa]$,
the corresponding parts in the fixed point are also inverted.

\begin{proposition}\label{prop10}
	Suppose algorithm $\mathcal{A}$ has state-space realization $(A, B, C, D)$,
	where $D_{ii} \neq 0$ and $D_{jj} \neq 0$.
	Then $\C_i\C_j\mathcal{A} = \C_j\C_i\mathcal{A} = \C_{\{ij\}}\mathcal{A}$.
\end{proposition}

\begin{proof}
	By \cref{coro1}, if $D_{ii} \neq 0$ and $D_{jj} \neq 0$,
	then $\C_i\mathcal A$ and $\C_j\mathcal A$ are causal.
	Note that entries above diagonal of $D$ are all zero because $\mathcal A$ is causal.
	Thus, $\det(D[\{ij\}]) = D_{ii}D_{jj} \neq 0$ and $\C_{\{ij\}}\mathcal A$ is causal.
	The commutative property of the Sweep operator gives the
	result $\C_i\C_j\mathcal A = \C_j\C_i\mathcal A = \C_{\{ij\}}\mathcal A$ \cite{10.2307/2683825,TSATSOMEROS2000151}.
\end{proof}

\Cref{prop10} states that conjugation of different oracles commutes.
This justifies our notation $\C_\kappa$ for set $\kappa$, as the order of the
oracles in $\kappa$ is irrelevant.
Further, conjugation and cyclic permutation also commute;
see \cref{prop11} and proof in \cref{proof-commutative}.

\paragraph{DR and ADMM}
We showed in \cref{app-shift} that the DR (\cref{algo7}) and ADMM (\cref{algo8})
are related by permutation with a certain choice of parameters.
Here, we show that they are related by permutation and conjugation (in either order, as they commute),
with a different choice of parameters:
$A = L^T, B = I, c = 0, \rho = t$ for ADMM.
Further suppose that linear operator $L$ is invertible.
The transfer function of this special parameterization of ADMM is shown as $\hat H_{17}(z)$.
Relations between DR and ADMM can be illustrated as follows.
Recall $\hat H_{10}(z)$ is the transfer function of DR.
Here we can observe that different choices of parameters of algorithms
can lead to different relations between algorithms.

\small
\begin{displaymath}
	\hat H_{17}(z) =  \begin{bmatrix}
		-\frac{z}{t(z-1)}I & \frac{z}{t(z-1)}L^{-1} \\
		\frac{2z-1}{tz(z-1)}L^{-T}& -\frac{z}{t(z-1)}(LL^T)^{-1}
	\end{bmatrix} \xrightarrow{\C_{12}}
	\begin{bmatrix}
		-\frac{tz}{z-1}I & -\frac{tz}{z-1}L^T\\
		\frac{t(1-2z)}{z(z-1)}L & -\frac{tz}{z-1}LL^T
	\end{bmatrix} \xrightarrow{ P_{21}}
	\begin{bmatrix}
		-\frac{tz}{z-1}I & -\frac{t}{z-1}L^T\\
		\frac{t(1-2z)}{z-1}L & -\frac{tz}{z-1}LL^T
	\end{bmatrix}= \hat H_{10}(z)
\end{displaymath}
\normalsize

The commutative property is important to identify relations between algorithms efficiently.
For example, suppose we would like to identify the relations between \cref{algo7,algo8},
with transfer functions $\hat H_{10}(z)$ and $\hat H_{17}(z)$.
We can first perform conjugation and next permutation on \cref{algo7},
and then test equivalence between the resulting algorithm and \cref{algo8}.
We need not try permutation followed by conjugation;
as these commute, both orders lead to the same transfer function.

We have already shown several relations between
DR (\cref{algo7}),
ADMM (\cref{algo8}), and the Chambolle-Pock method (\cref{algo14})
using conjugation and permutation.
We represent these relations in \cref{fig11}.
The figure relates 8 different algorithms:
Starting from DR, since it contains 2 oracles,
there are 2 possible different algorithms by permutation.
From the state-space realization, we can conjugate both oracles,
which yields 4 different algorithms by conjugation of different oracles.
Therefore, in total there are 2 $\times$ 4 = 8 possible different algorithms,
including both ADMM and Chambolle-Pock.
In the figure, $\C_1$ and $\C_2$ denote conjugation with respect to the first and second oracles respectively,
$P$ denotes permutation,
and we can move between algorithms by applying the transformation on each edge,
in either direction,
as each transformation is an involution.

\begin{figure}[tbhp]
	\centering
	\includegraphics[scale=0.55]{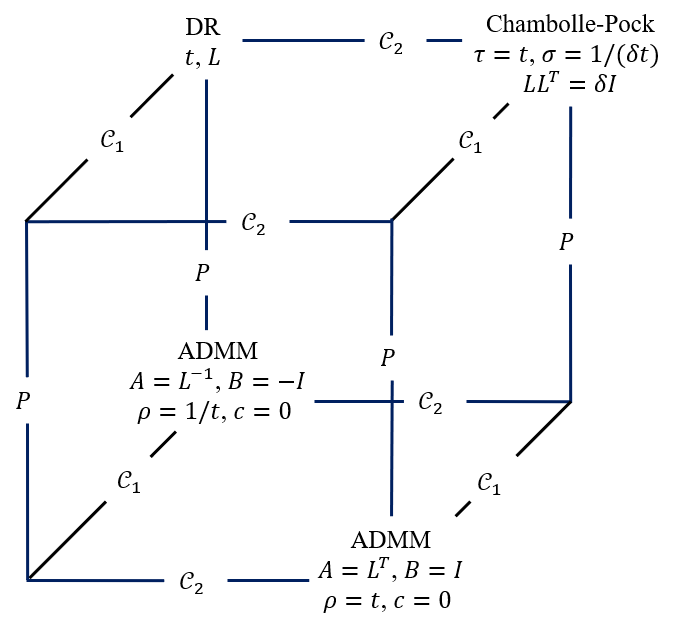}
	\caption{Connections between DR, ADMM, and Chambolle-Pock method.}
	\label{fig11}
\end{figure}

\section{Linnaeus}\label{package}
We have presented a framework for detecting equivalence between iterative
algorithms for continuous optimization.
In this section, we briefly introduce a software package called \lin{}
that implements these ideas.
The implementation and documentation are available at
\url{https://github.com/udellgroup/Linnaeus_software}.
More detailed information can be found in \cref{package_detailed}.

The input is an algorithm described in user-friendly syntax
with variables, parameters, functions, oracles, and update equations. 
The system will automatically translate the input algorithm into a canonical form
(the transfer function)
and use the canonical form to identify whether the algorithm is equivalent
to any reference algorithm,
possibly after transformations such as permutation, conjugation, or repetition.
All expressions in \lin{} are defined symbolically,
using the python package for symbolic mathematics \emph{sympy}.

Given two input algorithms,
\lin{} computes the transfer functions
and can compare them to detect equivalence and other relations.
Some algorithms are equivalent or related only when the parameters satisfy a certain
condition: for example, DR and ADMM.
If the transfer functions of each algorithm use different parameters,
\lin{} form symbolic equations and solve the equations
to determine conditions that, if satisfied by the algorithm parameters,
yield the desired relation between the algorithms;
see \cref{eq3} in \cref{charac-oracle}.

This package can be used by researchers (or peer reviewers) who wish
to understand the novelty of new algorithmic ideas and connections
to existing algorithms.
Further, the software can also serve as a search engine,
which will identify connections from the input algorithm
to existing algorithms in the literature
that appear in \lin{}'s algorithm library.

\section{Conclusion and future work}
In this paper, we have presented a framework for reasoning about equivalence
between a broad class of iterative algorithms
by using ideas from control theory to represent optimization algorithms.
The main insight is that by representing an algorithm as a
linear dynamical system in feedback with a static nonlinearity,
we can recognize equivalent algorithms by detecting
algebraic relations between the transfer functions of the associated linear systems.
This framework can identify algorithms
that result in the same sequence of oracle calls,
or algorithms that are the same up to shifts of the update equations,
repetition of the updates with the same unit block,
and conjugation of the function oracles.
These ideas are implemented in the software package \lin{},
which allows researchers to search for algorithms that are related
to a given input and identify parameter settings that make the algorithms equivalent.
Our goal is to allow researchers add new algorithms to \lin{} as they
are developed, so that \lin{} can remain a valuable resource for
algorithm designers seeking to understand connections (if any) to previous methods.

Our framework requires that the algorithm is linear in the state and oracle outputs,
but not necessarily in the parameters.
This constraint still allows us to handle a surprisingly large class of algorithms.
There are several interesting directions for future work.

Can we detect equivalence between stochastic or randomized algorithms?
Our framework applies to such algorithms with almost no modifations,
simply by allowing random oracles.
For example, we can accept oracles like
random search $\argmin \{ f(x + \omega_i): i=1,\ldots,k \}$,
stochastic gradient $\nabla f(x) + \omega$,
or noisy gradient $\nabla f(x + \omega)$.
The definition of oracle equivalence would need a slight modification:
for algorithms that use (pseudo-)randomized oracles,
two algorithms are oracle-equivalent if they
generate identical sequences of oracle calls given the same random seed.

Can we detect equivalence between parallel or distributed algorithms?
Surprisingly, our framework still works for parallel or distributed algorithms.
Notice that in a parallel algorithm, many oracle calls may be independently executed
on different processors at about the same time.
The precise ordering of these calls is not determined by the algorithm,
and so different runs of the algorithm can generate different oracle sequences.
However, all the possible oracle sequences generated by the same algorithm
share the same dependence graph.
Using the formalism defined in \cref{odg},
we can see that our framework can identify equivalence between parallel or
distributed algorithms using the expanded definition of oracle equivalence:
two algorithms are oracle-equivalent if there exists a
way of writing each algorithm as a sequence of updates so that
they generate identical sequences of oracle calls.

Can we detect equivalence between adaptive or nonlinear algorithms?
Transfer functions are only defined for linear time-invariant (LTI) systems, so the LTI assumption in our framework is critical. Nevertheless, many of the other concepts from \cref{control} do extend to systems that are \emph{almost} LTI.
For example, an algorithm with parameters that change on a fixed schedule but is otherwise linear, such as gradient descent with a diminishing stepsize, can be regarded as a linear time-varying (LTV) system~\cite{antsaklis2006linear}, and the notion of a transfer function has been generalized to LTV systems~\cite{LTV_TF}.
If, instead, the parameters change adaptively based on the other state variables,
the system can be regarded as a linear parameter varying (LPV) system~\cite{LPV_book} or a switched system~\cite{sun2006switched}. Examples of such algorithms include nonlinear conjugate gradient methods and quasi-Newton methods.

For these more complicated cases, it is still reasonable to ask whether two algorithms invoke the same sequence of oracle calls.
Discovering representations for nonlinear or time-varying algorithms that suffice to check equivalence is an interesting direction for future research.

\bibliographystyle{siamplain}
\bibliography{references}

\newpage
\appendix

\section{Linnaeus}\label{package_detailed}
In this section, we introduce our software package called \lin{}
that implements these ideas in detail.
This package can be used by researchers (or peer reviewers) who wish
to understand the novelty of new algorithmic ideas and connections
to existing algorithms.
The input is an algorithm described in user-friendly syntax
with variables, parameters, functions, oracles, and update equations.
The system will automatically translate the input algorithm into a canonical form
(the transfer function)
and use the canonical form to identify whether the algorithm is equivalent
to any reference algorithm,
possibly after transformations such as permutation, conjugation, or repetition.
Further, the software can also serve as a search engine,
which will identify connections from the input algorithm
to existing algorithms in the literature
that appear in \lin{}'s algorithm library.

\subsection{Illustrative examples}
We use \lin{} to identify the relations
between algorithms presented previously in the paper.
These examples demonstrate the power and simplicity of \lin{}.
Code for these examples can be found at \url{https://github.com/udellgroup/Linnaeus_software}.

\titleparagraph{\Cref{algo1,algo2}}
The following code identifies that
\cref{algo1,algo2} are oracle-equivalent.
We input \cref{algo1,algo2} with variables, oracles, and update equations,
and parse them into state-space realizations.
Then we check oracle equivalence using the function \texttt{is\_equivalent}.
The system returns \texttt{True},
consistent with our analytical results in \cref{example,charac-oracle}.

\begin{changemargin}{1cm}{1cm}
\begin{tcolorbox}[breakable, size=fbox, boxrule=1pt, pad at break*=1mm,colback=cellbackground, colframe=cellborder]
\begin{Verbatim}[commandchars=\\\{\}]
\PY{c+c1}{\PYZsh{} define Algorithm 3.1}
\PY{n}{algo1} \PY{o}{=} \PY{n}{Algorithm}\PY{p}{(}\PY{l+s+s2}{\PYZdq{}}\PY{l+s+s2}{Algorithm 3.1}\PY{l+s+s2}{\PYZdq{}}\PY{p}{)}

\PY{c+c1}{\PYZsh{} add oracle gradient of f to Algorithm 3.1}
\PY{n}{gradf} \PY{o}{=} \PY{n}{algo1}\PY{o}{.}\PY{n}{add\PYZus{}oracle}\PY{p}{(}\PY{l+s+s2}{\PYZdq{}}\PY{l+s+s2}{gradf}\PY{l+s+s2}{\PYZdq{}}\PY{p}{)}

\PY{c+c1}{\PYZsh{} add variables x1, x2, and x3 to Algorithm 3.1}
\PY{n}{x1}\PY{p}{,} \PY{n}{x2}\PY{p}{,} \PY{n}{x3} \PY{o}{=} \PY{n}{algo1}\PY{o}{.}\PY{n}{add\PYZus{}var}\PY{p}{(}\PY{l+s+s2}{\PYZdq{}}\PY{l+s+s2}{x1}\PY{l+s+s2}{\PYZdq{}}\PY{p}{,} \PY{l+s+s2}{\PYZdq{}}\PY{l+s+s2}{x2}\PY{l+s+s2}{\PYZdq{}}\PY{p}{,} \PY{l+s+s2}{\PYZdq{}}\PY{l+s+s2}{x3}\PY{l+s+s2}{\PYZdq{}}\PY{p}{)}

\PY{c+c1}{\PYZsh{} add update equations}
\PY{c+c1}{\PYZsh{} x3 \PYZlt{}\PYZhy{} 2x1 \PYZhy{} x2 }
\PY{n}{algo1}\PY{o}{.}\PY{n}{add\PYZus{}update}\PY{p}{(}\PY{n}{x3}\PY{p}{,} \PY{l+m+mi}{2}\PY{o}{*}\PY{n}{x1} \PY{o}{\PYZhy{}} \PY{n}{x2}\PY{p}{)}
\PY{c+c1}{\PYZsh{} x2 \PYZlt{}\PYZhy{} x1}
\PY{n}{algo1}\PY{o}{.}\PY{n}{add\PYZus{}update}\PY{p}{(}\PY{n}{x2}\PY{p}{,} \PY{n}{x1}\PY{p}{)}
\PY{c+c1}{\PYZsh{} x1 \PYZlt{}\PYZhy{} x3 \PYZhy{} 1/10*gradf(x3)}
\PY{n}{algo1}\PY{o}{.}\PY{n}{add\PYZus{}update}\PY{p}{(}\PY{n}{x1}\PY{p}{,} \PY{n}{x3} \PY{o}{\PYZhy{}} \PY{l+m+mi}{1}\PY{o}{/}\PY{l+m+mi}{10}\PY{o}{*}\PY{n}{gradf}\PY{p}{(}\PY{n}{x3}\PY{p}{)}\PY{p}{)}

\PY{c+c1}{\PYZsh{} parse Algorithm 3.1, translate it into canonical form}
\PY{n}{algo1}\PY{o}{.}\PY{n}{parse}\PY{p}{(}\PY{p}{)}
\end{Verbatim}
\end{tcolorbox}
\end{changemargin}

\begin{Verbatim}[commandchars=\\\{\}]
	--------------------------------------------------------------
	Parse Algorithm 3.1:
\end{Verbatim}

$\quad \qquad
\begin{aligned} 
	x_3 & \gets 2x_1 - x2 \\
	x_2 & \gets x_1 \\
	x_1 & \gets x_3 - 0.1\text{gradf}(x_3)
\end{aligned}$

\begin{Verbatim}[commandchars=\\\{\}]
	--------------------------------------------------------------
\end{Verbatim}

\begin{changemargin}{1cm}{1cm}
\begin{tcolorbox}[breakable, size=fbox, boxrule=1pt, pad at break*=1mm,colback=cellbackground, colframe=cellborder]
\begin{Verbatim}[commandchars=\\\{\}]
\PY{n}{algo2} \PY{o}{=} \PY{n}{Algorithm}\PY{p}{(}\PY{l+s+s2}{\PYZdq{}}\PY{l+s+s2}{Algorithm 3.2}\PY{l+s+s2}{\PYZdq{}}\PY{p}{)}
\PY{n}{xi1}\PY{p}{,} \PY{n}{xi2}\PY{p}{,} \PY{n}{xi3} \PY{o}{=} \PY{n}{algo2}\PY{o}{.}\PY{n}{add\PYZus{}var}\PY{p}{(}\PY{l+s+s2}{\PYZdq{}}\PY{l+s+s2}{xi1}\PY{l+s+s2}{\PYZdq{}}\PY{p}{,} \PY{l+s+s2}{\PYZdq{}}\PY{l+s+s2}{xi2}\PY{l+s+s2}{\PYZdq{}}\PY{p}{,} \PY{l+s+s2}{\PYZdq{}}\PY{l+s+s2}{xi3}\PY{l+s+s2}{\PYZdq{}}\PY{p}{)}
\PY{n}{gradf} \PY{o}{=} \PY{n}{algo2}\PY{o}{.}\PY{n}{add\PYZus{}oracle}\PY{p}{(}\PY{l+s+s2}{\PYZdq{}}\PY{l+s+s2}{gradf}\PY{l+s+s2}{\PYZdq{}}\PY{p}{)}

\PY{c+c1}{\PYZsh{} xi3 \PYZlt{}\PYZhy{} xi1}
\PY{n}{algo2}\PY{o}{.}\PY{n}{add\PYZus{}update}\PY{p}{(}\PY{n}{xi3}\PY{p}{,} \PY{n}{xi1}\PY{p}{)}
\PY{c+c1}{\PYZsh{} xi1 \PYZlt{}\PYZhy{} xi1 \PYZhy{} xi2 \PYZhy{} 1/5*gradf(xi1)}
\PY{n}{algo2}\PY{o}{.}\PY{n}{add\PYZus{}update}\PY{p}{(}\PY{n}{xi1}\PY{p}{,} \PY{n}{xi1} \PY{o}{\PYZhy{}} \PY{n}{xi2} \PY{o}{\PYZhy{}} \PY{l+m+mi}{1}\PY{o}{/}\PY{l+m+mi}{5}\PY{o}{*}\PY{n}{gradf}\PY{p}{(}\PY{n}{xi3}\PY{p}{)}\PY{p}{)}
\PY{c+c1}{\PYZsh{} xi2 \PYZlt{}\PYZhy{} xi2 + 1/10*gradf(xi3)}
\PY{n}{algo2}\PY{o}{.}\PY{n}{add\PYZus{}update}\PY{p}{(}\PY{n}{xi2}\PY{p}{,} \PY{n}{xi2} \PY{o}{+} \PY{l+m+mi}{1}\PY{o}{/}\PY{l+m+mi}{10}\PY{o}{*}\PY{n}{gradf}\PY{p}{(}\PY{n}{xi3}\PY{p}{)}\PY{p}{)}

\PY{n}{algo2}\PY{o}{.}\PY{n}{parse}\PY{p}{(}\PY{p}{)}
\end{Verbatim}
\end{tcolorbox}
\end{changemargin}

\begin{Verbatim}[commandchars=\\\{\}]
	--------------------------------------------------------------
	Parse Algorithm 3.2:
\end{Verbatim}

$\quad \qquad
\begin{aligned} 
	\xi_3 & \gets \xi_1 \\
	\xi_1 & \gets \xi_1 - \xi_2 - 0.2\text{gradf}(\xi_3) \\
	\xi_2 & \gets \xi_2 + 0.1\text{gradf}(\xi_3)
\end{aligned}$

\begin{Verbatim}[commandchars=\\\{\}]
	--------------------------------------------------------------
\end{Verbatim}

\begin{changemargin}{1cm}{1cm}
\begin{tcolorbox}[breakable, size=fbox, boxrule=1pt, pad at break*=1mm,colback=cellbackground, colframe=cellborder]
\begin{Verbatim}[commandchars=\\\{\}]
\PY{c+c1}{\PYZsh{} check oracle equivalence}
\PY{n}{lin}\PY{o}{.}\PY{n}{is\PYZus{}equivalent}\PY{p}{(}\PY{n}{algo1}\PY{p}{,} \PY{n}{algo2}\PY{p}{,} \PY{n}{verbose} \PY{o}{=} \PY{k+kc}{True}\PY{p}{)}
\end{Verbatim}
\end{tcolorbox}
\end{changemargin}

\begin{Verbatim}[commandchars=\\\{\}]
	--------------------------------------------------------------
	Algorithm 3.1 is equivalent to Algorithm 3.2.
	--------------------------------------------------------------
	True
\end{Verbatim}

\titleparagraph{\Cref{algo5,algo6}}
The second example identifies that
\cref{algo5,algo6} are shift-equivalent.
We input and parse the algorithms into state-space realizations
and then check shift equivalence (cyclic permutation)
using the function \texttt{is\_permutation}.
The system returns \texttt{True},
consistent with results in \cref{example,charac-shift}.

\begin{changemargin}{1cm}{1cm}
\begin{tcolorbox}[breakable, size=fbox, boxrule=1pt, pad at break*=1mm,colback=cellbackground, colframe=cellborder]
\begin{Verbatim}[commandchars=\\\{\}]
\PY{n}{algo5} \PY{o}{=} \PY{n}{Algorithm}\PY{p}{(}\PY{l+s+s2}{\PYZdq{}}\PY{l+s+s2}{Algorithm 3.5}\PY{l+s+s2}{\PYZdq{}}\PY{p}{)}
\PY{n}{x1}\PY{p}{,} \PY{n}{x2}\PY{p}{,} \PY{n}{x3} \PY{o}{=} \PY{n}{algo5}\PY{o}{.}\PY{n}{add\PYZus{}var}\PY{p}{(}\PY{l+s+s2}{\PYZdq{}}\PY{l+s+s2}{x1}\PY{l+s+s2}{\PYZdq{}}\PY{p}{,} \PY{l+s+s2}{\PYZdq{}}\PY{l+s+s2}{x2}\PY{l+s+s2}{\PYZdq{}}\PY{p}{,} \PY{l+s+s2}{\PYZdq{}}\PY{l+s+s2}{x3}\PY{l+s+s2}{\PYZdq{}}\PY{p}{)}
\PY{n}{proxf}\PY{p}{,} \PY{n}{proxg} \PY{o}{=} \PY{n}{algo5}\PY{o}{.}\PY{n}{add\PYZus{}oracle}\PY{p}{(}\PY{l+s+s2}{\PYZdq{}}\PY{l+s+s2}{proxf}\PY{l+s+s2}{\PYZdq{}}\PY{p}{,} \PY{l+s+s2}{\PYZdq{}}\PY{l+s+s2}{proxg}\PY{l+s+s2}{\PYZdq{}}\PY{p}{)}

\PY{c+c1}{\PYZsh{} x1 \PYZlt{}\PYZhy{} proxf(x3)}
\PY{n}{algo5}\PY{o}{.}\PY{n}{add\PYZus{}update}\PY{p}{(}\PY{n}{x1}\PY{p}{,} \PY{n}{proxf}\PY{p}{(}\PY{n}{x3}\PY{p}{)}\PY{p}{)}
\PY{c+c1}{\PYZsh{} x2 \PYZlt{}\PYZhy{} proxg(2x1 \PYZhy{} x3)}
\PY{n}{algo5}\PY{o}{.}\PY{n}{add\PYZus{}update}\PY{p}{(}\PY{n}{x2}\PY{p}{,} \PY{n}{proxg}\PY{p}{(}\PY{l+m+mi}{2}\PY{o}{*}\PY{n}{x1} \PY{o}{\PYZhy{}} \PY{n}{x3}\PY{p}{)}\PY{p}{)}
\PY{c+c1}{\PYZsh{} x3 \PYZlt{}\PYZhy{} x3 + x2 \PYZhy{} x1}
\PY{n}{algo5}\PY{o}{.}\PY{n}{add\PYZus{}update}\PY{p}{(}\PY{n}{x3}\PY{p}{,} \PY{n}{x3} \PY{o}{+} \PY{n}{x2} \PY{o}{\PYZhy{}} \PY{n}{x1}\PY{p}{)}

\PY{n}{algo5}\PY{o}{.}\PY{n}{parse}\PY{p}{(}\PY{p}{)}
\end{Verbatim}
\end{tcolorbox}
\end{changemargin}

\begin{Verbatim}[commandchars=\\\{\}]
	--------------------------------------------------------------
	Parse Algorithm 3.5:
\end{Verbatim}

$\quad \qquad
\begin{aligned} 
	x_1 & \gets \text{proxf}(x_3) \\
	x_2 & \gets \text{proxg}(2x_1 - x_3) \\
	x_3 & \gets x_3 + x_2 - x_1
\end{aligned}$

\begin{Verbatim}[commandchars=\\\{\}]
	--------------------------------------------------------------
\end{Verbatim}

\begin{changemargin}{1cm}{1cm}
\begin{tcolorbox}[breakable, size=fbox, boxrule=1pt, pad at break*=1mm,colback=cellbackground, colframe=cellborder]
\begin{Verbatim}[commandchars=\\\{\}]
\PY{n}{algo4} \PY{o}{=} \PY{n}{Algorithm}\PY{p}{(}\PY{l+s+s2}{\PYZdq{}}\PY{l+s+s2}{Algorithm 3.6}\PY{l+s+s2}{\PYZdq{}}\PY{p}{)}
\PY{n}{xi1}\PY{p}{,} \PY{n}{xi2} \PY{o}{=} \PY{n}{algo4}\PY{o}{.}\PY{n}{add\PYZus{}var}\PY{p}{(}\PY{l+s+s2}{\PYZdq{}}\PY{l+s+s2}{xi1}\PY{l+s+s2}{\PYZdq{}}\PY{p}{,} \PY{l+s+s2}{\PYZdq{}}\PY{l+s+s2}{xi2}\PY{l+s+s2}{\PYZdq{}}\PY{p}{)}
\PY{n}{proxf}\PY{p}{,} \PY{n}{proxg} \PY{o}{=} \PY{n}{algo4}\PY{o}{.}\PY{n}{add\PYZus{}oracle}\PY{p}{(}\PY{l+s+s2}{\PYZdq{}}\PY{l+s+s2}{proxf}\PY{l+s+s2}{\PYZdq{}}\PY{p}{,} \PY{l+s+s2}{\PYZdq{}}\PY{l+s+s2}{proxg}\PY{l+s+s2}{\PYZdq{}}\PY{p}{)}

\PY{c+c1}{\PYZsh{} xi1 \PYZlt{}\PYZhy{} proxg(\PYZhy{}xi1 + 2xi2) + xi1 \PYZhy{} xi2}
\PY{n}{algo4}\PY{o}{.}\PY{n}{add\PYZus{}update}\PY{p}{(}\PY{n}{xi1}\PY{p}{,} \PY{n}{proxg}\PY{p}{(}\PY{o}{\PYZhy{}}\PY{n}{xi1} \PY{o}{+} \PY{l+m+mi}{2}\PY{o}{*}\PY{n}{xi2}\PY{p}{)} \PY{o}{+} \PY{n}{xi1} \PY{o}{\PYZhy{}} \PY{n}{xi2}\PY{p}{)}
\PY{c+c1}{\PYZsh{} xi2 \PYZlt{}\PYZhy{} proxf(xi1)}
\PY{n}{algo4}\PY{o}{.}\PY{n}{add\PYZus{}update}\PY{p}{(}\PY{n}{xi2}\PY{p}{,} \PY{n}{proxf}\PY{p}{(}\PY{n}{xi1}\PY{p}{)}\PY{p}{)}

\PY{n}{algo4}\PY{o}{.}\PY{n}{parse}\PY{p}{(}\PY{p}{)}
\end{Verbatim}
\end{tcolorbox}
\end{changemargin}

\begin{Verbatim}[commandchars=\\\{\}]
	--------------------------------------------------------------
	Parse Algorithm 3.6:
\end{Verbatim}

$\quad \qquad
\begin{aligned} 
	\xi_1 & \gets \text{proxg}(-\xi_1 + 2\xi_2) + \xi_1 - \xi_2 \\
	\xi_2 & \gets \text{proxf}(\xi_1) 
\end{aligned}$

\begin{Verbatim}[commandchars=\\\{\}]
	--------------------------------------------------------------
\end{Verbatim}

\begin{changemargin}{1cm}{1cm}
\begin{tcolorbox}[breakable, size=fbox, boxrule=1pt, pad at break*=1mm,colback=cellbackground, colframe=cellborder]
\begin{Verbatim}[commandchars=\\\{\}]
\PY{c+c1}{\PYZsh{} check cyclic permutation (shift equivalence)}
\PY{n}{lin}\PY{o}{.}\PY{n}{is\PYZus{}permutation}\PY{p}{(}\PY{n}{algo5}\PY{p}{,} \PY{n}{algo6}\PY{p}{,} \PY{n}{verbose} \PY{o}{=} \PY{k+kc}{True}\PY{p}{)}
\end{Verbatim}
\end{tcolorbox}
\end{changemargin}

\begin{Verbatim}[commandchars=\\\{\}]
	--------------------------------------------------------------
	Algorithm 3.5 is a permutation of Algorithm 3.6.
	--------------------------------------------------------------
	True
\end{Verbatim}

\titleparagraph{DR and ADMM}
The third illustrative example shows that DR and ADMM are related
by permutation and conjugation, as we saw in \cref{conjugation}.
Further, \lin{} can even reveal the specific parameter choice
required for the relation to hold.
Just as in \cref{conjugation}, suppose both DR and ADMM solve problem \cref{eqp1}
with $A = L^T$, $B = I$, and $c = 0$.
We input and parse DR and ADMM.
To detect the relations,
we use function \texttt{test\_conjugate\_permutation} to check
conjugation and permutation between DR and ADMM.
The results are the same as \cref{conjugation}.

\begin{changemargin}{1cm}{1cm}
\begin{tcolorbox}[breakable, size=fbox, boxrule=1pt, pad at break*=1mm,colback=cellbackground, colframe=cellborder]
\begin{Verbatim}[commandchars=\\\{\}]
\PY{n}{DR} \PY{o}{=} \PY{n}{Algorithm}\PY{p}{(}\PY{l+s+s2}{\PYZdq{}}\PY{l+s+s2}{Douglas\PYZhy{}Rachford splitting}\PY{l+s+s2}{\PYZdq{}}\PY{p}{)}
\PY{n}{x1}\PY{p}{,} \PY{n}{x2}\PY{p}{,} \PY{n}{x3} \PY{o}{=} \PY{n}{DR}\PY{o}{.}\PY{n}{add\PYZus{}var}\PY{p}{(}\PY{l+s+s2}{\PYZdq{}}\PY{l+s+s2}{x1}\PY{l+s+s2}{\PYZdq{}}\PY{p}{,} \PY{l+s+s2}{\PYZdq{}}\PY{l+s+s2}{x2}\PY{l+s+s2}{\PYZdq{}}\PY{p}{,} \PY{l+s+s2}{\PYZdq{}}\PY{l+s+s2}{x3}\PY{l+s+s2}{\PYZdq{}}\PY{p}{)}
\PY{n}{t} \PY{o}{=} \PY{n}{DR}\PY{o}{.}\PY{n}{add\PYZus{}parameter}\PY{p}{(}\PY{l+s+s2}{\PYZdq{}}\PY{l+s+s2}{t}\PY{l+s+s2}{\PYZdq{}}\PY{p}{)}
\PY{n}{L} \PY{o}{=} \PY{n}{DR}\PY{o}{.}\PY{n}{add\PYZus{}parameter}\PY{p}{(}\PY{l+s+s2}{\PYZdq{}}\PY{l+s+s2}{L}\PY{l+s+s2}{\PYZdq{}}\PY{p}{,} \PY{n}{commutative} \PY{o}{=} \PY{k+kc}{False}\PY{p}{)}

\PY{c+c1}{\PYZsh{} x1 \PYZlt{}\PYZhy{} prox\PYZus{}tf(x3)}
\PY{n}{DR}\PY{o}{.}\PY{n}{add\PYZus{}update}\PY{p}{(}\PY{n}{x1}\PY{p}{,} \PY{n}{lin}\PY{o}{.}\PY{n}{prox}\PY{p}{(}\PY{n}{f}\PY{p}{,} \PY{n}{t}\PY{p}{)}\PY{p}{(}\PY{n}{x3}\PY{p}{)}\PY{p}{)}
\PY{c+c1}{\PYZsh{} x2 \PYZlt{}\PYZhy{} prox\PYZus{}tgL(2x1 \PYZhy{} x3)}
\PY{n}{DR}\PY{o}{.}\PY{n}{add\PYZus{}update}\PY{p}{(}\PY{n}{x2}\PY{p}{,} \PY{n}{lin}\PY{o}{.}\PY{n}{prox}\PY{p}{(}\PY{n}{g}\PY{p}{,} \PY{n}{t}\PY{p}{,} \PY{n}{L}\PY{p}{)}\PY{p}{(}\PY{l+m+mi}{2}\PY{o}{*}\PY{n}{x1} \PY{o}{\PYZhy{}} \PY{n}{x3}\PY{p}{)}\PY{p}{)}
\PY{c+c1}{\PYZsh{} x3 \PYZlt{}\PYZhy{} x3 + x2 \PYZhy{} x1}
\PY{n}{DR}\PY{o}{.}\PY{n}{add\PYZus{}update}\PY{p}{(}\PY{n}{x3}\PY{p}{,} \PY{n}{x3} \PY{o}{+} \PY{n}{x2} \PY{o}{\PYZhy{}} \PY{n}{x1}\PY{p}{)}

\PY{n}{DR}\PY{o}{.}\PY{n}{parse}\PY{p}{(}\PY{p}{)}
\end{Verbatim}
\end{tcolorbox}
\end{changemargin}

\begin{Verbatim}[commandchars=\\\{\}]
	--------------------------------------------------------------
	Parse Douglas-Rachford splitting:
\end{Verbatim}

$\quad \qquad
\begin{aligned} 
	x_1 & \gets \text{prox}_{tf}(x_3) \\
	x_2 & \gets \text{prox}_{t(g \circ L)}(2x_1 - x_3) \\
	x_3 & \gets x_3 + x_2 - x_1
\end{aligned}$

\begin{Verbatim}[commandchars=\\\{\}]
	--------------------------------------------------------------
\end{Verbatim}

\begin{changemargin}{1cm}{1cm}
\begin{tcolorbox}[breakable, size=fbox, boxrule=1pt, pad at break*=1mm,colback=cellbackground, colframe=cellborder]
\begin{Verbatim}[commandchars=\\\{\}]
\PY{n}{ADMM} \PY{o}{=} \PY{n}{Algorithm}\PY{p}{(}\PY{l+s+s2}{\PYZdq{}}\PY{l+s+s2}{ADMM}\PY{l+s+s2}{\PYZdq{}}\PY{p}{)}
\PY{n}{f}\PY{p}{,} \PY{n}{g} \PY{o}{=} \PY{n}{ADMM}\PY{o}{.}\PY{n}{add\PYZus{}function}\PY{p}{(}\PY{l+s+s2}{\PYZdq{}}\PY{l+s+s2}{f}\PY{l+s+s2}{\PYZdq{}}\PY{p}{,} \PY{l+s+s2}{\PYZdq{}}\PY{l+s+s2}{g}\PY{l+s+s2}{\PYZdq{}}\PY{p}{)}
\PY{n}{rho} \PY{o}{=} \PY{n}{ADMM}\PY{o}{.}\PY{n}{add\PYZus{}parameter}\PY{p}{(}\PY{l+s+s2}{\PYZdq{}}\PY{l+s+s2}{rho}\PY{l+s+s2}{\PYZdq{}}\PY{p}{)}
\PY{n}{L} \PY{o}{=} \PY{n}{ADMM}\PY{o}{.}\PY{n}{add\PYZus{}parameter}\PY{p}{(}\PY{l+s+s2}{\PYZdq{}}\PY{l+s+s2}{L}\PY{l+s+s2}{\PYZdq{}}\PY{p}{,} \PY{n}{commutative} \PY{o}{=} \PY{k+kc}{False}\PY{p}{)}
\PY{n}{xi1}\PY{p}{,} \PY{n}{xi2}\PY{p}{,} \PY{n}{xi3} \PY{o}{=} \PY{n}{ADMM}\PY{o}{.}\PY{n}{add\PYZus{}var}\PY{p}{(}\PY{l+s+s2}{\PYZdq{}}\PY{l+s+s2}{xi1}\PY{l+s+s2}{\PYZdq{}}\PY{p}{,} \PY{l+s+s2}{\PYZdq{}}\PY{l+s+s2}{xi2}\PY{l+s+s2}{\PYZdq{}}\PY{p}{,} \PY{l+s+s2}{\PYZdq{}}\PY{l+s+s2}{xi3}\PY{l+s+s2}{\PYZdq{}}\PY{p}{)}

\PY{c+c1}{\PYZsh{} xi1 \PYZlt{}\PYZhy{} argmin(x1, g\PYZca{}*(xi1) + 1/2*rho*||T(L)xi1 + xi2 + xi3||\PYZca{}2)}
\PY{n}{ADMM}\PY{o}{.}\PY{n}{add\PYZus{}update}\PY{p}{(}\PY{n}{xi1}\PY{p}{,} \PY{n}{lin}\PY{o}{.}\PY{n}{argmin}\PY{p}{(}\PY{n}{xi1}\PY{p}{,} \PY{n}{g}\PY{p}{(}\PY{n}{xi1}\PY{p}{)} \PY{o}{+} \PY{l+m+mi}{1}\PY{o}{/}\PY{l+m+mi}{2}\PY{o}{*}\PY{n}{rho}\PY{o}{*}\PY{n}{lin}\PY{o}{.}\PY{n}{norm\PYZus{}square}\PY{p}{(}\PY{n}{T}\PY{p}{(}\PY{n}{L}\PY{p}{)}\PY{o}{*}\PY{n}{xi1} \PY{o}{+} \PY{n}{xi2} \PY{o}{+} \PY{n}{xi3}\PY{p}{)}\PY{p}{)}\PY{p}{)}
\PY{c+c1}{\PYZsh{} xi2 \PYZlt{}\PYZhy{} argmin(x2, f\PYZca{}*(xi2) + 1/2*rho*||T(L)xi1 + xi2 + xi3||\PYZca{}2)}
\PY{n}{ADMM}\PY{o}{.}\PY{n}{add\PYZus{}update}\PY{p}{(}\PY{n}{xi2}\PY{p}{,} \PY{n}{lin}\PY{o}{.}\PY{n}{argmin}\PY{p}{(}\PY{n}{xi2}\PY{p}{,} \PY{n}{f}\PY{p}{(}\PY{n}{xi2}\PY{p}{)} \PY{o}{+} \PY{l+m+mi}{1}\PY{o}{/}\PY{l+m+mi}{2}\PY{o}{*}\PY{n}{rho}\PY{o}{*}\PY{n}{lin}\PY{o}{.}\PY{n}{norm\PYZus{}square}\PY{p}{(}\PY{n}{T}\PY{p}{(}\PY{n}{L}\PY{p}{)}\PY{o}{*}\PY{n}{xi1} \PY{o}{+} \PY{n}{xi2} \PY{o}{+} \PY{n}{xi3}\PY{p}{)}\PY{p}{)}\PY{p}{)}
\PY{c+c1}{\PYZsh{} xi3 \PYZlt{}\PYZhy{} xi3 + T(L)xi1 + xi2}
\PY{n}{ADMM}\PY{o}{.}\PY{n}{add\PYZus{}update}\PY{p}{(}\PY{n}{xi3}\PY{p}{,} \PY{n}{xi3} \PY{o}{+} \PY{n}{T}\PY{p}{(}\PY{n}{L}\PY{p}{)}\PY{o}{*}\PY{n}{xi1} \PY{o}{+} \PY{n}{xi2}\PY{p}{)}

\PY{n}{ADMM}\PY{o}{.}\PY{n}{parse}\PY{p}{(}\PY{p}{)}
\end{Verbatim}
\end{tcolorbox}
\end{changemargin}

\begin{Verbatim}[commandchars=\\\{\}]
	--------------------------------------------------------------
	Parse ADMM:
\end{Verbatim}

$\quad \qquad
\begin{aligned} 
	\xi_1 & \gets \text{argmin}_{\xi_1}\{ g(\xi_1) + 0.5 \rho \text{norm}_{\text{square}}(T(L)\xi_1 + \xi_2 + \xi_3) \} \\
	\xi_2 & \gets \text{argmin}_{\xi_2}\{ f(\xi_2) + 0.5 \rho \text{norm}_{\text{square}}(T(L)\xi_1 + \xi_2 + \xi_3) \} \\
	\xi_3 & \gets T(L)\xi_1 + \xi_2 + \xi_3
\end{aligned}$

\begin{Verbatim}[commandchars=\\\{\}]
	--------------------------------------------------------------
\end{Verbatim}

\begin{changemargin}{1cm}{1cm}
\begin{tcolorbox}[breakable, size=fbox, boxrule=1pt, pad at break*=1mm,colback=cellbackground, colframe=cellborder]
\begin{Verbatim}[commandchars=\\\{\}]
\PY{c+c1}{\PYZsh{} check permutation and conjugation }
\PY{c+c1}{\PYZsh{} between DR and ADMM}
\PY{n}{lin}\PY{o}{.}\PY{n}{test\PYZus{}conjugate\PYZus{}permutation}\PY{p}{(}\PY{n}{DR}\PY{p}{,} \PY{n}{ADMM}\PY{p}{)}
\end{Verbatim}
\end{tcolorbox}
\end{changemargin}

\begin{Verbatim}[commandchars=\\\{\}]
	--------------------------------------------------------------
	==============================================================
	Parameters of Douglas-Rachford splitting:
\end{Verbatim}

$\quad \qquad t$, $L$

\begin{Verbatim}[commandchars=\\\{\}]
	Parameters of ADMM:
\end{Verbatim}

$\quad \qquad \rho$, $L$

\begin{Verbatim}[commandchars=\\\{\}]
	Douglas-Rachford splitting is a conjugate permutation of ADMM, if the parameters
	satisfy:
\end{Verbatim}

$\quad \qquad
\begin{aligned} 
	& \rho = t \\
	& L = L
\end{aligned}$

\begin{Verbatim}[commandchars=\\\{\}]
	==============================================================
	--------------------------------------------------------------
\end{Verbatim}

\subsection{Implementation}
In this subsection, we briefly describe the implementation of \lin{}.
All expressions in \lin{} are defined symbolically,
using the python package for symbolic mathematics \emph{sympy}.
In \lin{}, an algorithm is specified by defining
variables, parameters, functions, oracles, and update equations.
All variables and parameters are symbolic, so there is no need to specialize
problem dimensions or parameter choices.
The system automatically translates an input algorithm into its
state-space realization and computes the transfer function.
The transfer functions can be compared and manipulated as needed
to establish various kinds of equivalences or other relations between algorithms.


\titleparagraph{Parameter declaration}
Parameters of the algorithm can be declared as scalar (commutative)
or vector or matrix (noncommutative).
The following code shows how to add scalar \texttt{t} and matrix \texttt{L} to \texttt{algo1}.
\begin{changemargin}{1cm}{1cm}
\begin{tcolorbox}[breakable, size=fbox, boxrule=1pt, pad at break*=1mm,colback=cellbackground, colframe=cellborder]
\begin{Verbatim}[commandchars=\\\{\}]
\PY{c+c1}{\PYZsh{} add a scalar parameter t}
\PY{n}{t} \PY{o}{=} \PY{n}{algo1}\PY{o}{.}\PY{n}{add\PYZus{}parameter}\PY{p}{(}\PY{l+s+s2}{\PYZdq{}}\PY{l+s+s2}{t}\PY{l+s+s2}{\PYZdq{}}\PY{p}{)}
\PY{c+c1}{\PYZsh{} add a matrix parameter L}
\PY{n}{L} \PY{o}{=} \PY{n}{algo1}\PY{o}{.}\PY{n}{add\PYZus{}parameter}\PY{p}{(}\PY{l+s+s2}{\PYZdq{}}\PY{l+s+s2}{L}\PY{l+s+s2}{\PYZdq{}}\PY{p}{,} \PY{n}{commutative} \PY{o}{=} \PY{k+kc}{False}\PY{p}{)}
\end{Verbatim}
\end{tcolorbox}
\end{changemargin}

\titleparagraph{Parameter specification}
Given two input algorithms,
\lin{} computes the transfer functions
and can compare them to detect equivalence and other relations.
Some algorithms are equivalent or related only when the parameters satisfy a certain
condition: for example, DR and ADMM.
If the transfer functions of each algorithm use different parameters,
\lin{} form symbolic equations and solve the equations
to determine conditions that, if satisfied by the algorithm parameters,
yield the desired relation between the algorithms;
see \cref{eq3} in \cref{charac-oracle}.

\titleparagraph{Oracles and function}
Oracles play the starring role in our framework: oracle equivalence
is possible only if two algorithms share the same oracles.
In \lin{}, we provide two approaches to declare and add oracles to an algorithm.
The black-box approach is to define oracles as black boxes.
When parsing the algorithm, the system treats each oracle as a distinct entity
unrelated to any other oracle.
An oracle declared using syntax \texttt{add\_oracle} uses the black-box approach.
For example, we may
add oracles $\nabla f$ and $\textnormal{prox}_g$ to algorithm \texttt{algo1}:
\begin{changemargin}{1cm}{1cm}
\begin{tcolorbox}[breakable, size=fbox, boxrule=1pt, pad at break*=1mm,colback=cellbackground, colframe=cellborder]
\begin{Verbatim}[commandchars=\\\{\}]
\PY{c+c1}{\PYZsh{} add oracle gradient of f in the first approach}
\PY{n}{gradf} \PY{o}{=} \PY{n}{algo1}\PY{o}{.}\PY{n}{add\PYZus{}oracle}\PY{p}{(}\PY{l+s+s2}{\PYZdq{}}\PY{l+s+s2}{gradf}\PY{l+s+s2}{\PYZdq{}}\PY{p}{)}
\PY{c+c1}{\PYZsh{} add oracle prox of g in the first approach}
\PY{n}{proxg} \PY{o}{=} \PY{n}{algo1}\PY{o}{.}\PY{n}{add\PYZus{}oracle}\PY{p}{(}\PY{l+s+s2}{\PYZdq{}}\PY{l+s+s2}{proxg}\PY{l+s+s2}{\PYZdq{}}\PY{p}{)}
\end{Verbatim}
\end{tcolorbox}
\end{changemargin}

The functional approach is to define oracles in terms of the (sub)gradient of a function.
When parsing an algorithm, all the oracles will be decomposed into (sub)gradients and
the state-space realization given in terms of (sub)gradients.
We say that two algorithms are oracle-equivalent in terms of functional oracles
if they are oracle-equivalent after rewriting the algorithm
to use only (sub)gradient oracles.
This approach is critical to allow us to identify algorithm conjugation,
since conjugate algorithms use different (conjugate) oracles.
If every algorithm is represented in terms of (sub)gradients,
algorithm conjugation can be detected using \cref{prop8}.
Fortunately, common oracles such as prox and argmin
can be easily written in terms of (sub)gradients:
for example, $\prox_f(x) = (I - \partial f)^{-1}(x)$ and argmin as \cref{eq58}.

\mnote{Have we defined what it means to use an argmin as an oracle, or given any examples? Please add one.}
\snote{In \cref{charac-shift}, there is an example of ADMM using argmin oracle, is that enough, or shall we provide another example here?
	Or we provide a more explicit way to define argmin in the proof of black-box vs functional, might add a reference?}

To use the functional approach, users must define and add functions
to the algorithm first using \texttt{add\_function}
and then declare and add oracles.
The following code shows how to use the functional approach to declare and add oracles $\nabla f$ and $\textnormal{prox}_f$.
\begin{changemargin}{1cm}{1cm}
\begin{tcolorbox}[breakable, size=fbox, boxrule=1pt, pad at break*=1mm,colback=cellbackground, colframe=cellborder]
\begin{Verbatim}[commandchars=\\\{\}]
\PY{c+c1}{\PYZsh{} add function f}
\PY{n}{f} \PY{o}{=} \PY{n}{algo1}\PY{o}{.}\PY{n}{add\PYZus{}function}\PY{p}{(}\PY{l+s+s2}{\PYZdq{}}\PY{l+s+s2}{f}\PY{l+s+s2}{\PYZdq{}}\PY{p}{)}
\PY{c+c1}{\PYZsh{} gradient of f with repect to x1}
\PY{n}{lin}\PY{o}{.}\PY{n}{grad}\PY{p}{(}\PY{n}{f}\PY{p}{)}\PY{p}{(}\PY{n}{x1}\PY{p}{)}
\PY{c+c1}{\PYZsh{} prox of f with repect to x2 and parameter t}
\PY{n}{lin}\PY{o}{.}\PY{n}{prox}\PY{p}{(}\PY{n}{f}\PY{p}{,}\PY{n}{t}\PY{p}{)}\PY{p}{(}\PY{n}{x2}\PY{p}{)}
\end{Verbatim}
\end{tcolorbox}
\end{changemargin}

\subsection{Black-box vs functional oracles}
Are two algorithms equivalent with respect to black-box oracles
if and only if they are equivalent with respect to functional oracles?
Intuitively, when oracles are defined
in terms of (sub)gradients,
it might be possible to identify more relations with other algorithms.
However, as stated in \cref{prop12}, for algorithms
that use only proximal operators, argmins, and (sub)gradients as oracles,
equivalence is preserved under both black-box and functional definitions of oracles.

\begin{proposition}\label{prop12}
	Suppose two algorithms
	use only proximal operators, argmins, and (sub)gradients as oracles.
	Then the two algorithms are equivalent with respect to black-box oracles if and only if
	they are also equivalent with respect to functional oracles.
	
\end{proposition}

\begin{proof}
	Since for any function $g$ and any $t$,
	$\textnormal{prox}_{tg}(x) = \textnormal{argmin}_y\{tg(y) + \frac{1}{2} \|x - y\|^2\}$,
	we can treat proximal operator as a special case of argmin.
	Without loss of generality, any argmin oracle in a linear algorithm has the form
	\begin{displaymath}
		z = \textnormal{argmin}_x \left\{ \lambda g(x) + \frac{1}{2}\left[\begin{array}{c} x \\ y \end{array} \right]^T
		\left[\begin{array}{cc}	Q_{11} & Q_{12}\\ Q_{21} & Q_{22}
		\end{array} \right] \left[\begin{array}{c}	x \\ y \end{array} \right]\right\}.
	\end{displaymath}
	Here $z$ is the value of the oracle and $y$ can be regarded as the argument,
	which means from the perspective of a linear system,
	$z$ is the input and $y$ is the output.
	The parameter $\lambda$ can be a scalar or matrix, $g$ is a function,
	and $Q_{11}$, $Q_{12}$, $Q_{21}$, $Q_{22}$ are parameter matrices.
	Specifically,
	\begin{displaymath}
		\left[\begin{array}{cc}	Q_{11} & Q_{12}\\ Q_{21} & Q_{22}
		\end{array} \right]
	\end{displaymath}
	is a symmetric matrix and
	\begin{displaymath}
		\frac{1}{2}\left[\begin{array}{c} x \\ y \end{array} \right]^T
		\left[\begin{array}{cc}	Q_{11} & Q_{12}\\ Q_{21} & Q_{22}
		\end{array} \right] \left[\begin{array}{c}	x \\ y \end{array} \right]
	\end{displaymath}
	is a quadratic term with respect to $x$ and $y$.
	The matrix $Q_{11}$ must be invertible if the argmin oracle is single-valued.
	To recover the proximal operator,
	choose a scalar $\lambda$ and set
	\begin{displaymath}
		\left[\begin{array}{cc}	Q_{11} & Q_{12}\\ Q_{21} & Q_{22}
		\end{array} \right] = \left[\begin{array}{cc}	I & -I\\ -I & I
		\end{array} \right].
	\end{displaymath}
	If $g$ is a convex function, the argmin oracle can be
	written in terms of the subgradient oracle $\partial g$
	as follows,
	
	\begin{equation}\label{eq58}
		z \in -Q_{11}^{-1}\lambda \partial g (z) - Q_{11}^{-1}Q_{12}y.
	\end{equation}
	
	Suppose we have an algorithm with $n+m$ oracles in total,
	consisting of $n$ argmins and $m$ (sub)gradients.
	We can group the argmins and the (sub)gradients together respectively
	and partition the state-space realization accordingly as
	\begin{equation}\label{eq54}
		\left[\begin{array}{c:c c}
			A & B_1 & B_2\\
			\hdashline
			C_1 & D_{11}& D_{12} \\
			C_2 & D_{21} & D_{22}
		\end{array}\right],
	\end{equation}
	where $C_1$, $B_1$ correspond to the argmins, $C_2$, $B_2$ correspond to the (sub)gradients,
	and $D$ is partitioned accordingly into $D_{11}$, $D_{12}$, $D_{21}$, and $D_{22}$.
	The transfer function can be represented accordingly as
	\begin{displaymath}
		\hat H(z) = \left[\begin{array}{c c}
			\hat H_{11}(z) & \hat H_{12}(z) \\
			\hat H_{21}(z) & \hat H_{22}(z) \\
		\end{array}\right] = \left[\begin{array}{c c}
			C_1(zI - A)^{-1}B_1 + D_{11} & C_1(zI - A)^{-1}B_2 + D_{12} \\
			C_2(zI - A)^{-1}B_1 + D_{21} & C_2(zI - A)^{-1}B_2 + D_{22} \\
		\end{array}\right].
	\end{displaymath}
	The input and output are partitioned as $(\bar{u}_1, \bar{u}_2)$ and $(\bar{y}_1, \bar{y}_2)$,
	where $\bar{y}_1 = (y_1, \dots, y_n)$, $\bar{y}_2 = (y_{n+1}, \dots, y_{n+m})$,
	$\bar{u}_1 = (z_1, \dots, z_n)$, and $\bar{u}_2 = (\nabla f_{n+1}(y_{n+1}), \dots, \nabla f_{n+m}(y_{n+m}))$.
	For each $i \in \{1, \dots, n\}$ we have
	\begin{equation}\label{eq55}
		z_i = \textnormal{argmin}_x \left\{ \lambda_i f_i(x) + \frac{1}{2}\left[\begin{array}{c} x \\ y_i \end{array} \right]^T
		\left[\begin{array}{cc}	Q_{11}^i & Q_{12}^i\\ Q_{21}^i & Q_{22}^i
		\end{array} \right] \left[\begin{array}{c}	x \\ y_i \end{array} \right]\right\}
	\end{equation}
	where $Q_{11}^i$ is invertible for any $i \in \{1, \dots, n\}$.
	
	Now we rewrite the linear system so that the nonlinearities
	corresponding to the argmins for the new linear system are (sub)gradients.
	Let $\lambda = \textnormal{diag}(\lambda_{1}, \dots, \lambda_n)$,
	$Q_1 = \textnormal{diag}(Q^1_{11}, \dots, Q_{11}^n)$,
	$Q_2 = \textnormal{diag}(Q^1_{12}, \dots, Q_{12}^n)$,
	and $M_1 = Q_1^{-1}Q_2 $, and $M_2 = Q_1^{-1}\lambda$.
	The new state-space realization in terms of the (sub)gradient oracles is
	\begin{equation}\label{eq34}
		\left[\begin{array}{c:c c }
			A - B_1(I + M_1D_{11})^{-1}M_1C_1 & -B_1(I + M_1D_{11})^{-1}M_2
			&  B_2- B_1(I + M_1D_{11})^{-1}M_1D_{12} \\
			\hdashline
			-(I + M_1D_{11})^{-1}M_1C_1 & -(I + M_1D_{11})^{-1}M_2
			& -(I + M_1D_{11})^{-1}M_1D_{12} \\
			C_2-D_{21}(I + M_1D_{11})^{-1}M_1C_1 & -D_{21}(I + M_1D_{11})^{-1}M_2
			& D_{22}-D_{21}(I + M_1D_{11})^{-1}M_1D_{12}   \\
		\end{array}\right].
	\end{equation}
	We can compute the transfer function as
	\small
	\begin{equation}\label{eq35}
		\hat H'(z) = \left[\begin{array}{c c}
			\hat H'_{11}(z) & \hat H'_{12}(z) \\
			\hat H'_{21}(z) & \hat H'_{22}(z) \\
		\end{array}\right] = \left[\begin{array}{c c}
			-(I + M_1\hat H_{11}(z))^{-1}M_2 & -(I + M_1 \hat H_{11}(z))^{-1}M_1\hat H_{12}(z) \\
			-\hat H_{21}(z)(I +M_1\hat H_{11}(z))^{-1}M_2 & \hat H_{22}(z)-\hat H_{21}(z)(I +M_1\hat H_{11}(z))^{-1}M_1\hat H_{12}(z) \\
		\end{array}\right].
	\end{equation}
	\normalsize
	Note that $I + M_1D_{11}$ is invertible
	(otherwise the algorithm is not causal) and
	consequently $I +M_1\hat H_{11}(z)$ is invertible.
	The matrix $Q_1$ is also invertible, since $Q^i_{11}$ is invertible for any $i \in \{1, \dots, n\}$.
	A detailed proof of \cref{eq34} and \cref{eq35} is provided in \cref{proof-oracles}.
	Therefore, we know that if $\hat H(z)$ is fixed then $\hat H'(z)$ is also fixed.
	\mnote{Could end proof here; is it clear enough?}
	\snote{Yes, ending here should be clear. Then, if we end here, we might change the previous sentence as
		``if $\hat H(z)$ is fixed then $\hat H'(z)$ is also fixed and vice versa"?}
\end{proof}

\section{Proof of (4.14)}\label{proof-inversesystem}
Since $\bu$ and $\by$ are partitioned as $\bu = (\bu_1, \bu_2)$ and $\by = (\by_1, \by_2)$,
the state-space realization can be partitioned accordingly as
\begin{displaymath}
	\left[\begin{array}{c:c c}
		A & B_1 & B_2 \\
		\hdashline
		C_1 & D_{11} & D_{12} \\
		C_2 & D_{21} & D_{22}
	\end{array}\right].
\end{displaymath}
We can express the transfer function $\hat H(z)$ as
\begin{displaymath}
	\hat H(z) =  \left[\begin{array}{c c}
		\hat H_{11}(z) & \hat H_{12}(z) \\
		\hat H_{21}(z) & \hat H_{22}(z) \\
	\end{array}\right]=\left[\begin{array}{c c}
		C_1(zI - A)^{-1}B_1 + D_{11} & C_1(zI - A)^{-1}B_2 + D_{12} \\
		C_2(zI - A)^{-1}B_1 + D_{21} & C_2(zI - A)^{-1}B_2 + D_{22} \\
	\end{array}\right].
\end{displaymath}
The system equations show as
\begin{equation}\label{eq50}
	\begin{aligned}
		x^{k+1} & = Ax^k + B_1u^k_1 + B_2u^k_2 \\
		y^k_1 & = C_1x^k + D_{11}u^k_1 + D_{12}u^k_2 \\
		y^k_2 & = C_2x^k + D_{21}u^k_1 + D_{22}u^k_2. \\
	\end{aligned}
\end{equation}
As we invert the input-output map corresponding to $\by_1$ and $\bu_1$,
the input of this system becomes $(y^k_1, u^k_2)$ and
the output is $(u^k_1, y^k_2)$ at time $k$.
From \cref{eq50}, as $D_{11}$ is invertibe, we have
\begin{equation*}
	u^k_1 = -D_{11}^{-1}C_1x^k + D_{11}^{-1}y^k_1 - D_{11}^{-1}D_{12}u^k_2.
\end{equation*}
The new system equations change to
\begin{equation*}
	\begin{aligned}
		x^{k+1} & = (A-B_1D_{11}^{-1}C_1)x^k + B_1D_{11}^{-1}y^k_1 + (B_2-B_1D_{11}^{-1}D_{12})u^k_2 \\
		u^k_1 & = -D_{11}^{-1}C_1x^k + D_{11}^{-1}y^k_1 - D_{11}^{-1}D_{12}u^k_2 \\
		y^k_2 & = (C_2 - D_{21}D_{11}^{-1}C_1)x^k + D_{21}D_{11}^{-1}y^k_1 + (D_{22}- D_{21}D_{11}^{-1}D_{12})u^k_2,
	\end{aligned}
\end{equation*}
which correspond to state-space realization
\begin{equation}\label{eq31}
	\left[
	\begin{array}{c:c c}
		A-B_1D_{11}^{-1}C_1 & B_1D_{11}^{-1} & B_2-B_1D_{11}^{-1}D_{12} \\
		\hdashline
		-D_{11}^{-1}C_1 & D_{11}^{-1} & - D_{11}^{-1}D_{12} \\
		(C_2 - D_{21}D_{11}^{-1}C_1) & D_{21}D_{11}^{-1}  & D_{22}- D_{21}D_{11}^{-1}D_{12}
	\end{array}
	\right].
\end{equation}

To calculate the transfer function, note that
\begin{align*}
	(zI - A + B_1D_{11}^{-1}C_1)^{-1} &= (zI -A)^{-1} + (zI -A)^{-1}B_1(-D_{11} - C_1(zI-A)^{-1}B_1)^{-1}C_1(zI-A)^{-1}\\
	&= (zI -A)^{-1} - (zI -A)^{-1}B_1\hat H_{11}^{-1}(z)C_1(zI-A)^{-1}.
\end{align*}
We have
\begin{align*}
	\hat H'_{11}(z) & = -D_{11}^{-1}C_1((zI -A)^{-1} - (zI -A)^{-1}B_1\hat H_{11}^{-1}(z)C_1(zI-A)^{-1})B_1D_{11}^{-1} + D_{11}^{-1} \\
	& = -D_{11}^{-1}(\hat H_{11}(z) - D_{11} - (\hat H_{11}(z) - D_{11})\hat H_{11}^{-1}(z)(\hat H_{11}(z) - D_{11}))D_{11}^{-1} + D_{11}^{-1} \\
	& = -D_{11}^{-1}(\hat H_{11}(z) - D_{11})(I - \hat H_{11}^{-1}(z)(\hat H_{11}(z) - D_{11}))D_{11}^{-1} + D_{11}^{-1} \\
	& = -D_{11}^{-1}(\hat H_{11}(z) - D_{11})\hat H_{11}^{-1}(z)+ D_{11}^{-1} \\
	& = \hat H_{11}^{-1}(z)\\
	\hat H'_{12}(z) & = -D_{11}^{-1}C_1((zI -A)^{-1} - (zI -A)^{-1}B_1\hat H_{11}^{-1}(z)C_1(zI-A)^{-1})B_2 - \hat H_{11}^{-1}(z)D_{12} \\
	& = -D_{11}^{-1}(\hat H_{12}(z) - D_{12} - (\hat H_{11}(z) - D_{11})\hat H_{11}^{-1}(z)(\hat H_{12}(z) - D_{12})) - \hat H_{11}^{-1}(z)D_{12} \\
	& = -D_{11}^{-1}(I - (\hat H_{11}(z) - D_{11})\hat H_{11}^{-1}(z))(\hat H_{12}(z) - D_{12}) - \hat H_{11}^{-1}(z)D_{12} \\
	& = -\hat H_{11}^{-1}(z)(\hat H_{12}(z) - D_{12})- \hat H_{11}^{-1}(z)D_{12} \\
	& = -\hat H_{11}^{-1}(z)\hat H_{12}(z)\\
	\hat H'_{21}(z) & = C_2((zI -A)^{-1} - (zI -A)^{-1}B_1\hat H_{11}^{-1}(z)C_1(zI-A)^{-1})B_1D_{11}^{-1} + D_{21}\hat H_{11}^{-1}(z) \\
	& =(\hat H_{21}(z) - D_{21} - (\hat H_{21}(z) - D_{21})\hat H_{11}^{-1}(z)(\hat H_{11}(z) - D_{11}))D_{11}^{-1} + D_{21}\hat H_{11}^{-1}(z) \\
	& = (\hat H_{21}(z) - D_{21})(I - \hat H_{11}^{-1}(z)(\hat H_{11}(z) - D_{11}))D_{11}^{-1} + D_{21}\hat H_{11}^{-1}(z) \\
	& = (\hat H_{21}(z) - D_{21})\hat H_{11}^{-1}(z)+ D_{21}\hat H_{11}^{-1}(z) \\
	& = \hat H_{21}(z)\hat H_{11}^{-1}(z)\\
	\hat H'_{22}(z) & = \hat H_{22}(z) - (\hat H_{21}(z)- D_{21})\hat H_{11}^{-1}(z)(\hat H_{12}(z) - D_{12}) - D_{21}\hat H_{11}^{-1}(z)(\hat H_{12}(z)- D_{12}) \\
	& -(\hat H_{21}(z)- D_{21})\hat H_{11}^{-1}(z)D_{12}- D_{21}\hat H_{11}^{-1}(z)D_{12} \\
	& = \hat H_{22}(z) - \hat H_{21}(z)\hat H_{11}^{-1}(z)\hat H_{12}(z).
\end{align*}
Thus, we get the desired results as \cref{eqp13}.

\section{Proof of proposition 6.1}\label{proof-oracleequivalence}

Given an algorithm $\mathcal A$ with state-space realization $(A, B, C, D)$, 
the relation between the input $u$ and output $y$ can be expressed as 
\begin{equation}\label{input-output}
	y^k= C(A)^kx^0+ \sum_{j=0}^{k-1}C(A)^{k-(j+1)}Bu^j + Du^k.
\end{equation}
Relation \cref{input-output} is obtained by \cref{eqp3}, without the assumption that $x^0 = 0$.
The output $y^k$ is the sum of $C(A)^kx^0$, which is due to the initial condition $x^0$, and
$\sum_{j=0}^{k-1}C(A)^{k-(j+1)}Bu^j + Du^k$, which is due to the inputs $\{u^0,\dots,u^k\}$.
The linearity of $\mathcal A$ ($\mathcal A$ is treated as a linear system) 
allows the decomposition of two contributions and they can be studied separately:
\[
(\text{total response}) = \underbrace{(\text{zero input response})}_{\text{set $u^k=0$ for $k \geq 0$}}
\,+\, \underbrace{(\text{zero state response})}_{\text{set $x^0=0$}}.
\]
Since we would like to characterize $\mathcal A$ with its input-output map, 
we can only focus on the zero state response,
which allows us to avoid details about initialization.
With the definition of impulse response, 
\begin{equation*}
	H^k = \begin{cases}
		D 		  & k = 0 \\
		C(A)^{k-1}B & k \geq 1
	\end{cases}, 
\end{equation*}
we can express the zero state response as 
\begin{equation*}
	y^k = H^k u^0 + H^{k-1} u^1 + \cdots + H^1 u^{k-1} + H^0u^k.
\end{equation*}

Transfer function provides a compact form to represent the $H^k$ series 
by taking the $z$-transform. 
The transfer function $H(z)$ shows as follows, 
\begin{equation*}
	\hat H(z) = \left[\begin{array}{c|c}
		A & B\\
		\hline
		C & D
	\end{array}\right] = D+\sum_{k=1}^{\infty}C(A)^{k-1}Bz^{-k} = C(zI - A)^{-1}B +D.
\end{equation*}
Therefore, we know that algorithm $\mathcal A$ is uniquely characterized by its input-output map, 
and thus uniquely characterized by its impulse response and transfer function. 

From the definition of oracle equivalence, 
oracle-equivalent algorithms have identical sequences of output $\by$ 
for each possible sequence of input $\bu$ if initialized properly. 
Here $\by = \{u^k\}^\infty_{k = 0}$ and $\bu = \{y^k\}^\infty_{k = 0}$. 
Thus, they must have identical impulse responses and consequently identical transfer functions. 
This completes the proof. 

\section{Proof of proposition 7.1}\label{proof-cyclicpermutation}

\renewcommand{\thealgorithm}{D.\arabic{algorithm}}
\setcounter{algorithm}{0}
\begin{algorithm}[H]
	\centering
	\caption{General form of algorithm $\mathcal A$}
	\label{algo15}
	\begin{algorithmic}
		\FOR{$k=0, 1, 2,\ldots$}
		\STATE{$x^{k+1}_1 = L_1(x^{k}_1, \dots, x^{k}_m)$}
		\STATE{$x^{k+1}_2 = L_2(x^{k+1}_1, x^k_2, \dots, x^{k}_m)$}
		\STATE{$\vdots$}
		\STATE{$x^{k+1}_{i} = L_i(x^{k+1}_1, \dots, x^{k+1}_{i-1}, x^{k}_{i}, \dots, x^{k}_m, u^{k+1}_{1})$}
		\STATE{$\vdots$}
		\STATE{$x^{k+1}_{\tilde i} = L_{\tilde i}(x^{k+1}_1, \dots, x^{k+1}_{\tilde i -1}, x^{k}_{\tilde i}, \dots, x^{k}_m, u^{k+1}_{n})$}
		\STATE{$\vdots$}
		\STATE{$x^{k+1}_m = L_m(x^{k+1}_1, \dots, x^{k+1}_{m-1}, x^{k}_m)$}
		\ENDFOR
	\end{algorithmic}
\end{algorithm}

First, we prove that cyclic permutation implies shift equivalence. 
Without loss of generality, we can express algorithm $\mathcal A$
in the general form as \cref{algo15}.
Since $\mathcal A$ is an linear algorithm, $L_1, \dots, L_m$ are linear functions.
Given an initialization $\{x^{0}_1, \dots, x^{0}_m\}$,
$\mathcal{A}$ generates
state sequence $(x^k_1, \ldots, x^k_m)_{k\ge 0}$,
input sequence $(u^k_1, \ldots,u^k_n)_{k\ge 1}$,
and output sequence $(y^k_1, \ldots, y^k_n)_{k\ge 1}$.
The $i$th update equation is the first update equation that
contains an oracle call, corresponding to $u^k_1$ and $y^k_1$.
The $\tilde i$th update equation is the last update equation that
contains an oracle call, corresponding to $u^k_n$ and $y^k_n$.
The outputs are also linear functions of the states.
Specifically, we have
\begin{displaymath}
	\begin{aligned}
		y^k_1 = & Y_1(x^{k+1}_1, \dots, x^{k+1}_{i-1}, x^{k}_{i}, \dots, x^{k}_m),\ \ u^k_1 = \phi_1(y^k_1)\\
		& \vdots \\
		y^k_n = & Y_n(x^{k+1}_1, \dots, x^{k+1}_{\tilde i-1}, x^{k}_{\tilde i}, \dots, x^{k}_m),\ \ u^k_n = \phi_n(y^k_n).
	\end{aligned}
\end{displaymath}
Functions $Y_1, \dots, Y_n$ are linear functions and $\phi_1, \dots, \phi_n$ denote the oracle calls.
Without loss of generality, suppose permutation $\tilde \pi = (\tilde l+1,\ldots, m, 1, \ldots , \tilde l)$
with $1 < \tilde l < m$.

First case. Suppose the new order of oracle calls within one iteration
is a cyclic permutation $\pi$ of $(n)$ (not identical to $(n)$).
Without loss of generality, suppose $\pi = (j+1,\ldots, n, 1, \ldots , j)$ with $1 < j < n$,
the $j$th oracle call corresponds to the $\tilde j$th update equation,
and the $j+1$th oracle call corresponds to the $\tilde p$th update equation.
By definition, we have $i \le \tilde j < \tilde l +1 \le \tilde p \le \tilde i$.
At the first time step $k = 1$, the first input is $u^1_1$ and the first output is $y^1_1$,
and the $j+1$th input and output are $u^1_{j+1}$ and $y^1_{j+1}$.
We have
\begin{displaymath}
	\begin{aligned}
		y^1_1 & = Y_1(x^{1}_1, \dots, x^{1}_{i-1}, x^{0}_{i}, \dots, x^{0}_m) \\
		x^1_{\tilde l+1} & = L_{\tilde l+1}(x^{1}_1, \dots, x^{1}_{\tilde l}, x^{0}_{\tilde l+1}, \dots, x^{0}_m) \\
		y^1_{j+1} & = Y_1(x^{1}_1, \dots, x^{1}_{\tilde p-1}, x^{0}_{\tilde p}, \dots, x^{0}_m).
	\end{aligned}
\end{displaymath}
Here without loss of generality, suppose the $\tilde l+1$th update equation does not contain an oracle call.
In other words, $\tilde j < \tilde l +1 < \tilde p$.
By definition, $\mathcal{B}$ calls the update equations in the order $\tilde \pi$.
At the first time step, the $\tilde l +1$th update equation is first called.
If $\mathcal{B}$ is suitably initialized
with states $\{x^{1}_1, \dots, x^{1}_{\tilde l}, x^{0}_{\tilde l+1}, \dots, x^{0}_m\}$,
it will generate
state sequence $(x^{k}_{\tilde l+1}, \ldots, x^k_m, x^{k+1}_1, \dots, x^{k+1}_{\tilde l})_{k\ge 0}$,
input sequence
$(u^k_{j+1},\ldots,u^k_n, u^{k+1}_1, \ldots, u^{k+1}_{j})_{k\ge 1}$,
and output sequence $(y^k_{j+1},\ldots, y^k_n, y^{k+1}_1, \ldots, y^{k+1}_{j})_{k\ge 1}$.
The input and output sequences of $\mathcal{A}$ and $\mathcal{B}$ match up to prefixes
$(u^1_1, \ldots, u^1_{j})$ and $(y^1_1, \ldots, y^{1}_j)$ respectively.
Therefore, $\mathcal{A}$ and $\mathcal{B}$ are shift-equivalent.

Second case. Suppose the order of oracle calls within one iteration remain unchanged (identical to $(n)$).
By definition, we have $1 < \tilde l +1 \le i$.
We have
\begin{displaymath}
	\begin{aligned}
		x^1_{\tilde l+1} & = L_{\tilde l+1}(x^{1}_1, \dots, x^{1}_{\tilde l}, x^{0}_{\tilde l+1}, \dots, x^{0}_m) \\
		y^1_1 & = Y_1(x^{1}_1, \dots, x^{1}_{i-1}, x^{0}_{i}, \dots, x^{0}_m).
	\end{aligned}
\end{displaymath}
Here without loss of generality, suppose the $\tilde l+1$th update equation does not contain an oracle call.
In other words, $\tilde l +1 < i$.
By definition, $\mathcal{B}$ calls the update equations in the order $\tilde \pi$.
At the first time step, the $\tilde l +1$th update equation is first called.
If $\mathcal{B}$ is suitably initialized
with states $\{x^{1}_1, \dots, x^{1}_{\tilde l}, x^{0}_{\tilde l+1}, \dots, x^{0}_m\}$,
it will generate
state sequence $(x^{k}_{\tilde l+1}, \ldots, x^k_m, x^{k+1}_1, \dots, x^{k+1}_{\tilde l})_{k\ge 0}$.
The input and output sequences remain unchanged.
Therefore, $\mathcal{A}$ and $\mathcal{B}$ are oracle-equivalent.
Meanwhile, since oracle equivalence can be regarded as a special case of shift equivalence,
$\mathcal{A}$ and $\mathcal{B}$ are also shift-equivalent.

Next, we prove that shift equivalence implies cyclic permutation.
Suppose algorithms $\mathcal A$ and $\mathcal B$ are shift-equivalent.
If they are also oracle-equivalent, then they can be written using the same set of update equations, which is trivially related by a cyclic permutation (where the permutation is the identity).
Now suppose they are not oracle-equivalent.
Let $(u^k_1, \dots, u^k_{m})_{k \ge 0}$ and
$(\tilde u^k_1, \dots, \tilde u^k_{m})_{k \ge 0}$ be
the input sequences for $\mathcal A$ and $\mathcal B$.
The input sequences match up to a non-empty prefix.
Without loss of generality, suppose the length of this prefix is $q$: that is,
if we remove a prefix of length $q$ from the input
sequence of $\mathcal A$, then $\mathcal A$ and $\mathcal B$ have the same input sequence.
Recall we only need to consider the case $q < m$. If $q > m$, it is equivalent to consider
$q = q \bmod m$.
Comparing the input sequences of $\mathcal A$ and $\mathcal B$, and using the prefix length $q$,
we can write $(u^{k-1}_{q+1}, \dots, u^{k-1}_{m}, u^k_{1}, \dots, u^k_{q})=(\tilde u^k_1, \dots, \tilde u^k_{m})$
for $k \ge 1$.
The output sequences of $\mathcal A$ and $\mathcal B$ have the same relation.
Therefore, $\mathcal B$ and this shifted version of $\mathcal A$ are oracle equivalent,
and so we can write $\mathcal B$ and this shifted version of $\mathcal A$
using the same set of update equations. To undo the shift of $\mathcal A$,
we simply move the first $q$ update equations to the end of the algorithm.

\section{Proof of proposition 7.3}\label{proof-shiftequivalence}

The state-space realization of $\mathcal A$ corresponds to the state update equations
\begin{equation}\label{eq9}
	\begin{aligned}
		x^{k+1} & = Ax^k + B_1\bar{u}^k_1 + B_2\bar{u}^k_2 \\
		\bar{y}^k_1 & = C_1x^k + D_{11}\bar{u}^k_1 + D_{12}\bar{u}^k_2 \\
		\bar{y}^k_2 & = C_2x^k + D_{21}\bar{u}^k_1 + D_{22}\bar{u}^k_2.
	\end{aligned}
\end{equation}
Sufficiency. We will derive the state-space realization of $P_{\pi}\mathcal{A}$:
\begin{equation*}
	\left[\begin{array}{c c :c c}
		A & B_1 & 0 & B_2 \\
		0 & 0 & I & 0\\
		\hdashline
		C_1A & C_1B_1 & D_{11} & C_1B_2 \\
		C_2 & D_{21} & 0 & D_{22}
	\end{array}\right].
\end{equation*}
To verify this realization is correct,
we can write the system equations of this state-space realization as
\begin{equation}\label{eq13}
	\begin{aligned}
		x^{k+1} & = Ax^k + B_1\bar{u}^k_1 + B_2\bar{u}^k_2 \\
		\bar{u}^{k+1}_1 & = \bar{u}^{k+1}_1 \\
		\bar{y}^{k+1}_1 & = C_1Ax^k + C_1B_1\bar{u}^k_1 +D_{11}\bar{u}^{k+1}_1 +  C_1B_2\bar{u}^k_2 \\
		\bar{y}^k_2 & = C_2x^k + D_{21}\bar{u}^k_1 + D_{12}\bar{u}^k_2. \\
	\end{aligned}
\end{equation}
Note that equations \cref{eq13} are the results of equations \cref{eq9} after
applying permutation $\pi$.
As we perform cyclic permutation $\pi$, within each iteration,
the update order of the oracles is shifted as $(j+1,\ldots, n, 1,\ldots, j)$,
indicating oracles $(j+1, \ldots, n)$ are updated before $(1, \ldots, j)$.
Further, the input and output sequences within one iteration at time step $k$
become $(\bar{u}^k_{2}, \bar{u}^{k+1}_1)$ and $(\bar{y}^k_{2}, \bar{y}^{k+1}_1)$.
From the state-space realization, we may compute the transfer function as
\begin{equation}\label{eq14b}
	\hat H_{\mathcal{B}}(z) = \left[\begin{array}{c c}
		C_1(zI - A)^{-1}B_1 + D_{11} & zC_1(zI - A)^{-1}B_2 \\
		C_2(zI - A)^{-1}B_1/z + D_{21}/z & C_2(zI - A)^{-1}B_2 + D_{22} \\
	\end{array}\right] = \left[\begin{array}{c c}
		\hat H_{11}(z) & z\hat H_{12}(z) \\
		\hat H_{21}(z)/z & \hat H_{22}(z) \\
	\end{array}\right].
\end{equation}
To arrive at \cref{eq14b}, we have used the fact that $D_{12} = 0$ by assumption, and
\begin{equation*}
	\left( zI - \left[\begin{array}{c c}
		A & B_1 \\
		0 & 0 \\
	\end{array}\right]\right)^{-1} = \left[\begin{array}{c c}
		(zI - A)^{-1} & \frac{1}{z}(zI - A)^{-1}B_1 \\
		0 &  \frac{1}{z}I \\
	\end{array}\right].
\end{equation*}

Necessity is provided by \cref{prop1}.
Equivalent algorithms must have identical transfer functions.
Thus, if we find an algorithm and its transfer function is the same as \cref{eq11},
it must be equivalent to $\mathcal{B}$.

\section{Discussions on permutation and its generalization}\label{timedelay}
To take a revisit of \cref{prop4},
it can be found that as algorithm $\mathcal{A}$ is permuted
to make the order of oracle calls within one iteration
as $(j+1,\ldots, n, 1,\ldots, j)$ from $(1,\ldots, n)$,
the resulting transfer function is exactly the same as adding a one-step time delay
to channels (oracles) $(1,\ldots, j)$ according to results in control theory.
Another interpretation of adding a one-step time delay comes from the system equations \cref{eq13}.
We can see that the input and output corresponding to channels (oracles) $(1,\ldots, j)$ are
the input and output for the next time step $\bar{u}^{k+1}_1$ and $\bar{y}^{k+1}_1$,
however, the input and output of channels (oracles) $(j+1, \ldots, n)$ are still
the ones for the current time step $\bar{u}^{k}_2$ and $\bar{y}^{k}_2$.
Intrinsically, after cyclic permutation,
the intrinsic update order of oracles does not change,
but a one-step time delay is added to the oracles
that we would like to update latterly.

Using the idea of time delay,
we can generalize algorithm permutation as
adding any step of time delay to any channel (oracle) of an algorithm.
Suppose we add time delay to oracle $i$ of algorithm $\mathcal{A}$ by $d_i$ for any $i \in (n) $,
where $d_i$ can be any integer,
the resulting algorithm $\mathcal{B}$ has transfer function $\hat H_{\mathcal{B}}(z)$ as
\begin{equation}\label{eq16}
	\hat H_{\mathcal{B}}(z) = \begin{bmatrix}
		z^{d_1}&  &  & \\
		& z^{d_2} & & \\
		&  & \ddots  & \\
		&  &  & z^{d_n}
	\end{bmatrix} \hat H_{\mathcal{A}}(z) \begin{bmatrix}
		z^{-d_1}&  &  & \\
		& z^{-d_2} & & \\
		&  & \ddots  & \\
		&  &  & z^{-d_n}
	\end{bmatrix},
\end{equation}
where $\hat H_{\mathcal{A}}(z)$ is the transfer function of $\mathcal{A}$.

To be more specific, suppose we add time delay $d_i$
to oracle $i$ for algorithm $\mathcal{A}$,
$h^{kl}_{\mathcal{A}}(z)$ with $1\leq k \leq n$ and $1\leq l \leq n$
denotes the entry of $\hat H_{\mathcal{A}}(z)$.
The transfer function of the resulting algorithm $\mathcal{B}$
can be expressed entrywise as
\begin{equation}\label{eq17}
	h^{kl}_{\mathcal{B}}(z) = \left\{\begin{matrix*}[l]
		h^{ii}_{\mathcal{A}}(z) & k=i\ l=i\\
		h^{il}_{\mathcal{A}}(z)z^{d_i} & k=i\ l\neq i\\
		h^{ki}_{\mathcal{A}}(z)z^{-d_i}& k\neq i\ l=i\\
		h^{kl}_{\mathcal{A}}(z) & k\neq i \ l\neq i
	\end{matrix*}\right.
\end{equation}
In this way, we know that \cref{prop4} is a special case of \cref{eq16}
with $d_1 = \dots = d_{j} = 1$.

However, there are restrictions so that we cannot add
any arbitrary step of time delay to any oracle.
From \cref{control}, transfer functions are rational (matrix) functions
with respect to $z$.
Further, the rational functions must be proper in order to make the transfer function realizable.
From \cref{eq17}, as we add time delay $d_i$ to oracle $i$ for $\mathcal{A}$,
the off-diagonal entries in the $i$th row of $\hat H_{\mathcal{A}}(z)$
are multiplied by $z^{d_i}$
and the off-diagonal entries in the $i$th column of $\hat H_{\mathcal{A}}(z)$
are multiplied by $z^{-d_i}$
while the $i$th diagonal entry remains unchanged.
From the perspective of relative degrees,
as relative degree is the difference
between the degree of denominator and the degree of numerator,
the relative degrees of
the off-diagonal entries in the $i$th row are decreased by $d_i$
but the relative degrees of the off-diagonal entries
in the $i$th column are increased by $d_i$.
Suppose the smallest relative degree
among the off-diagonal entries in the $i$th row is $r_i$,
then $d_i$ must satisfy $d_i \leq r_i$ to
maintain properness of the resulting off-diagonal entries in the $i$th row.
Similarly, suppose the smallest relative degree
among the off-diagonal entries of the $i$th column is $c_i$,
then $d_i$ must satisfy $-d_i \leq c_i$ to
maintain properness of the resulting off-diagonal entries in the $i$th column.
In other words, we can add time delay $d_i$ to oracle $i$
only if $-c_i \leq d_i \leq r_i$.
Otherwise, at least one off-diagonal entry
in the $i$th row or the $i$th column is no longer proper,
leading to an invalid transfer function.

For any algorithm with state-space realization $(A, B, C, D)$,
the transfer function is calculated by $C(zI-A)^{-1}B+D$.
Term $C(zI-A)^{-1}B$ is a strictly proper (matrix) function,
where strictly proper means that
the degree of $z$ in the numerator polynomial
is strictly less than the degree of $z$ in the denominator polynomial.
Thus, for any nonzero entry of $D$,
the corresponding entry in the transfer function has relative degree zero.
Take a revisit of cyclic permutation,
for any causal algorithm,
the entries above diagonal of the $D$ matrix must be zero,
especially after necessary reordering.
Thus, the entries above diagonal in the transfer function
have strictly positive relative degrees.
This implies that any cyclic permutation of an algorithm always exists.
Note that before performing cyclic permutation,
we are required to reorder the state-space realization if needed.

Reconsider \cref{algo7,algo8},
in \cref{eq19} and \cref{eq20}, comparing $\hat H_{10}(z)$ to $\hat H_{11}(z)$,
we add a one-step time delay to the first channel.
Term $\frac{1}{z-1}$ in $\hat H_{10}(z)$ is multiplied by $z$ and
term $\frac{2z-1}{z-1}$ is multiplied by $z^{-1}$.
Further, the off-diagonal entry in the first row of $\hat H_{10}(z)$ has relative degree 1
and the off-diagonal entry in the first column of $\hat H_{10}(z)$ has relative degree 0.
Thus, we can only add time delay $d_1 = 1$ to the first oracle of \cref{algo7},
as $ 0 \leq d_1 \leq 1$ to maintain properness.

\section{Proof of proposition 7.4}\label{proof-shiftfixedpoint}

Partition the oracle calls of algorithm $\mathcal A: \mathcal{X} \to \mathcal{X}$ 
into two (nonlinear) oracles $\phi_1$ and $\phi_2$. 
Formally, write the update equations as
\begin{equation}\label{eq38}
	\begin{aligned}
		x^\star & = Ax^\star + B_1\bar{u}^\star_1 + B_2\bar{u}^\star_2 \\
		\bar{y}^\star_1 & = C_1x^\star + D_{11}\bar{u}^\star_1 + D_{12}\bar{u}^\star_2 \\
		\bar{y}^\star_2 & = C_2x^\star + D_{21}\bar{u}^\star_1 + D_{22}\bar{u}^\star_2 \\
		u^\star_1 & = \phi_1(y^\star_1)\\
		u^\star_2 & = \phi_2(y^\star_2).
	\end{aligned}
\end{equation}
Here the cyclic permutation $\pi$ swaps the first and second set of oracle calls.
Then the cyclic permutation $P_{\pi}\mathcal{A}$ converges to fixed point
$(\bar{y}_2^\star, \bar{y}_1^\star, \bar{u}_2^\star, \bar{u}_1^\star, x^\star)$.
To verify this, since $D_{12} = 0$, plugging in the fixed point conditions \cref{eq38} to
the system equations of the shifted algorithm \cref{eq13}, we have
\begin{equation*}
	\begin{aligned}
		x^\star & = Ax^\star + B_1\bar{u}^\star_1 + B_2\bar{u}^\star_2 \\
		\bar{u}^\star_1 & = \bar{u}^\star_1 \\
		\bar{y}^\star_1 & = C_1Ax^\star + C_1B_1\bar{u}^\star_1 +D_{11}\bar{u}^\star_1 +  C_1B_2\bar{u}^\star_2 = C_1x^\star + D_{11}\bar{u}^\star_1\\
		\bar{y}^\star_2 & = C_2x^\star + D_{21}\bar{u}^\star_1 + D_{12}\bar{u}^\star_2 \\
		u^\star_1 & = \phi_1(y^\star_1)\\
		u^\star_2 & = \phi_2(y^\star_2).
	\end{aligned}
\end{equation*}
This completes the proof. 

\section{Proof of shift-equivalence of DR and ADMM continued}\label{dradmmsubgradient}
Suppose the oracles for both DR (\cref{algo7}) and ADMM (\cref{algo8})
are subgradients of $f$ and $g$.
Oracles prox and argmin can be expanded as inclusions involving subgradients.
The update equations of DR and ADMM can be rewritten into formations of
\cref{algo16,algo17} respectively.
Note that the update equations involving subgradients are inclusions. 

\renewcommand{\thealgorithm}{H.\arabic{algorithm}}
\setcounter{algorithm}{0}
\vspace{-1em}
\noindent
\hfil
\begin{minipage}[t]{0.47\textwidth}
	\begin{algorithm}[H]
		\centering
		\caption{DR}
		\label{algo16}
		\begin{algorithmic}
			\FOR{$k=0, 1, 2,\ldots$}
			\STATE{$x^{k+1}_1 \in x^k_3 - t\partial f(x^{k+1}_1)$}
			\STATE{$x^{k+1}_2 \in 2x^{k+1}_1 - x^k_3 - tL^T\partial g(Lx^{k+1}_2)$}
			\STATE{$x^{k+1}_3 = x^k_3 + x^{k+1}_2 - x^{k+1}_1$}
			\ENDFOR
		\end{algorithmic}
	\end{algorithm}
\end{minipage}
\hfil
\begin{minipage}[t]{0.47\textwidth}
	\begin{algorithm}[H]
		\centering
		\caption{ADMM}
		\label{algo17}
		\begin{algorithmic}
			\FOR{$k=0, 1, 2,\ldots$}
			\STATE{$\xi^{k+1}_1 \in L\xi^k_2 - L\xi^k_3 - \frac{1}{\rho}LL^T\partial g(L\xi^{k+1}_1)$}
			\STATE{$\xi^{k+1}_2 \in L^{-1}\xi^{k+1}_1 + \xi^k_3 - \frac{1}{\rho}\partial f(\xi^{k+1}_2) $}
			\STATE{$\xi^{k+1}_3 = \xi^{k}_3 + L^{-1}\xi^{k+1}_1 - \xi^{k+1}_2$}
			\ENDFOR
		\end{algorithmic}
	\end{algorithm}
\end{minipage}
\hfil
\vspace{1em}

We still assume $\rho = 1/t$ in ADMM.
The transfer functions are computed as $\hat H_{18}(z)$ and $\hat H_{19}(z)$ respectively.
Note that $\hat H_{19}(z)$ is not written in the causal order.
\begin{displaymath}
	\hat H_{18}(z) = \left[\begin{array}{ c c c | c c }
		0 & 0 & I & -tI & 0 \\
		0 & 0 & I & -2tI & -tL^T \\
		0 & 0 & I & -tI & -tL^T \\
		\hline
		0 & 0 & I & -tI & 0 \\
		0 & 0 & L & -2tL & -tLL^T \\
	\end{array}\right] = \left[\begin{array}{ c c }
		-\frac{tz}{z-1}I & -\frac{t}{z-1}L^T  \\
		\frac{t - 2tz}{z-1}L & -\frac{tz}{z-1}LL^T  \\
	\end{array}\right]
\end{displaymath}
\begin{displaymath}
	\hat H_{19}(z) = \left[\begin{array}{ c c c | c c }
		0 & L & -L & 0 & -tLL^T \\
		0 & I & 0 & -tI & -tL^T \\
		0 & 0 & 0 & tI & 0 \\
		\hline
		0 & I & 0 & -tI & -tL^T \\
		0 & L & -L & 0 & -tLL^T \\
	\end{array}\right] =  \left[\begin{array}{ c c }
		-\frac{tz}{z-1}I & -\frac{tz}{z-1}L^T  \\
		\frac{t - 2tz}{z(z-1)}L & -\frac{tz}{z-1}LL^T  \\
	\end{array}\right]
\end{displaymath}
From \cref{prop3,prop4},
we know that they are still shift-equivalent.

\section{Proof of proposition 8.1}\label{proof-repetition}

Sufficiency. The update equations of $\mathcal{B}$ can be written as
\begin{equation*}
	\begin{aligned}
		x^k_1 & = Ax^{k}_\mathcal{B}+Bu^k_1 \\
		y^k_1 & = Cx^{k}_\mathcal{B}+Du^{k}_1 \\
		x^{k+1}_\mathcal{B} & = Ax^k_1+Bu^k_2 \\
		y^k_2 & = Cx^{k}_1+Du^{k}_2,
	\end{aligned}
\end{equation*}
where $x^k_1$ is an intermediate state.
Eliminating the intermediate state $x^k_1$,
we arrive at the new update equations:
\begin{equation*}
	\begin{aligned}
		x^{k+1}_\mathcal{B} & = A^2x^{k}_\mathcal{B}+ABu^k_1+Bu^k_2 \\
		y^k_1 & = Cx^{k}_\mathcal{B}+Du^{k}_1 \\
		y^k_2 & = CAx^{k}_\mathcal{B} + CB u^k_1 +Du^{k}_2.
	\end{aligned}
\end{equation*}
The corresponding state-space realization has transfer function
\begin{equation*}
	\left[\begin{array}{c|c c}
		A^2 & AB & B  \\
		\hline
		C & D & 0 \\
		CA & CB & D
	\end{array}\right] =
	\left[\begin{array}{c c}
		C(zI - A^2)^{-1}AB +D & C(zI - A^2)^{-1}B \\
		CA(zI - A^2)^{-1}AB+CB & CA(zI - A^2)^{-1}B+D
	\end{array}\right].
\end{equation*}

Necessity is provided by \cref{prop1} since the transfer function uniquely
characterizes an equivalence class of algorithms.

\section{Proof of proposition 8.3}\label{proof-repetitionfixedpoint}
Suppose the oracles of algorithm $\mathcal{A}: \mathcal{X} \to \mathcal{X}$
can be represented as $\phi: \mathcal{X} \to \mathcal{X}$.
Since $\mathcal{A}$ converges to fixed point $(y^\star, u^\star, x^\star)$, it satisfies
\begin{align*}\label{eq40}
	x^\star & = A x^\star + B u^\star \\
	y^\star & = C x^\star + D u^\star \\
	u^\star & = \phi (y^\star).
\end{align*}
Therefore, we have
\begin{displaymath}
	\begin{aligned}
		x^\star & = A x^\star + B u^\star \\
		& = A(A x^\star + B u^\star) + B u^\star \\
		& = A^2 x^\star + ABu^\star + Bu^\star \\
		& = \dots \\
		& = A^{n-1} x^\star + A^{n-2}B u^\star + \dots + AB u^\star + B u^\star \\
		& = A^n x^\star + A^{n-1}B u^\star + \dots + AB u^\star + B u^\star \\
		y^\star & = C x^\star + D u^\star \\
		& = C(A x^\star + B u^\star) + D u^\star \\
		& = CA x^\star + CBu^\star + Du^\star \\
		& = \dots \\
		& = CA^{n-1} x^\star + CA^{n-2}B u^\star + \dots + CB u^\star + D u^\star .
	\end{aligned}
\end{displaymath}
With \cref{eq25}, we have
\begin{align*}\label{eq41}
	x^\star & = A^n x^\star + A^{n-1}B u^\star + \dots + AB u^\star + B u^\star \\
	y^\star & = C x^\star + D u^\star \\
	y^\star & = CAx^\star+ CBu^\star + Du^\star \\
	\vdots & \\
	y^\star & = CA^{n-1}x^\star+ CA^{n-2}Bu^\star +  \dots  CBu^\star + Du^\star,
\end{align*}
which indicates that $\mathcal{A}^n$ converges to fixed point $(y', u', x^\star)$
with $y' = y^\star \bigotimes \mathbbm{1}^n$ and $u' = u^\star \bigotimes \mathbbm{1}^n$.

\section{Proof of proposition 9.3}\label{proof-conjugation}
Without loss of generality, let the permutation matrix equal to the identity as \cref{prop8}.
To simplify the notations, let
\begin{displaymath}
	\left[
	\begin{array}{c:c c}
		A & B[[n], \kappa] & B[[n], \bar{\kappa}] \\
		\hdashline
		C[\kappa, [n]] & D[\kappa] & D[\kappa, \bar{\kappa}] \\
		C[\bar{\kappa}, [n]] & D[\bar{\kappa}, \kappa] & D[\bar{\kappa}]
	\end{array}
	\right] = \left[\begin{array}{c:c c}
		A & B_1 & B_2 \\
		\hdashline
		C_1 & D_{11} & D_{12} \\
		C_2 & D_{21} & D_{22}
	\end{array}\right],
\end{displaymath}
\begin{displaymath}
	\begin{bmatrix}
		\hat H[\kappa](z) & \hat H[\kappa, \bar{\kappa}](z)\\
		\hat H[\bar{\kappa}, \kappa](z) & \hat H[\bar{\kappa}](z)
	\end{bmatrix} =  \left[\begin{array}{c c}
		\hat H_{11}(z) & \hat H_{12}(z) \\
		\hat H_{21}(z) & \hat H_{22}(z) \\
	\end{array}\right]=\left[\begin{array}{c c}
		C_1(zI - A)^{-1}B_1 + D_{11} & C_1(zI - A)^{-1}B_2 + D_{12} \\
		C_2(zI - A)^{-1}B_1 + D_{21} & C_2(zI - A)^{-1}B_2 + D_{22} \\
	\end{array}\right].
\end{displaymath}
In this way, $(y[\kappa]^\star, y[\bar{\kappa}]^\star, u[\kappa]^\star, u[\bar{\kappa}]^\star, x^\star)$
can be written as $(y_1^\star, y_2^\star, u_1^\star, u_2^\star, x^\star)$, and
$(u[\kappa]^\star, y[\bar{\kappa}]^\star, y[\kappa]^\star, u[\bar{\kappa}]^\star, x^\star)$ can be written as
$(u_1^\star, y_2^\star, y_1^\star, u_2^\star, x^\star)$.

Partition the oracle calls of algorithm $\mathcal{A}: \mathcal{X} \to \mathcal{X}$
into two nonlinear oracles $\phi_1$ and $\phi_2$.
Oracle $\phi_1$ corresponds to the oracle calls in set $\kappa$,
and $\phi_2$ corresponds to the remaining oracle calls.
Since $\mathcal{A}$ converges to fixed point $(y_1^\star, y_2^\star, u_1^\star, u_2^\star, x^\star)$,
it satisfies
\begin{align*}
	x^\star & = Ax^\star + B_1u^\star_1 + B_2u^\star_2 \\
	y^\star_1 & = C_1x^\star + D_{11}u^\star_1 + D_{12}u^\star_2 \\
	y^\star_2 & = C_2x^\star + D_{21}u^\star_1 + D_{22}u^\star_2 \\
	u^\star_1 & = \phi_1 (y^\star_1) \\
	u^\star_2 & = \phi_2 (y^\star_2).
\end{align*}
The state-space realization of $\C_{\kappa}\mathcal A$ is the same as \cref{eq31}. Note that $D_{11}$ is invertible, we have
\begin{align*}
	x^\star & = Ax^\star + B_2u^\star_2 + B_1u^\star_1 \\
	& = Ax^\star + B_2u^\star_2 + B_1(- D_{11}^{-1}C_1x^\star + D_{11}^{-1}y^\star_1 - D_{11}^{-1}D_{12}u^\star_2)\\
	& = (A - B_1D_{11}^{-1}C_1)x^\star + B_1D_{11}^{-1}y^\star_1 + (B_2-B_1D_{11}^{-1}D_{12})u^\star_2 \\
	u^\star_1 & = -D_{11}^{-1}C_1x^\star + D_{11}^{-1}y^\star_1 - D_{11}^{-1}D_{12}u^\star_2 \\
	y^\star_2 & = C_2x^\star + D_{22} u^\star_2 + D_{21}u^\star_2\\
	& = C_2x^\star + D_{22}u^\star_2 + D_{21}(D_{11}^{-1}y^\star_1- D_{11}^{-1}C_1x^\star - D_{11}^{-1}D_{12}u^\star_2 ) \\
	& = (C_2 - D_{21}D_{11}^{-1}C_1)x^\star + D_{21}D_{11}^{-1}y^\star_1 + (D_{22}- D_{21}D_{11}^{-1}D_{12})u^\star_2 \\
	y^\star_1 & = \phi_1^{-1} (u^\star_1) \\
	u^\star_2 & = \phi_2 (y^\star_2).
\end{align*}
Oracle $\phi_1^{-1}$ is the inverse oracle of oracle $\phi_1$.
Therefore, we get the desired results that
algorithm $\C_{\kappa}\mathcal A$ converges to fixed point $(u_1^\star, y_2^\star, y_1^\star, u_2^\star, x^\star)$.

\section{Commutativity between conjugation and cyclic permutation}\label{proof-commutative}
\begin{proposition}\label{prop11}
	Conjugation and cyclic permutation commute.
\end{proposition}
\begin{proof}
	Given an algorithm $\mathcal{A}$ with transfer function $\hat H(z)$.
	Suppose $\kappa$ is a subset of the oracles of $\mathcal{A}$, $D_{\kappa}$ is invertible,
	and $\pi = (m+1,\ldots, n, 1,\ldots, m)$ is an arbitrary cyclic permutation of
	the oracles of $\mathcal{A}$.
	We will show that the transfer functions of $\mathcal{C}_{\kappa}P_{\pi}\mathcal{A}$
	and $P_{\pi}\mathcal{C}_{\kappa}\mathcal{A}$ are identical.
	
	Suppose $\hat H^\star(z)$ is the transfer function of $P_{\pi}\mathcal{A}$, the results in \cref{prop4} can
	be written as
	\begin{equation}\label{eq56}
		\hat H^\star(z) = Q\hat H(z)Q^{-1}.
	\end{equation}
	Here $Q$ is a diagonal matrix where the first $m$ diagonal entries are all $z$ and the rest of the
	diagonal entries are all ones.
	We will use the same settings and notations as \cref{prop8} to express changes in transfer function
	of conjugation $\mathcal{C}_{\kappa}$.
	Without loss of generality,
	the transfer function $\hat H'(z)$ of $\mathcal{C}_{\kappa}\mathcal{A}$ satisfies
	\begin{equation}\label{eq57}
		\hat H'(z) = \left[\begin{array}{c c}
			\hat H_{11}^{-1}(z) & -\hat H_{11}^{-1}(z)\hat H_{12}(z) \\
			\hat H_{21}(z)\hat H_{11}^{-1}(z) & \hat H_{22}(z) - \hat H_{21}(z)\hat H_{11}^{-1}(z)\hat H_{12}(z)
		\end{array}\right].
	\end{equation}
	Thus we can partition matrix $Q$ as $\textnormal{diag}(Q_1, Q_2)$, where $Q_1$ corresponds to the oracles
	in $\kappa$ and $Q_2$ corresponds to the rest part of oracles. Consequently, $Q^{-1}$ can be written
	as $\textnormal{diag}({Q_1}^{-1}, {Q_2}^{-1})$.
	
	From \cref{eq56} and \cref{eq57}, we have
	
	\begin{displaymath}
		\begin{aligned}
			\hat H(z) & \xrightarrow{\mathcal{C}_{\kappa}} \left[\begin{array}{c c}
				\hat H_{11}^{-1}(z) & -\hat H_{11}^{-1}(z)\hat H_{12}(z) \\
				\hat H_{21}(z)\hat H_{11}^{-1}(z) & \hat H_{22}(z) - \hat H_{21}(z)\hat H_{11}^{-1}(z)\hat H_{12}(z)
			\end{array}\right] \\
			& \xrightarrow{P_{\pi}} \left[\begin{array}{c c}
				Q_1\hat H_{11}^{-1}(z)Q_1^{-1} & -Q_1\hat H_{11}^{-1}(z)\hat H_{12}(z)Q_2^{-1} \\
				Q_2\hat H_{21}(z)\hat H_{11}^{-1}(z)Q_1^{-1} & Q_2\hat H_{22}(z)Q_2^{-1} - Q_2\hat H_{21}(z)\hat H_{11}^{-1}(z)\hat H_{12}(z)Q_2^{-1}
			\end{array}\right], \\
			\hat H(z) & \xrightarrow{P_{\pi}} \left[\begin{array}{c c}
				Q_1\hat H_{11}(z)Q_1^{-1} & Q_1\hat H_{12}(z)Q_2^{-1} \\
				Q_2\hat H_{21}(z)Q_1^{-1} & Q_2\hat H_{22}(z)Q_2^{-1}
			\end{array}\right] \\
			& \xrightarrow{\mathcal{C}_{\kappa}} \left[\begin{array}{c c}
				Q_1\hat H_{11}^{-1}(z)Q_1^{-1} & -Q_1\hat H_{11}^{-1}(z)\hat H_{12}(z)Q_2^{-1} \\
				Q_2\hat H_{21}(z)\hat H_{11}^{-1}(z)Q_1^{-1} & Q_2\hat H_{22}(z)Q_2^{-1} - Q_2\hat H_{21}(z)\hat H_{11}^{-1}(z)\hat H_{12}(z)Q_2^{-1}
			\end{array}\right].
		\end{aligned}
	\end{displaymath}
	
	We get the desired results to show $\mathcal{C}_{\kappa}$ and $P_{\pi}$ commute. Therefore, conjugation
	and cyclic permutation commute.
\end{proof}

\section{Proof of (A.4) and (A.5)}\label{proof-oracles}
For each $i \in \{1, \dots, n\}$ we have
\begin{displaymath}
	z_i = \textnormal{argmin}_x \left\{ \lambda_i f_i(x) + \frac{1}{2}\left[\begin{array}{c} x \\ y_i \end{array} \right]^T
	\left[\begin{array}{cc}	Q_{11}^i & Q_{12}^i\\ Q_{21}^i & Q_{22}^i
	\end{array} \right] \left[\begin{array}{c}	x \\ y_i \end{array} \right]\right\}.
\end{displaymath}
Besides, $Q_{11}^i$ is invertible for any $i \in \{1, \dots, n\}$. Since $f_i$ is a convex function,
the argmin oracle can be written as $z_i \in -{Q_{11}^i}^{-1}\lambda \partial
f_i (z_i) - {Q_{11}^i}^{-1}Q_{12}^iy_i$ by treating $\partial f_i$ as the oracle.
Written into matrix form, we have
\begin{equation}\label{eq53}
	\bar{u}_1 = -{Q_1}^{-1}\lambda\tilde{u}_1 - {Q_1}^{-1}Q_2 \bar{y}_1,
\end{equation}
where $\tilde{u}_1 = [\partial f_1 (z_1), \dots, \partial f_n (z_n)]^T$.
Combine \cref{eq53} with the state-space realization \cref{eq54}, we get the desired results for \cref{eq34}.
The corresponding system equations show as
\small
\begin{displaymath}
	\begin{aligned}
		x^{k+1} & = (A - B_1(I + M_1D_{11})^{-1}M_1C_1)x^k -B_1(I + M_1D_{11})^{-1}M_2\tilde{u}^k_1 +
		(B_2- B_1(I + M_1D_{11})^{-1}M_1D_{12})\bar{u}^k_2 \\
		\bar{u}^k_1 & = -(I + M_1D_{11})^{-1}M_1C_1x^k -(I + M_1D_{11})^{-1}M_2\tilde{u}^k_1
		-(I + M_1D_{11})^{-1}M_1D_{12}\bar{u}^k_2 \\
		\bar{y}^k_2 & = (C_2-D_{21}(I + M_1D_{11})^{-1}M_1C_1)x^k -D_{21}(I + M_1D_{11})^{-1}M_2\tilde{u}^k_1
		+ (D_{22}-D_{21}(I + M_1D_{11})^{-1}M_1D_{12})\bar{u}^k_2.
	\end{aligned}
\end{displaymath}
\normalsize
To calculate the transfer function, note that
\small
\begin{displaymath}
	(zI - A + B_1(I+M_1D_{11})^{-1}M_1C_1)^{-1}
	= (zI -A)^{-1} - (zI -A)^{-1}B_1(I+M_1\hat H_{11}(z))^{-1}M_1C_1(zI-A)^{-1}.
\end{displaymath}
\normalsize
We have
\small
\begin{align*}
	\hat H'_{11}(z) & = (I + M_1D_{11})^{-1}M_1C_1(zI - A + B_1(I+M_1D_{11})^{-1}M_1C_1)^{-1}B_1(I + M_1D_{11})^{-1}M_2 -(I + M_1D_{11})^{-1}M_2\\
	& = (I + M_1D_{11})^{-1}(M_1\hat H_{11}(z) - M_1D_{11})(I - (I+M_1\hat H_{11}(z))^{-1}(M_1\hat H_{11}(z) - M_1D_{11}))(I + M_1D_{11})^{-1}M_2\\
	& -(I + M_1D_{11})^{-1}M_2 \\
	& = (I + M_1D_{11})^{-1}(M_1\hat H_{11}(z) - M_1D_{11})(I+M_1\hat H_{11}(z))^{-1}M_2-(I + M_1D_{11})^{-1}M_2 \\
	& = -(I + M_1\hat H_{11}(z))^{-1}M_2\\
	\hat H'_{12}(z) & = -(I + M_1D_{11})^{-1}M_1C_1(zI - A + B_1(I+M_1D_{11})^{-1}M_1C_1)^{-1}B_2 -(I + M_1\hat H_{11}(z))^{-1}M_1D_{12} \\
	& =-(I + M_1D_{11})^{-1}(I - (M_1\hat H_{11}(z) - M_1D_{11})(I+M_1\hat H_{11}(z))^{-1})(M_1\hat H_{12}(z) \\
	& - M_1D_{12}) -(I + M_1\hat H_{11}(z))^{-1}M_1D_{12} \\
	& = -(I + M_1\hat H_{11}(z))^{-1}(M_1\hat H_{12}(z) - M_1D_{12})-(I + M_1\hat H_{11}(z))^{-1}M_1D_{12} \\
	& = -(I + M_1 \hat H_{11}(z))^{-1}M_1\hat H_{12}(z)\\
	\hat H'_{21}(z) & = -C_2(zI - A + B_1(I+M_1D_{11})^{-1}M_1C_1)^{-1}B_1(I + M_1D_{11})^{-1}M_2 - D_{21}(I + M_1\hat H_{11}(z))^{-1}M_2 \\
	& = -(\hat H_{21}(z) - D_{21})(I - (I+M_1\hat H_{11}(z))^{-1}(M_1\hat H_{11}(z) - M_1D_{11}))(I + M_1D_{11})^{-1}M_2 \\
	& - D_{21}(I + M_1\hat H_{11}(z))^{-1}M_2 \\
	& = -(\hat H_{21}(z) - D_{21})(I + M_1\hat H_{11}(z))^{-1}M_2 - D_{21}(I + M_1\hat H_{11}(z))^{-1}M_2 \\
	& = -\hat H_{21}(z)(I +M_1\hat H_{11}(z))^{-1}M_2\\
	\hat H'_{22}(z) & = \hat H_{22}(z) - (\hat H_{21}(z)- D_{21})(I+M_1\hat H_{11}(z))^{-1}M_1(\hat H_{12}(z)- D_{12}) - D_{21}(I+M_1\hat H_{11}(z))^{-1}M_1(\hat H_{12}(z)- D_{12}) \\
	&- (\hat H_{21}(z)- D_{21})(I+M_1\hat H_{11}(z))^{-1}M_1D_{12}- D_{21}(I+M_1\hat H_{11}(z))^{-1}M_1D_{12} \\
	& = \hat H_{22}(z)-\hat H_{21}(z)(I +M_1\hat H_{11}(z))^{-1}M_1\hat H_{12}(z).
\end{align*}
\normalsize
Thus, we get the desired results as \cref{eq35}.

\end{document}